\documentclass[a4paper,11pt]{amsart}
\usepackage[latin1]{inputenc}
\usepackage{xcolor}

\usepackage{multicol}
\usepackage{amsfonts, amssymb, amsmath, amsthm, amscd,epsfig,mathrsfs, euscript}
\usepackage{latexsym}
\usepackage{graphicx}

\usepackage{cancel}
\usepackage[normalem]{ulem}

\usepackage[all]{xy}

\usepackage{enumerate}

\usepackage{verbatim}

\usepackage{hyperref}

\newtheorem{theorem}{Theorem}[section]

\newtheorem{corollary}[theorem]{Corollary}

\newtheorem{lemma}[theorem]{Lemma}

\newtheorem{proposition}[theorem]{Proposition}

	{\par\noindent{\bf Proposition \ref{res:hiper}.}\!\!
	\nopagebreak\normalsize\it}%

	{\par\noindent{\bf Theorem \ref{result43}.}\!\!
	\nopagebreak\normalsize\it}%

	{\par\noindent{\it Sketch of the proof}.  
	\nopagebreak\normalsize}%
	{\hfill\linebreak[2]\hspace*{\fill}$\circlearrowleft$}

	{\par\noindent{\it Proof of Proposition }\ref{prop:stab:smc}.  
	\nopagebreak\normalsize}%
	{\hfill\linebreak[2]\hspace*{\fill}$\circlearrowleft$}

	{\par\noindent{\it Proof of Propositions }\ref{adap:mon}{\it and }\ref{simult:adap}.\!\!\!
	\nopagebreak\normalsize}%
	{\hfill\linebreak[2]\hspace*{\fill}$\circlearrowleft$}

{\theoremstyle{definition}
       \newtheorem{definition}[theorem]{Definition}

       \newtheorem{remark}[theorem]{Remark}
       \newtheorem{example}[theorem]{Example}
    
       \newtheorem{parrafo}[theorem]{{\!}}}

\numberwithin{equation}{theorem}

\DeclareMathOperator{\caract}{char}

\DeclareMathOperator{\Dir}{Dir}

\DeclareMathOperator{\ord}{ord}

\DeclareMathOperator{\Rid}{Rid}
\DeclareMathOperator{\Sing}{Sing}
\DeclareMathOperator{\Spec}{Spec}

\DeclareMathOperator{\Hord}{H-ord}
\DeclareMathOperator{\Slaux}{-sl}
\DeclareMathOperator{\In}{In}
\DeclareMathOperator{\Gr}{Gr}

\DeclareMathOperator{\red}{red}

\newcommand{\G}{{\mathcal G}}

\renewcommand{\H}{{\mathcal H}}

\newcommand{\SSl}{\mathcal{S}\!\Slaux}
\DeclareMathOperator{\nublin}{\overline{\nu}-lin}
\DeclareMathOperator{\nubgen}{\overline{\nu}-gen}

\DeclareMathAlphabet{\mathpzc}{OT1}{pzc}{m}{it}
\newcommand{\nub}{\overline{\nu}}

\newcommand{\Mm}{\mathrm{\underline{Max}\; mult}_X}

\newcommand{\Gd}{\G^{(d)}}

\newcommand{\Diff}{{\mathrm{Diff}}}

\title{The asymptotic Samuel functions and applications to resolution of singularities techniques}
\author{A. Benito, A. Bravo, S. Encinas}
\thanks{The authors were partially supported by 
PID2022-138916NB-I00 funded by MCIN/AEI/ 10.13039/501100011033 and by ERDF/EU;
The second author was partially supported through the ``Severo Ochoa'' Programme for Centres of Excellence in R\&D,
Grant CEX2023-001347-S funded by MICIU/AEI/10.13039/501100011033}


\AtEndDocument{\bigskip{\footnotesize
\textsc{Depto. Did\'acticas Espec\'{\i}ficas,
Facultad de Educaci\'on, Universidad Aut\'onoma de Madrid, Cantoblanco 28049 Madrid, Spain} \newline
\textit{E-mail address}, A. Benito: \texttt{angelica.benito@uam.es}
\medskip

\noindent\textsc{Depto. Matem\'aticas, Facultad de Ciencias, Universidad Aut\'onoma de Madrid and Instituto de Ciencias Matem\'aticas CSIC-UAM-UC3M-UCM, Cantoblanco 28049 Madrid, Spain} \newline
\textit{E-mail address}, A. Bravo: \texttt{ana.bravo@uam.es}
\medskip

\noindent\textsc{Depto. \'Algebra, An\'alisis Matem\'atico, Geometr\'{\i}a y Topolog\'{\i}a, and IMUVA, Instituto de Matem\'aticas.	Universidad de Valladolid} \newline
\textit{E-mail address}, S. Encinas: \texttt{santiago.encinas@uva.es}
}}

\begin{document}
	
\begin{abstract}
We use the asymptotic Samuel function to define the \emph{Samuel slope} of a Noetherian local ring,
and we prove that it characterizes regularity in the case of local excellent rings.
In addition, we introduce a second invariant that refines the Samuel slope using generic linear cuts, \emph{the refined Samuel slope}.
 We show that both, the Samuel slope and the refined Samuel slope,
are connected to invariants coming from resolution of
singularities.
\end{abstract}
	
	\maketitle
	

\section{Introduction}

Most of the local invariants of singularities used in different approaches to resolution of singularities are constructed or defined using extrinsic methods, such as the choice of local embeddings, finite projections or/and étale neighborhoods, or local coordinates. As a consequence, deep theorems are needed to show that the definitions are independent on the possible many choices.
Our purpose here is to show that some of these invariants can be defined directly,  in the local ring of the singular point. This will be achieved after a careful analysis of the local cone of the singularity  and with the help
of the asymptotic Samuel function, which can be defined for any local ring.
\medskip

To fix ideas, given a singular local ring, the multiplicity, or the Hilbert Samuel function, are usually a   first choice
to measure the complexity of the singularity. Since they are upper semi cointinuous functions, both choices induce a stratification of the singular locus.
After considering either of these  functions, 
there are different strategies that permit to distinguish points  within the same stratum.
One possibility, following Villamayor's approach, is to use Hironaka's order function, $\ord_X^{(d)}$, or the refined version $\Hord_X^{(d)}$ in positive characteristic, see \cite{Villa2007}, \cite{Br_V} \cite{BVIndiana}, \cite{BVComp}, \cite{Benito_V}).
Another one, as in the school of Cossart, Jannsen, Piltant, Schober, is to consider the \emph{$\delta$ invariant} of a singularity obtained from the construction of  Hironaka's polyhedron,  which in fact, can be defined when $R$ is an arbitrary regular local ring,
see \cite{CossartJannsenSchober2019}, \cite{CossartPiltant2015},
\cite{CossartPiltant2019}, and \cite{CossartSchober2021}.
\medskip

In \cite{Complutense2023}, we introduced the notion of {\em Samuel slope} of a local Noetherian ring, and
we studied its connection with $\Hord_X^{(d)}$ in the context of  algebraic varieties defined over  perfect fields of positive characteristic.
In this paper we show that  the previous invariants 
$\ord_X^{(d)}$, $\Hord_X^{(d)}$, $\delta$ 
are, somehow connected,  and,  moreover, under suitable conditions,  unified  by the Samuel slope
and the \emph{refined Samuel slope} of a local ring.
\medskip

Moreover, in \cite{PrepJavier} we explored some properties of the Samuel slope in the case of local excellent  equidimensional rings of equal characteristic.
Here we extend these results to arbitrary excellent local rings, showing in particular that the
non-finiteness of the Samuel slope characterizes regular local rings.
\medskip

We give more details and precise statements in the following paragraphs.

\

\noindent{\bf The Samuel slope of a Noetherian local ring}

\medskip

Let $(A,{\mathfrak m},k)$ be a local Noetherian  ring. For an element $a\in A$,  we can consider the ordinary order function, $\nu_{\mathfrak m} (a)=\sup\{\ell\in\mathbb{N}\cup\{\infty\}\mid a\in {\mathfrak m}^{\ell}\}$, which, with no further assumptions on $A$, might had a wild behavior. The {\em asymptotic Samuel function}, $\nub_{\mathfrak{m}}$,  introduced by Samuel in \cite{Samuel} and  studied afterwards  by D. Rees in a series of papers (\cite{Rees1955}, \cite{Rees_Local},  \cite{Rees_Ideals}, \cite{Rees_Ideals_II}), is a normalized version of the usual order with nicer  properties:  
\begin{equation} \label{IntroSamuelF}
	\bar{\nu}_{\mathfrak m}(a)=\lim_{n\to\infty}\frac{\nu_{\mathfrak m}(a^n)}{n}, \qquad a\in A.
\end{equation}
It is worthwhile noticing that the definition is valid for any proper ideal $I\subset A$ and that $\bar{\nu}_{I}(a)\in {\mathbb Q}\cup \{\infty\}$. 
When  $(A,{\mathfrak m}, k)$  is regular, $\nub_{\mathfrak m} (a)=\nu_{\mathfrak m} (a)$ for all $a\in A$.   We refer to \cite{Hu_Sw} for a thorough expositions on this topic, and to section \ref{Samuel_a} in this paper  where more details are given. 
 
\medskip

We use the asymptotic Samuel function to define the  {\em Samuel slope of the local ring} $A$, $\SSl(A)$, as follows. To start with, if $A$ is  regular then  we set $\SSl(A):=\infty$. To treat the  case of a non-regular ring, we  consider first:  
$${\mathfrak m}^{\geq 1} :=\{a\in A: \nub_{\mathfrak m}(a)\geq 1\}, \qquad \text{ and } \qquad {\mathfrak m}^{>1} :=\{a\in A: \nub_{\mathfrak m}(a)> 1\},$$
and note that there is a natural $k$-linear map,  
\begin{equation}
	\label{graduado_parte_uno}
	\lambda_{\mathfrak{m}}: {\mathfrak m}/ {\mathfrak m}^2 \to {\mathfrak m}^{\geq 1}/{\mathfrak m}^{> 1}. 
\end{equation}
When $(A, \mathfrak{m},k)$ is non-regular,   $\ker(\lambda_{\mathfrak{m}})$ might be non trivial, and its 
  dimension   as a $k$-vector space is an invariant of the local ring. It can be checked that 
$$0\leq \dim_{k}(\ker(\lambda_{\mathfrak{m}})) \leq \text{exc-emb-dim}(A,\mathfrak{m}):= \dim_{k}\left({\mathfrak m}/ {\mathfrak m}^2\right)-\dim(A),$$
where $\dim(A)$ denotes the Krull dimension of the ring $A$.

\medskip

When $\dim_{k}(\ker(\lambda_{\mathfrak{m}})) < \text{exc-emb-dim}(A,\mathfrak{m})$ we  say that $A$  is in the in the {\em non-extremal case}, and then we set   $\SSl(A):=1$. Otherwise, the local ring $A$ is in the {\em extremal case}, and its Samuel slope is defined as follows. Let  $t:=\text{exc-emb-dim}(A,\mathfrak{m})>0$. 
We say that  a collection of $t$-elements $\gamma_1,\ldots, \gamma_t\in {\mathfrak m}$ form a  {\em $\lambda_{\mathfrak{m}}$-sequence} if their classes,
$\In_{\mathfrak m}\gamma_1,\ldots,\In_{\mathfrak m}\gamma_t\in {\mathfrak m}/ {\mathfrak m}^2$, form a basis of $\ker(\lambda_{\mathfrak{m}})$.
 Then,   the {\em Samuel slope of $(A,\mathfrak{m})$, $\SSl(A)$}, is defined as  
\begin{equation*}
	\SSl(A):=
	\sup\limits_{{\lambda_{\mathfrak{m}}}{\text{-sequence}}} \left\{
	\min\left\{\bar{\nu}_{\mathfrak{m}}(\gamma_1),\ldots,\bar{\nu}_{\mathfrak{m}}(\gamma_t)\right\} 
	\right\},
\end{equation*}
where the supremum is taken over all the $\lambda_{\mathfrak{m}}$-sequences of $(A,\mathfrak{m},k)$.
To have an intuition of what this invariant means, it can be checked that  when $(A,\mathfrak{m})$ is in the extremal case, the maximal ideal $\mathfrak{m}$ has a reduction generated by $d=\dim(A)$ elements, $\kappa_1,\ldots, \kappa_{d}$.
In addition we have that 
$$\langle \In_{\mathfrak m}\kappa_1,\ldots, \In_{\mathfrak m}\kappa_{d} \rangle \oplus \ker(\lambda_{\mathfrak{m}}) =
 {\mathfrak m}/ {\mathfrak m}^2.$$
In this situation, the Samuel slope of $(A,\mathfrak{m},k)$ measures how deep the ideals generated by $\lambda_{\mathfrak{m}}$-sequences
lie into the integral closure of powers of ${\mathfrak m}$.

\begin{example}
Let $k$ a field of characteristic $2$ and let  $A=\left(k[x,y]/\langle x^2+y^4+ y^5\rangle\right)_{\langle x,y\rangle}$.
For an element $r\in k[x,y]$ we will use $r'$ to denote its natural image  in $A$.
Let ${\mathfrak m}=\langle x', y'\rangle$. Observe $\ker(\lambda_{{\mathfrak m}})=\langle \In_{\mathfrak m}{x}'\rangle\subset  {\mathfrak m}/{{\mathfrak m}}^2$, hence $A$ is in the extremal case. Note that $\{ x'\}$ is a $\lambda_{{\mathfrak m}}$-sequence and so is   $\{x'+{y'}^2\}$. Here $\nub_{{\mathfrak m}}(x')=1$ while $\nub_{{\mathfrak m}}(x'+{y'}^2)=\frac{5}{2}$. To conclude,    it can be checked that  ${\mathcal S}\text{-sl}(A)=\frac{5}{2}$. 
\end{example}
 
\

\noindent{\bf Results  }

\medskip

From the definition it is not clear that the Samuel slope of a (non-regular)  local ring be finite. In \cite{PrepJavier}  we focused in the case of excellent equidimensional local rings containing a field, and we showed that the non-finiteness of this ring invariant   characterizes local regular  rings. Here we extend this statement to local excellent rings with no extra assumptions. More precisely we show: 

\medskip

\noindent{\bf Theorem \ref{pendiente_finita}} {\em Let $(R',\mathfrak{m}')$ be  a reduced  excellent  non-regular  local ring.
	Then $\SSl(R')\in\mathbb{Q}$.
}

\medskip

\medskip

From where it follows:

\medskip

\noindent{\bf Corollary \ref{corolario_pendiente_finita}}
{\em Let $(R',\mathfrak{m}')$ be an excellent ring.
Then $\SSl(R')=\infty$ if and only if $R'_{\red}$ is a regular local ring.
}

\medskip
Next,  we have explored further connections between the Samuel slope of a local singular ring and invariants used in  resolution of singularities. To give some intuition on the role played by this function,  let us consider the following (simple) situation. 
\medskip

 Let $k$ be a perfect field,   let $S$ be a smooth $k$-algebra of dimension $d$, and define  $R=S[z]$ as the polynomial ring in one variable with coefficients in $S$. Suppose $X$ is a hypersurface in $\Spec(R)$ of maximum multiplicity $m>1$ determined by a polynomial  of the form
 \begin{equation}
 	\label{local_etale}
 	f(z)=z^m+a_1z^{m-1}+\ldots+a_m \in   S[z].
 \end{equation}
Note that, after considering a convenient \'etale neighborhood, a hypersurface can be assumed to be described as the zero set of a polynomial as  in (\ref{local_etale}).
Certainly, the multiplicity is a measure of the singularities of $X$. But singularities with the same multiplicity can have different complexity, and this information  is codified in the coefficients of the expression above.

\medskip  

Now, a series of techniques have been developed to collect information from the coefficients $a_1,\ldots, a_m\in S$ that be independent on the choice of coordinates, or on the selection of the \'etale neighborhood, and yet, capture the complexness of the singularity.
One example, as mentioned before, is Hironaka's order function, $\ord_X^{(d)}$,  or the refined version $\Hord_X^{(d)}$ in positive characteristic (see section \ref{Rees_Algebras}  for further details and definitions).
Another example is the $\delta$ invariant of a singularity obtained from the construction of  Hironaka's polyhedron (see section \ref{poliedro_Hironaka} for references and more information). 
Here we show that there is a connection among the previous invariants and the Samuel slope.
To start with, the functions $\ord_X^{(d)}$  and  $\Hord_X^{(d)}$ can be defined for arbitrary equidimensional algebraic varieties over a perfect field. In this setting, we obtain, as a corollary of Theorem \ref{todas_igual}, 

\medskip

\noindent {\bf Corollary \ref{corolario_H_ord}} 
	{\em Let $X$ be an equidimensional algebraic variety of dimension $d$  defined over a perfect field $k$. Let $\zeta\in X$ be a  singular point of multiplicity $m>1$.  
	Then
	$$\SSl(\mathcal{O}_{X,\zeta})=\Hord_X^{(d)}(\zeta)\leq \ord_{X}^{(d)}(\zeta).$$
	Moreover if $\caract(k)=0$  then,
	$\Hord_X^{(d)}(\zeta)=\SSl(\mathcal{O}_{X,\zeta})=\ord_{X}^{(d)}(\zeta)$.}

\medskip

When the characteristic is zero, {\em hypersurfaces of maximal contact} play a 
key role in resolution of singularities, since they permit an inductive argument on the dimension \cite{GiraudMaximal}, \cite{CossartBulletin}, see also \cite{Tirol}. 
 As it turns out, elements defining  hypersurfaces of maximal contact appear naturally when computing the Samuel slope of a local ring.
More precisely, suppose that $S$ is a regular local ring, let  $X=\Spec(S[z]/\langle f\rangle)$ with $f\in S[z]$, be the local ring of a hypersurface of multiplicity $m>1$. Consider a basis of $\Diff_{S[z]/S}$,  the relative differential operators of order $m-1$,   $\{\Delta_z^i\}_{i=0,\ldots, m-1}$, and let $\omega \in \Delta_z^{(m-1)}(f):=\langle \Delta_z^i(f): i=0,\ldots, m-1\rangle$ be an element or order 1 in $S[z]$ (thus $\mathbb V(\langle w\rangle)$ defines a hypersurface of maximal contact for the top multiplicity locus of $X$). Then, setting $R'= S[z]/\langle f\rangle$, and denoting by $w'$ the image of $w$ in $R'$, we have that:    
\medskip

\noindent{\bf Corollary \ref{orden_maximal}} 
$\SSl(R')=\nub_{\mathfrak{m}'}({\omega}')$.

\medskip

On the other hand, regarding Hironaka's polyhedron, new techniques have been introduced to define the $\delta$ invariant in very general contexts in which the presence of a field is not required (see, among others, \cite{HirPoly}, \cite{CossartBulletin}, \cite{CossartJannsenSaito},
\cite{CossartJannsenSchober2019}, \cite{CossartPiltant2015}, \cite{CossartSchober2021}). In this context the usual approach is to consider a regular local ring $R$ and to  assume that the hypersurface singularity is given by some non-zero element $f\in R$.
Then, a regular system of parameters is selected in $R$, $(u,y)=(u_1,\ldots,u_s,y_1,\ldots,y_r)$, where $(y)$ is connected to the notion of {\em directrix}, an invariant of the local cone of the singularity (see section \ref{poliedro_Hironaka} and Definition \ref{Directriz}).
Next, $f$ can be written as a finite sum, 
\begin{equation}
                \label{expresion}
f=\sum c_{\alpha,\beta} u^\alpha y^\beta,
\end{equation}
where $c_{\alpha,\beta}\in R$ are units and $\alpha, \beta$ are multi-indexes.
This is what we refer to as the {\em Cossart-Piltant expansion of $f$ (CP-expansion)}, see section \ref{Samuel_a} for details. Considering expressions like (\ref{expresion}) for different selections of $(y)$, leads to the construction of a polyhedron,  $\Delta(f,u)$.
The $\delta$ invariant, $\delta(\Delta(f,u))$, derives from here.
It can be shown that  $\delta(\Delta(f,u))$ does not depend on the choice of the parameters in $R$ and once more it is designed to  capture  the complexity of  the singularity at the point (see section \ref{poliedro_Hironaka} for precise statements and definitions).
 
\medskip
 
We should  emphasize here that Hironaka's polyhedron can be defined in a more general setting always linked to the maximum stratum of the Hilbert-Samuel function of an excellent scheme. However, the Samuel slope is a function naturally related to the multiplicity. Thus, in this context,  we stick to the case of  a  hypersurface, and distinguish two cases, depending on the codimension of the  directrix. If the directrix has codimension one, which corresponds to the extremal case,  as a corollary  to Theorem \ref{Casor1}, we get:

\medskip

\noindent{\bf Corollary \ref{delta_dir_dim_one}} {\em Suppose  $R'=R/(f)$ is a hypersurface in the extremal case, and let $y_1\in R$ be a regular parameter such that  $\In_{\mathfrak{m}}{f}=\In_{\mathfrak{m}}{y}_1^m$. Then, for any collection of regular parameters   $u_1,\ldots, u_{n-1}\in R$ such that  $(u_1,\ldots,u_{n-1},y_1)$ is a regular system of parameters in $R$ we have that:
	$$\SSl(R')=\delta(\Delta(f,u)).$$
Moreover, if $R$ is excellent, then there exists a regular parameter $\tilde{y}_1\in R$ 
	such that  $\In_{\mathfrak{m}}{f}=\In_{\mathfrak{m}}\tilde{y}_1^m$ and
	$$\SSl(R')=\nub_{\mathfrak{m}'}(\tilde{y}'_1)=
	\delta(\Delta(f,u,\tilde{y}_1))=\delta(\Delta(f,u)).$$}

\medskip

Next, we approach the case of a singular point in a hypersurface for which the directrix has codimension larger than one, which corresponds to the non-extremal case. In this context we introduce the notion  of {\em generic linear cut} (Definition \ref{def_generic_linear_cut}), that leads us to define  the {\em refined Samuel slope of a local ring $R'$}, ${\mathcal R}\SSl(R')$, see Definition \ref{def_refined_Samuel_slope}. Thus, the refined Samuel slope of a local ring is an invariant that  refines    the Samuel slope when the later equals  1.
\medskip 

To be more specific,  given a collection of parameters $y'_1,\ldots,y'_r\in R'$ determining 
the directrix, the notion of generic linear cut allows us to assign a generic value
$\nubgen(y')$ to the previous collection, see Definition \ref{Defnugen}.
Then the refined Samuel slope is the supremum of such $\nubgen(y')$ when $y'$ runs over all possible choices determining the directrix.
\medskip

It is worthwhile mentining that to make sense of the notion of generic linear cuts we need that the residue field of $R'$ be infinite. Thus, when the residue field is finite,   we set  
	$${\mathcal R}\SSl(R'):= {\mathcal R}\SSl(R'_1), $$
where $(R_1', {\mathfrak m}_1, k_1):=(R'[t], {\mathfrak m}'[t], k(t))$ (see Lemma \ref{Ssl_InfinitoK}, where we show that the two notions agree when the residue field of $R'$ is non-finite). Among other results, we prove:

\medskip

\noindent{\bf Theoren \ref{RSSl_completion}}
{\em 
Let $R$ be a local regular ring, let $R'=R/\langle f\rangle$ and let $\widehat{R'}$   be the ${\mathfrak m}'$-adic completion of $R'$. Then
\begin{enumerate}
\item[(i)] $\mathcal{R}\SSl(R')=\mathcal{R}\SSl(\widehat{R'})$.
\item[(ii)] Let $u_1,\ldots,u_{n-r}\in R$ be a collection of elements so    that there are $z_1,\ldots,z_r\in R$ determining the directrix of $f$ and   $(u,z)$ is a regular system of parameters of $R$.
Then
$$\mathcal{R}\SSl(R')=\mathcal{R}\SSl(\widehat{R'})=\delta(\Delta(f,u)).$$
\item[(iii)] Moreover, if $\delta(\Delta(f,u))<\infty$
 and if the residue field of $R'$ is infinite,
 there are elements  $y'_1,\ldots,y'_r\in R$ determining the directrix of $f$ such that
$$\nubgen(y')=\mathcal{R}\SSl(R')=\delta(\Delta(f,u)).$$
\end{enumerate} 
}

\medskip

In addition to relating the Samuel Slope and the refined Samuel slope to the $\delta$ invariant of a singularity, our results reprove that the later does not depend of choices of coordinates nor local embeddings, and hence it is indeed an invariant. 

\medskip 

To approach our results we have made use of some special properties of Cossart-Piltant expansions, and a formula of Hickel for the computation of the asymptotic Samuel function for local equicharacteristic Noetherian rings.

\medskip

On the one hand, CP-expansions were introduced in \cite{CossartPiltant2019}. As indicated before, this is a technique to express an element in a local regular ring   in terms of a finite number of monomials on a regular system of parameters. This makes it  possible  to work with finite expressions in a wide context,  where   base fields might not  be present,  and the necessity to pass to the completion is  avoided. To facilitate the reading of the paper we give an overview on these expansions in the last part of section \ref{Samuel_a} (starting in \S \ref{definition_Cossart_Piltant})   and   also in section \ref{CossartPiltantHickel}. 

\medskip 

On the other, in \cite{Hickel}, M. Hickel  presented a series of results regarding the asymptotic Samuel function and derived a explicit  formula for the computation in the setting of local rings containing a field. The previous hypothesis was needed to present the completion of the local ring as a convenient finite extension of a regular local ring (see section \ref{CossartPiltantHickel} for precise details).  Here we use CP-expansions to adapt  Hickel's approach  in order to obtain  a formula  replacing  finite projections by embeddings in regular local  rings  that might not  contain a field: 

\medskip

\noindent{\bf Theorem  \ref{GeneralHickel}} 
{\em Let $(R,\mathfrak{m})$ be a regular local ring of dimension $n$, and let  $u_1,\ldots,u_n$ be any regular system of parameters in $R$.
Let $f\in R\setminus \{0\}$ be a non-unit,  set $R'=R/\langle f\rangle$ and let ${\mathfrak m}'={\mathfrak m}/\langle f\rangle$.  Then: 
\begin{itemize}
	\item[(i)] After relabeling the elements $u_1,\ldots,u_n$, we may assume that
	$$\nub_{{\mathfrak m}'}(u'_i)=1, \ \ \ \text{ for } i=1,\ldots, n-1,$$
	where $u'_i$ denotes  the class of $u_i$ in $R'$.
\end{itemize}
Set $y=u_{n}$ and let $y'$ denote the class of $y=u_{n}$  in $R'$. 
\begin{itemize} 
	\item[(ii)] If  $\nub_{{\mathfrak m}'}(y')\in {\mathbb Q}_{>1}\cup \{\infty\}$ then: 
	\begin{enumerate}
		\item $f\notin\langle u_1,\ldots, u_{n-1}\rangle$ and $y'\in \overline{ \langle u_1',\ldots, u'_{n-1}\rangle} \subset R'$, where $\overline{ \langle u_1',\ldots, u'_{n-1}\rangle}$ is the integral closure of $\langle u_1',\ldots, u'_{n-1}\rangle$;
		\item up to a unit,  $f$ can be written as a pseudo-Weierstrass element of degree $\ell=e(R')>1$ w.r.t. $(u_1,\ldots, u_{n-1};y=u_{n})$ (see Definition \ref{DefWeierstrass}), 
		\begin{equation} 
			f=y^{\ell}+\sum_{i=1}^{\ell}a_i y^{\ell -i}, 
		\end{equation}
		and 
		\begin{equation} 
			\nub_{{\mathfrak m}'}(y')=\min_{i}\left\{\frac{\ord_{\mathfrak m}(a_i)}{i}: i=1,\ldots, \ell\right\}.
		\end{equation}
	\end{enumerate}
\end{itemize}
\begin{itemize}
	\item[(iii)] If $\nub_{{\mathfrak m}'}(y')=1$ and if $f\notin\langle u_1,\ldots,u_{n-1}\rangle$, then up to a unit $f$ can be written as a pseudo-Weierstrass element w.r.t. $(u_1,\ldots, u_{n-1};y=u_{n})$ as in (\ref{expresion_f}) and formula (\ref{formula_orden}) holds if and only if $\langle u_1',\ldots, u'_{n-1}\rangle$ is a reduction of ${\mathfrak m'}$.  
\end{itemize}}

As another application  we obtain a criterion to compute the asymptotic Samuel order  in a quotient of a regular local ring, $R'=R/{\mathfrak a}$, for elements whose CP-expansion in $R$ have a suitable form  (see Proposition \ref{orden_reduction}). 
 
\

\noindent{\bf Organization of the paper}
\medskip

In section \ref{Samuel_a}  we review some known properties of both, the asymptotic Samuel
function and  the Samuel slope of a local ring.
In addition, some facts on Cossart-Piltant expansions are exposed.
In section \ref{SecFinita} we extend some   results from \cite{PrepJavier}   
to the case of excellent rings that do not necessarily  contain  a field.
Section \ref{CossartPiltantHickel} is devoted to proving  a version of   Hickel's formula.
In section \ref{Rees_Algebras} we sketch a  short introduction to the functions
$\Hord^{(d)}$ and $\ord^{(d)}$ appearing in resolution of singularities, and we establish  their connection 
  to  the Samuel slope  in
Section \ref{SecSSl_Hord}.
We recall the construction  and some properties of Hironaka's polyhedron in
section \ref{poliedro_Hironaka}.
To conclude, in section \ref{Sec_Refined}   the notion of  refined Samuel slope  is introduced and some connections to Hironaka's polyhedron are presented.
\noindent
 
\

{\noindent \emph{Acknowledgments:}  We  profited from several conversations with O.~E. Villamayor U. We also thank V.~Cossart and S.~D. Cutkosky for useful suggestions.
}

\section{Preliminaries}\label{Samuel_a}
The purpose of this section is twofold. On the one hand,   we will review some facts about the asymptotic Samuel function and the Samuel slope of a local Noetherian ring. On the other,  we conclude the section with an overview of the Cossart-Piltant expansions, that will be used in several parts of the paper.

\begin{parrafo}
\label{properties_Samuel}
{\bf Some facts about the asymptotic Samuel function on Noetherian rings.} As indicated in the Introduction, the asymptotic Samuel function can be defined for an arbitrary ideal $I\subsetneq A$, and for $f\in A$,  $\nub_I(f)$  measures how deep the element $f$ lies in the integral closure of powers of $I$. In fact, the following results hold: 
\end{parrafo}
	
	\begin{proposition}
	\cite[Corollary 6.9.1]{Hu_Sw}
	Suppose $A$ is Noetherian. Then for a proper ideal    $I\subset A$  and every $a\in\mathbb{N}$, \begin{equation*}
	\overline{I^a}=\{f\in R \mid \bar{\nu}_I(f)\geq a\}.
	\end{equation*}
\end{proposition}

		\begin{corollary} \label{nub_integral}
		Let $A$ be a Noetherian ring and $I\subset A$ a proper ideal. If $f\in A$ then
		\begin{equation*}
		\bar{\nu}_I(f)\geq \frac{a}{b} \Longleftrightarrow f^b\in\overline{I^a}.
		\end{equation*}
	\end{corollary}
	
The previous characterization of $\nub_I$  leads to the following result that gives a valuative version of the function. 

\begin{theorem}\label{nu_barra_valuations} 
Let $A$ be a Noetherian ring, and let $I\subset A$ be a proper ideal not contained in a minimal prime of $A$.
Let $v_1,\ldots,v_s$ be a set of Rees valuations of the ideal $I$.
If $f\in A$ then
\begin{equation*}
	\bar{\nu}_I(f)=\min\left\{\frac{v_i(f)}{v_i(I)}\mid i=1,\ldots,s\right\}.
\end{equation*}
\end{theorem}
	
\begin{proof} See \cite[Lemma 10.1.5, Theorem 10.2.2]{Hu_Sw} and \cite[Proposition 2.2]{Irena}. 
\end{proof}

It can also be checked that the following properties hold: 	

\begin{proposition}
\label{PotenciaNuBar}
If $A$ is a Noetherian ring, and $I\subset A$ is a proper ideal not contained in a minimal prime of $A$, then
the function $\bar{\nu}_I$ satisfies the following properties
\begin{enumerate}
\item[(i)]  $\nub_I(f+g)\geq \min\{\nub_I(f),\nub_I(g)\}$, for all $f,g\in A$,
moreover $\nub_I(f+g)=\min\{\nub_I(f),\nub_I(g)\}$ if $\nub_I(f)\neq\nub_I(g)$;
\item[(ii)]  $\nub_I(f\cdot g)\geq \nub_I(f)+\nub_I(g)$, for all $f,g\in A$;
\item[(iii)] $\nub_I(0)=\infty$ and $\nub_I(1)=0$;
\item[(iv)] $\bar{\nu}_I(f^r)=r\bar{\nu}_I(f)$, for all $f\in A$ and $r\in\mathbb{N}$;
\item[(v)] there exists a positive integer $m$ such that 
$\nub_I(f)\in\frac{1}{m}\mathbb{N}\cup\{\infty\}$, for all $f\in A$;
\item[(vi)]   $\nub_I(f)=\infty$ if and only if $f$ is nilpotent in $A$;
\item[(vii)]  If 
$\{b_t\}_{t=1}^{\infty}\subset A$ is Cauchy sequence for the $I$-adic topology   sequence such that $\lim_{t\to\infty}b_t=b\in A$, and if 
  $\nub_{I}(b)<\infty$ (i.e., if $b$ is not nilpotent) then
$\nub_{I}(b_t)=\nub_{I}(b)$ for $t\gg 0$;
\item[(viii)] If 
$\{b_t\}_{t=1}^{\infty}\subset A$ is Cauchy sequence for the $I$-adic topology   sequence such that $\lim_{t\to\infty}b_t=b\in A$, and if 
$\nub_{I}(b)=\infty$   then
$\lim_{t\to \infty}\nub_{I}(b_t)=\infty$;
\item[(ix)] If $(A_1,\mathfrak{m_1})\subset (A_2,\mathfrak{m_2})$ is a faithfully flat extension of local rings, then 
$\nub_{\mathfrak{m}_1}(a)=\nub_{\mathfrak{m}_2}(a)$ for all $a\in A_1$. 
\end{enumerate}
\end{proposition}

\medskip

\begin{parrafo}
\label{GraduadoBarra}
To finish this part,  for $b\in {\mathbb Q}$ the following ideals can be defined, \begin{equation*}
	I^{\geq b}=\{g\in A\mid \bar{\nu}_I(g)\geq b\}, \qquad \text{ and } \qquad 
	I^{> b}=\{g\in A\mid \bar{\nu}_I(g)> b\},
\end{equation*}
 and    a graded ring can be considered in a natural way:
\begin{equation*}
	\overline{\Gr}_I(A)=\bigoplus_{b\in\mathbb{Q}_{\geq 0}}I^{\geq b}/I^{> b}.
\end{equation*}
The   ring $\overline{\Gr}_I(A)$ is reduced,  and there is a natural morphism of graded algebras, 
\begin{equation}
	\label{morfismo_graduados}
\begin{array}{rccl}
	\Lambda_I: & \Gr_I(A) & \longrightarrow &  \overline{\Gr}_I(A) \\
	&   f+I^{b+1} & \mapsto  & 	\Lambda_I(f+I^{b+1}):=f+I^{>b}, 
\end{array}	
\end{equation}
whose  kernel is the nilradical of  $\Gr_I(A)$. 
As a consequence, $\Lambda_I$ is injective if and only if $\Gr_I(A)$ is reduced.    
In \cite{PrepJavier} we studied  $	\overline{\Gr}_{\mathfrak m}(A)$ for a local excellent ring $(A,{\mathfrak m}, k)$, and showed that it is a $k$-algebra of finite type (cf. \cite[Theorem 3.4]{PrepJavier}).  
\end{parrafo}
\medskip

\begin{parrafo}
\label{properties_slope} 
{\bf The Samuel slope of a Noetherian local ring.} Let $(A,{\mathfrak m}, k)$ be a  Noetherian local ring. Restricting to the degree-one-part  the morphism $\Lambda_{\mathfrak m}$ from (\ref{morfismo_graduados}), we obtain the linear map $	\lambda_{\mathfrak{m}}: {\mathfrak m}/ {\mathfrak m}^2 \to {\mathfrak m}^{\geq 1}/{\mathfrak m}^{> 1}$ from (\ref{graduado_parte_uno}), which we used in the definition of the Samuel slope of $A$, $\SSl(A)$. From the comment in the previous paragraph,  it follows that non-trivial elements in $\ker(\lambda_{\mathfrak m})$ are degree-one nilpotents in $\text{Gr}_{\mathfrak m}(A)$.  

\medskip 

Part of the work developed in \cite{Complutense2023}, and  mostly in   \cite{PrepJavier}, consisted in studying some natural properties of the Samuel slope. Among others, it can be shown that if $\widehat{A}$ is the ${\mathfrak m}$-adic completion of $A$, then $\SSl(A)=\SSl(\widehat{A})$, and in the same line, if $A$ is equidimensional excellent and contains a field, and if  $A\to A'$ is an arbitrary \'etale extension, then $\SSl(A)=\SSl(A')$ (see \cite[Lemma 3.9]{Complutense2023},
and \cite[Theorem 5.5]{PrepJavier}). 

\medskip 

To conclude,  the Samuel slope has a nice behavior when restricted to primes in the same multiplicity strata of $\Spec(A)$. More precisely, it can be shown that if   $(A, {\mathfrak m}, k)$ is excellent, equidimensional and contains a field, and if ${\mathfrak p}\subset A$ is a prime so that $A/{\mathfrak p}$ is regular and with the same multiplicity as ${\mathfrak m}$, then $\SSl(A_{\mathfrak p})\leq \SSl(A)$
\cite[ Theorem 6.2]{PrepJavier}.
\end{parrafo}

\begin{parrafo} \label{definition_Cossart_Piltant} 
{\bf Cossart-Piltant expansions.} The following result of Cossart and Piltant will allow us to work with finite expansions in regular local rings in a very broad setting, with no need to pass to the completion nor assuming that the ring contains a field.    
\end{parrafo}

\begin{proposition} \label{ExpanCosPil}
	\cite[Proposition 2.1]{CossartPiltant2019}
	Let $(R,\mathfrak{m})$ be a regular local ring of dimension $n$, and let 
	$(u_1,\ldots,u_n)$ be a regular system of parameters of $R$. Then, for any $f\in R$ there is a unique finite set $S(f)\in\mathbb{N}^n$ such that:
	\begin{enumerate}
		\item[(i)] The set of monomials $\{u^{\alpha}\}_{\alpha\in S(f)}$ is a minimal set of 
		generators of the ideal
		$$\left\langle u^{\alpha}\mid \alpha\in S(f)\right\rangle.$$
		\item[(ii)] There is a   finite  collection of units  $\{c_{\alpha}(f)\}_{\alpha\in S(f)}$,
		$c_{\alpha}(f)\in R\setminus\mathfrak{m}$ such that
		\begin{equation} \label{EqExpansion}
			f=\sum_{\alpha\in S(f)}c_{\alpha}(f)u^{\alpha}.
		\end{equation}
	\end{enumerate}
\end{proposition}

Under the previous assumptions, we will refer to  expression  (\ref{EqExpansion}) as the  \emph{Cossart-Piltant expansion} of $f\in R$ 
with respect to $(u_1,\ldots,u_n)$, which often will be shortened as {\em CP-expansion w.r.t. $u$}.

\begin{remark}\label{ObsCosPil} It is worthwhile to emphasize a few points regarding the statement in the previous proposition:  
	\begin{itemize}
		\item[(i)] Both, the set $S(f)$ and the collection of  units $c_{\alpha}$,   depend on $f$.
		\item[(ii)] The set $S(f)$ is unique but the collection of units  $c_{\alpha}$ is not: 
		if $\alpha,\beta\in S(f)$ are two different elements, then  any term multiple of 
		$u^{\alpha+\beta}$ can be added either to $c_{\alpha}$ or to $c_{\beta}$.
		\item[(iii)]  If $(0,\ldots,0,r)\in S(f)$ then $(0,\ldots,0,r)$ is the unique element of $S(f)$ in the $u_n$-axis.
		\item[(iv)] For any $f\in R$ and any unit $v\in R$ we have that   
		$S(f)=S(vf)$.
		\item[(v)] It can be checked that 
	$$\ord_{\mathfrak{m}}(f)=\min\{|\alpha|   \mid  \alpha\in S(f)\}.$$
\item[(vi)] If $f,g\in R$ are so that
$S(f)\cap(S(g)+\mathbb{N}^n)=\emptyset$ and 
$S(g)\cap(S(f)+\mathbb{N}^n)=\emptyset$,
then $S(f+g)=S(f)\sqcup S(g)$.
		\item[(vii)] If we have two choices of units 
$\{c_{\alpha} : \alpha\in S(f)\}$ and $\{\tilde{c}_{\alpha} : \alpha\in S(f)\}$
such that 
$$f=\sum_{\alpha\in S(f)}c_{\alpha}(f)u^{\alpha}=\sum_{\alpha\in S(f)}\tilde{c}_{\alpha}(f)u^{\alpha},$$
then $c_{\alpha}(f)-\tilde{c}_{\alpha}(f)\in\mathfrak{m}$ for all $\alpha\in S(f)$.
	\end{itemize}
\end{remark}

The following lemma will be used in the last section of the paper to compare CP-expansions when referred to regular systems of parameters that are "${\mathfrak m}$-adically close". 
	
	\begin{lemma} \label{LemaCPaprox} Let $(R,\mathfrak{m})$ be a regular local ring, let $f\in R$ and 
		let $(y_1,\ldots,y_n)$ and $(z_1,\ldots,z_n)$ be two regular system of parameters of  $R$. Consider the CP-expansion of $f$ with respect to both,  $y$ and $z$:
		$$f=\sum_{\beta\in S_y(f)}c_{\beta}(y)y^{\beta},
		\qquad
		f=\sum_{\beta\in S_z(f)}c_{\beta}(z)z^{\beta}.$$
Suppose there is some  $N\in\mathbb{Z}_{\geq 1}$ such that $z_j-y_j\in\mathfrak{m}^{N+1}$ for  $j=1,\ldots,n$. Then, 		
		setting  $\mathcal{M}_N:=\{\beta\mid |\beta|\leq N\}$ we have that 
		$$S_y(f)\cap \mathcal{M}_N=S_z(f)\cap \mathcal{M}_N.$$
\end{lemma}
	
	\begin{proof} The CP-expansions of $f$ can be written as follows: 
		$$f=\sum_{\beta\in S_y(f)\cap\mathcal{M}_N}c_{\beta}(y)y^{\beta}+
		\sum_{\beta\in S_y(f)\setminus\mathcal{M}_N}c_{\beta}(y)y^{\beta}, \ \text{ and }  \  f=\sum_{\beta\in S_z(f)\cap\mathcal{M}_N}c_{\beta}(z)z^{\beta}+
		\sum_{\beta\in S_z(f)\setminus\mathcal{M}_N}c_{\beta}(z)z^{\beta}.$$
		Set 
		$$g_1:=\sum_{\beta\in S_y(f)\setminus\mathcal{M}_N}c_{\beta}(y)y^{\beta}-
		\sum_{\beta\in S_z(f)\setminus\mathcal{M}_N}c_{\beta}(z)z^{\beta},$$
		and note that $g_1\in\mathfrak{m}^{N+1}$. Now let $\beta_0\in S_y(f)\cap\mathcal{M}_N$. Then 
		$$c_{\beta_0}y^{\beta_0}=
		-\sum_{\substack{{\beta\in S_y(f)\cap\mathcal{M}_N}\\ {\beta\neq\beta_0}}}c_{\beta}(y)y^{\beta}+\sum_{\beta\in S_z(f)\cap\mathcal{M}_N}c_{\beta}(z)z^{\beta}-g_1,$$
		and define
		$$g_2:=\sum_{\beta\in S_z(f)\cap\mathcal{M}_N}c_{\beta}(z)(z^{\beta}-y^{\beta}). $$
		 Observe  that $g_2\in\mathfrak{m}^{N+1}$ and that
		$$c_{\beta_0}y^{\beta_0}=
		-\sum_{\substack{{\beta\in S_y(f)\cap\mathcal{M}_N}\\ {\beta\neq\beta_0}}}c_{\beta}(y)y^{\beta}+
		\sum_{\beta\in S_z(f)\cap\mathcal{M}_N}c_{\beta}(z)y^{\beta}+
		g_2-g_1.$$
		Now, since $c_{\beta_0}$ is a unit and $|\beta_0|\leq N$, it follows that
		$$y^{\beta_0}\in\left\langle y^{\beta}\mid \beta\in\left(\left( S_y(f)\cap\mathcal{M}_N\right)\setminus\{\beta_0\}\right)
		\cup \left( S_z(f)\cap\mathcal{M}_N\right) \right\rangle.$$
		Using the  properties of the  CP-expansions, 
		$y^{\beta_0}\not\in\left\langle y^{\beta}\mid \beta\in \left(S_y(f)\cap\mathcal{M}_N\right)\setminus\{\beta_0\} \right\rangle$,  thus we conclude that
		$y^{\beta_0}\in\left\langle y^{\beta}\mid \beta\in S_z(f)\cap\mathcal{M}_N \right\rangle$, and hence, 
		$$\left\langle y^{\beta}\mid \beta\in S_y(f)\cap\mathcal{M}_N \right\rangle
		\subset
		\left\langle y^{\beta}\mid \beta\in S_z(f)\cap\mathcal{M}_N \right\rangle.$$
		The other inclusion can be proven analogously to show that both ideals are equal. 
		To conclude, using again the properties of the CP-expansions, the set
		$S_y(f)\cap\mathcal{M}_N$ gives a minimal set of monomials in $y$ generating the ideal 
		$\left\langle y^{\beta}\mid \beta\in S_y(f)\cap\mathcal{M}_N \right\rangle$,
		from where we obtain the required equality.
	\end{proof}

\section{On the finiteness of the  Samuel slope of a local ring}
\label{SecFinita}
 
Our purpose in this section is to prove Theorems \ref{pendiente_finita},  \ref{pendiente_red}  and Corollary \ref{corolario_pendiente_finita}, regarding the finiteness the Samuel slope of reduced excellent non-regular rings. To this end we need a couple technical results.
To start with we state Lemma \ref{Lemadelacomplu} which plays a role when dealing
with rings in the extremal case.
We continue with Lemma \ref{lema_lema_v4}, which generalizes   \cite[Lemma 8.9]{Complutense2023} to the case of complete reduced Noetherian rings that do not necessarily contain a field, and that are not necessarily equidimensional. And lastly, Lemma \ref{LimiteNil}, characterizes limits of Cauchy sequences whose asymptotic order is arbitrary large. 
 
\begin{lemma} \cite[Proposition 8.6]{Complutense2023},
\cite[Lemma 4.4]{PrepJavier}
\label{Lemadelacomplu}
Let $(B,\mathfrak{m},k)$ be a Noetherian local ring of dimension $d\geq 1$ which is in the extremal case. Then $B$ contains a reduction of $\mathfrak{m}$ generated by $d$ elements. 
\end{lemma}

\begin{lemma}\label{lema_lema_v4} 
 Let $(R',\mathfrak{m}')$ be a complete reduced local ring of dimension $d$ and embedding dimension $d+t$, with $t>0$. Suppose that 
 	$R'$ is in the extremal case, and let    $u'_1,\ldots,u'_d\in    R'$  be any set of $d$ elements generating  reduction of 
 	$\mathfrak{m}'$. Let   $y'_1,\ldots,y'_t$ be a $\lambda_{\mathfrak{m}'}$-sequence and set
 	$$\ell:=\min \{\nub_{{\mathfrak m}'}(y'_i): i=1,\ldots, t\}.$$
Suppose that  $\delta'_1,\ldots, \delta'_t\in {\mathfrak m}'$ is another $\lambda_{{\mathfrak m}'}$-sequence with 
 	$$\min \{\nub_{{\mathfrak m}'}(\delta'_i): i=1,\ldots, t\}> \ell,$$ 
  then there are elements 
 	$s'_1,\ldots, s'_t\in R'$
 	such that setting $z'_i:=y'_i-s'_i$, for $i=1,\ldots, t$, we have that: 
 	\begin{enumerate}
 		\item[(i)] ${\mathfrak m}'=\langle u'_1,\ldots, u'_d, z'_1,\ldots, z'_t\rangle$; 
 		\item[(ii)] $\min \{\nub_{{\mathfrak m}'}(z_i'): i=1,\ldots, t\}\geq\min \{\nub_{{\mathfrak m}'}(\delta'_i): i=1,\ldots, t\}$;
 		\item[(iii)] $z'_1,\ldots,z'_t$ is a $\lambda_{{\mathfrak m}'}$-sequence;
 		\item[(iv)] for 
 		$j=1,\ldots,t$,  $\nub_{\mathfrak{m}'}(s_i)\geq\ell$, i.e., $s'_j\in \langle u'_1,\ldots, u'_d\rangle^{\geq\ell}$. 
 	\end{enumerate}		
\end{lemma}

\begin{proof}	From the hypotheses we have that,  
 	$${\mathfrak m}'/({\mathfrak m}')^2=\langle \In_{{\mathfrak m}'} u_1',\ldots, \In_{{\mathfrak m}'}u_d'\rangle \oplus \ker(\lambda_{{\mathfrak m}'}),$$
 	from where it follows that  given any  $\lambda_{{\mathfrak m}'}$-sequence, $\delta_1,\ldots,\delta_t$,  
 	$${\mathfrak m}'=\langle u_1,\ldots, u_{d}, \delta_1,\ldots, \delta_t\rangle.$$
If $R'$ contains a field and is equidimensional,  the statement was proved in \cite[Lemma 8.9]{Complutense2023}.
To treat all the cases,  
by \cite[Theorem 29.4]{Matsumura}, $R'$ is a quotient of a power series ring of the form $R=A[[V_1,\ldots,V_{d+t}]]$, where $A$ is either a field or a $p$-ring,  where $p$ is the characteristic of the residue field of $R'$.
Now we distinguish three cases, depending on the role of $p'$, the class of $p$ in $R'$ or if $R'$ contains a field.

\medskip
 	
 	{\bf Case 1.} Suppose    $p'\in ({\mathfrak m}')^2$. 
 	Note that then  $\dim(R)=d+t+1$, and if $\mathfrak{m}\subset R$ is the maximal ideal then 
 	$$\mathfrak{m}=\langle p,u_1,\ldots,u_d,y_1,\ldots,y_t\rangle=
 	\langle p,u_1,\ldots,u_d,\delta_1,\ldots,\delta_t\rangle,$$
 	for suitable choices of $u_j, y_i, \delta_i\in R$ mapping respectively to $u'_j, y_i', \delta_i'\in R'$, $j=1,\ldots, d, i=1\ldots, t$. For $i=1,\ldots,t$,   write the CP-expansion of $y_i\in R$
 	w.r.t. $(p,u,\delta)$, 
 	
 		$$y_i=\sum_{(j,\alpha,0)\in S(y_i)}c_{(j,\alpha,0)}(y_i)p^j u^{\alpha}+
 		\sum_{\tiny{\begin{array}{c}
 				(j,\alpha,\beta)\in S(y_i)\\
 			\beta\neq 0\end{array}}}c_{(j,\alpha,\beta)}(y_i)p^j u^{\alpha}\delta^{\beta}.$$

 	Set $s_i:=\sum_{(j,\alpha,0)\in S(y_i)}c_{(\alpha,0)}(y_i)p^j u^{\alpha}$, and let  $z_i:=y_i-s_i$, for $i=1,\ldots,t$. Then 
 	$\nub_{\mathfrak{m}'}(z'_i)\geq
 	\min \{\nub_{{\mathfrak m}'}(\delta'_j): j=1,\ldots, t\}$, 
 	and the inequality in (ii) follows. 
 	
 	Note that
 	\begin{equation}
 		\label{maximal_R}
 		{\mathfrak m}=\langle p,u_1,\ldots, u_d,y_1,\ldots, y_t\rangle=
 		\langle p,u_1,\ldots, u_d,z_{1},\ldots, z_{t}\rangle,
 	\end{equation}
 	observe also that ${{\mathfrak m}'}=\langle u'_1,\ldots, u'_d, z_1',\ldots, z_t'\rangle$,
since $p'\in( {\mathfrak m}')^2$, thus
   $z'_1,\ldots,z'_t$ is a $\lambda_{{\mathfrak  m}'}$-sequence.
 	
 	It is worthwhile noticing that
 	$$\min\{\nub_{{\mathfrak m}'}(z_i') : i=1,\ldots,t\} > \ell=\min\{\nub_{{\mathfrak m}'}(y_i') : i=1,\ldots,t\}. $$
 	Therefore, since $z'_i=y'_i-s'_i$ necessarily, 
 	$\nub_{{\mathfrak m}'}(s'_i)\geq \ell$ for  $i=1,\ldots,t$, and (iv) follows. 
 	
 	\medskip

 	{\bf Case 2.} Suppose 
 	$p'\in {\mathfrak m}'\setminus ({\mathfrak m}')^2$.
 	 Then  $R'$ is a quotient of a regular local ring of the form $R=A[[V_1,\ldots,V_{d+t-1}]]$, and its maximal ideal  
 	$\mathfrak{m}\subset R$ is such that 
 	$$\mathfrak{m}=\langle u_1,\ldots,u_d,y_1,\ldots,y_t\rangle=
 	\langle u_1,\ldots,u_d,\delta_1,\ldots,\delta_t\rangle,$$
 	for suitable choices of $u_j, y_i, \delta_i\in R$ mapping respectively to $u'_j, y_i', \delta_i'\in R'$, $j=1,\ldots, d, i=1\ldots, t$. Again, for $i=1,\ldots,t$, in $R$,  consider the CP-expansion of $y_i$
 	w.r.t. $(u,\delta)$, 
  $$y_i=\sum_{(\alpha,0)\in S(y_i)}c_{(\alpha,0)}(y_i) u^{\alpha}+
 		\sum_{\tiny{\begin{array}{c}
 				(\alpha,\beta)\in S(y_i)\\
 				\beta\neq 0
 				\end{array}}}c_{(\alpha,\beta)}(y_i) u^{\alpha}\delta^{\beta}.$$
 	Setting  $s_i=\sum_{(\alpha,0)\in S(y_i)}c_{(\alpha,0)}(y_i) u^{\alpha}$, and letting   $z_i:=y_i-s_i$, for $i=1,\ldots,t$, 
   the argument follows as in case 1.

{\bf Case 3.} If $R'$ contains a field, then $R'$ is a quotient of
a regular local ring of the form $R=k[[V_1,\ldots,V_{d+t}]]$ for some coefficient field $k$.
As above, the maximal ideal  
$\mathfrak{m}\subset R$ is such that 
$$\mathfrak{m}=\langle u_1,\ldots,u_d,y_1,\ldots,y_t\rangle=
\langle u_1,\ldots,u_d,\delta_1,\ldots,\delta_t\rangle,$$
for suitable choices of $u_j, y_i, \delta_i\in R$ mapping respectively to $u'_j, y_i', \delta_i'\in R'$, $j=1,\ldots, d, i=1\ldots, t$.
And the result follows as in case 1.
\end{proof}

\begin{remark}\label{lemma_lemma_remark}
Under the assumptions of Lemma \ref{lema_lema_v4}, if $R'$ is equicharacteristic and if $k$ is the residue field of $R'$ then $R'=k[[V_1,\ldots,V_{d+t}]]$.
Moreover, $R'$ is finite over $S=k[[u_1',\ldots, u'_d]]$, and the elements $s'_1,\ldots, s'_t$ can be taken in $S$, with the same arguments as in case 2. In such case, the statement in (iv) is that $s_j'\in \langle u'_1,\ldots, u'_d\rangle ^{\ell}\subset S$ (cf.,  \cite[Lemma 8.9]{Complutense2023}). The same statement holds if  $R'$ does not contain a field, but there is a regular local subring $S\subset R$ with maximal ideal ${\mathfrak m}_S=\langle u_1',\ldots,u_d'\rangle$. 
\end{remark} 
\begin{lemma} \label{LimiteNil}
Let $(R',\mathfrak{m}')$ be a local Noetherian ring.
Let $\{\theta_n\}_{n=1}^{\infty}$ be a sequence of elements in $R'$ such that
$\lim_{n\to\infty}\theta_n=\theta\in R'$ for the $\mathfrak{m}'$-adic topology.
If $\lim_{n\to \infty} \nub_{\mathfrak{m}'}(\theta_n)=\infty$,  then $\theta$ is nilpotent in $R'$.
\end{lemma}

\begin{proof}
 	It suffices to prove that $\nub_{\mathfrak{m}'}(\theta)=\infty$.  Since  $\theta_n\to\theta$ and
 	$\nub_{\mathfrak{m}'}(\theta_n)\to\infty$, for each   $N\in {\mathbb N}_{>0}$,  there is an integer 
 	$n_0>0$ such that
 	$$\theta-\theta_n\in{\mathfrak{m}'}^N,
 	\quad\text{and}\quad \nub_{\mathfrak{m}'}(\theta_n)\geq N
 	\qquad\forall n\geq n_0.$$
 	Now the result follows  since $\nub_{\mathfrak{m}'}(\theta)\geq
 	\min\{\nub_{\mathfrak{m}'}(\theta-\theta_n),\nub_{\mathfrak{m}'}(\theta_n)\}\geq N$. 
 \end{proof}

 \begin{theorem}
 	\label{pendiente_finita}
 	Let $(R',\mathfrak{m}')$ be a an excellent reduced  non-regular local ring.
 	Then $\SSl(R')\in\mathbb{Q}$.
 \end{theorem}
 
 \begin{proof} The statement if clear if   $\SSl(R')=1$, thus 
 	we can assume that $R'$ is in the extremal case. 
 	Since $R'$ is excellent, we can also assume that $R'$ is complete, reduced and non-regular (see \S \ref{properties_slope}).
 	Now, by Theorem \ref{nu_barra_valuations}, 
 	 there is some $m\in\mathbb{Z}_{\geq 1}$ such that
 	$\nub_{\mathfrak{m}'}(a')\in\frac{1}{m}\mathbb{Z}_{\geq 1}$ for all $a'\in R'\setminus\{0\}$, hence  
 	\begin{equation}
 		\label{casi_natural}
 		\SSl(R')\in \frac{1}{m}\mathbb{Z}_{\geq 1}\cup\{\infty\}. 
 	\end{equation}
 	
 	Suppose, to get a contradiction,  that $\SSl(R')=\infty$.
 	Since $R'$  is in the extremal case there are 
 	$d=\dim(R')$ elements $u'_1,\ldots,u'_d$ generating a
 	reduction of $\mathfrak{m}'$. Let $t>0$ be the excess of embedding dimension of $R'$. Since the assumption is that $\SSl(R')=\infty$,  there is a sequence of 
 	$\lambda_{\mathfrak{m}'}$-sequences
 	$\{{\delta'}_1^{(i)},\ldots,{\delta'}_t^{(i)}\}_{i=1}^{\infty}$ such that
 	$$\lim_{i\to\infty}\min\left\lbrace
 	\nub_{\mathfrak{m}'}({\delta'}_1^{(i)}),\ldots,
 	\nub_{\mathfrak{m}'}({\delta'}_t^{(i)})
 	\right\rbrace=\infty.$$

 	Set ${\theta'}_j^{(1)}:={\delta'}_j^{(1)}$ for $j=1,\ldots,t$.
 	By Lemma \ref{lema_lema_v4}, for $j=1,\ldots,t$, there are sequences
 	$\{{s'_j}^{(i)}\}_{i=1}^{\infty}\subset R'$, such that 
 	if we set ${\theta'}_j^{(i+1)}:={\theta'}_1^{(i)}+{s'}_j^{(i)}$,  then
 	\begin{itemize}
 		\item $s_j^{(i)}\in\langle u'_1,\ldots,u'_d\rangle^{\geq\ell_i}$, with 
 		$\lim_{i\to\infty}\ell_i=\infty$; 
 		\item ${\theta'}_1^{(i)},\ldots,{\theta'}_t^{(i)}$ form   a
 		$\lambda_{\mathfrak{m}'}$-sequence;
 		\item for $i\geq 1$,
 		$$\min\left\lbrace
 		\nub_{\mathfrak{m}'}({\theta'}_1^{(i+1)}),\ldots,
 		\nub_{\mathfrak{m}'}({\theta'}_t^{(i+1)})
 		\right\rbrace>
 		\min\left\lbrace
 		\nub_{\mathfrak{m}'}({\theta'}_1^{(i)}),\ldots,
 		\nub_{\mathfrak{m}'}({\theta'}_t^{(i)})
 		\right\rbrace,$$
 		and by construction   $\lim_{i\to \infty} \min\left\lbrace
 		\nub_{\mathfrak{m}'}({\theta'}_1^{(i)}),\ldots,
 		\nub_{\mathfrak{m}'}({\theta'}_t^{(i)})
 		\right\rbrace=\infty$.
 	\end{itemize}
 	
 	Now observe that the filtration
 	$\{\langle u'_1,\ldots,u'_d\rangle^{\geq\ell}\}_{\ell\in\frac{1}{m}\mathbb{Z}_{\geq 1}}$ is $\mathfrak{m}'$-stable by \cite[Lemma 10.8]{Atiyah} and the proof
 	of \cite[Theorem 3.4]{PrepJavier}.
 Thus, by  construction,  $\{{\theta'}_1^{(i)}\}_{i=1}^{\infty}$ is a Cauchy sequence for the  $\mathfrak{m}'$-adic topology, because 
 	$${\theta'}_1^{(i+1)}-{\theta'}_1^{(i)}=s_1^{(i)}\in
 	\langle u'_1,\ldots,u'_d\rangle^{\geq\ell_i}.$$
 Since $R'$ is complete, the sequence $\{{\theta'}_1^{(i)}\}_{i=1}^{\infty}$
 	has a limit $\tilde{\theta}'$ in $R'$. Finally, by Lemma \ref{LimiteNil}, $\tilde{\theta}'$ is nilpotent.  But this is a contradiction since $\tilde{\theta}'\not\in({\mathfrak{m}'})^2$ and
 	$R'$ is reduced.
\end{proof}

 \begin{theorem}
 	\label{pendiente_red}
 	Let $(R',\mathfrak{m}')$ be an excellent ring local ring.
 	Then $\SSl(R')=\SSl(R'_{\red})$.
 \end{theorem}
 
 \begin{proof}
 	The proof goes as in \cite[Theorem 4.7]{PrepJavier} once Theorem \ref{pendiente_finita}
 	is proven for excellent reduced non-regular local rings.
 \end{proof}
 
 \begin{corollary}
 	\label{corolario_pendiente_finita}
 	Let $(R',\mathfrak{m}')$ be an excellent   local ring.
 	Then $\SSl(R')=\infty$ if and only if $R'_{\red}$ is a regular local ring.
 \end{corollary}

\section{Finite expansions and generalization of a Theorem of Hickel}
\label{CossartPiltantHickel}

In \cite{Hickel}, Hickel gives a formula for the computation of the assymptotic Samuel function for local domains that contain a field. More precisely he shows: 
\begin{theorem}\label{OriginalHickel}
\cite[Theorem 2.1]{Hickel}
Let $(R', {\mathfrak m}',k)$ be a Noetherian, complete, equicharacteristic  local ring  of Krull dimension $n'$. 
Let $I\subset {\mathfrak m}'$ be an  ${\mathfrak m}'$-primary ideal that has a reduction generated by $n'$ elements,
$J=\langle x'_1,\ldots,x'_{n'}\rangle$.
Let $S=k[[x'_1,\ldots,x'_{n'}]]\subset R'$ and denote by $\mathfrak{m}_S$ the maximal
ideal of $S$. Let $\theta\in R'$. If 
$$p(Z)=Z^{\ell}+a_1Z^{\ell-1}+\ldots+a_{\ell}$$
is the minimal polynomial of $\theta\in R'$ over the fraction field of $S$, $K(S)$, then 
$p(Z)\in S[Z]$ and 
\begin{equation}
	\label{EqOriginalHickel}
	\nub_I(\theta)=\min_i \left\{\frac{\nu_{{\mathfrak m}_S} (a_i)}{i}: i=1,\ldots,\ell\right\}.
\end{equation}
\end{theorem} 

\begin{parrafo} {\bf Finite transversal extensions.} Hickel's formula can be used for any local Noetherian ring that contains a field,  since we can always reduce to the domain case, and then, consider the completion  (see \cite[Section 2]{Hickel} for precise details). However, the crucial part of Hickel's argument is the existence of the regular local ring $S$ and a particular property of the containment $S\subset R'$. To be more precise, consider a finite extension of local rings $S\subset R'$, where $(S,\mathfrak{m}_S)$ is
regular and such that every non-zero elements of $S$ is not a zero divisor in $R'$. Let $L=K(S)\otimes_S R'$ where $K(S)$ is the fraction field of $S$. 
By Zariski's multiplicity formula
\cite[Theorem 24, page 297 and Corollary 1, page 299]{ZariskiSamuel1960} we have the inequality
\begin{equation}\label{ZariskiFormula}
e(R')\leq [L:K(S)],
\end{equation}
where $e(R')$ denotes the multiplicity of $R'$ at the ${\mathfrak m}'$. 
We say that the finite extension $S\subset R'$ is \emph{finite-transversal} if the  
equality holds in (\ref{ZariskiFormula}), see \cite[\S 4.8]{V} and also \cite{Handbook}.
\end{parrafo} 

\begin{remark} \label{transversal_equiv}
Given a regular local ring $(S,\mathfrak{m}_S)$ and a finite-transversal extension 
$S\subset R'$ at some maximal ideal $\mathfrak{m}'\subset R'$ (here $R'$ might not be local),
by  Zariski's multiplicity formula for finite projections it follows that:
\begin{itemize}
\item[(i)] The ideal ${\mathfrak m}'\subset R'$ is the unique maximal ideal dominating ${\mathfrak m}_S$, hence $R'$ is local with maximal ideal ${\mathfrak m}'$; 
\item[(ii)] The residue field of $R'$, $k'$,    is equal  to the residue field of $S$;
\item[(iii)] $e_{R'}({\mathfrak m}_SR')= e_{R'}$ (in particular, if $R'$ is formally equidimensional, then  ${\mathfrak m}_SR'$ is a reduction of the maximal ideal ${\mathfrak m}'$ by Rees' Theorem).
\end{itemize}
\end{remark}

When $R'$ contains a field $k'$,  the existence of the finite transversal extension $S\subset R'$ is ensured, at least after enlarging $k'$, and  then passing to the ${\mathfrak m}'$-adic completion of $R'$. In general, this argument cannot be carried on if $R'$ is not equicharacteristic. In Theorem \ref{GeneralHickel} below we give a formula for the computation of the asymptotic Samuel function of certain elements  when $R'$ is the quotient of a local regular ring $R$ by a principal ideal. In this setting, the  lack of a finite transversal extension  will be replaced by the    use of  CP-expansions. This will allow us to give a version of Hickel's statement in the case of local rings that do not necessarily contain a field,  see Theorem \ref{General_General_Hickel}.  

\medskip

In the next lines we establish some technical results on CP-expansions in regular local rings that will be needed for the proof of the theorem. Firstly, we will see that, even though CP-expansions might not be compatible with addition or products, they still satisfy nice properties that allow us to define quite general (monomial) valuations in $R$ (see Proposition \ref{valoracion_definida}).  Secondly, we will introduce the notion of pseudo-Weierstrass elements to refer to those elements with a special form in their CP-expansions (see Definition \ref{DefWeierstrass}), and will establish some properties, see Proposition \ref{mismo_orden}. 

\ 

\noindent{\bf CP-expansions and monomial valuations}

\medskip

\noindent Recall that 
 a monomial ordering in $\mathbb{N}^n$ is a total ordering in
$(\mathbb{N}^n,\prec)$ such that if $\alpha\prec\beta$ and $\gamma\in\mathbb{N}^n$
then $\alpha+\gamma\prec\beta+\gamma$).

\begin{lemma} \label{LemaCPproduct}
Let $(R,\mathfrak{m})$ be a regular local   ring of Krull dimension $n$.
Let $(u_1,\ldots,u_n)$ be a regular system of parameters in $R$, and fix 
 any monomial ordering $\prec$ in $\mathbb{N}^n$ such that $1\prec u_i$ for
$i=1,\ldots,n$.
Let  $f,g\in R$ and consider their  CP-expansions as in  (\ref{EqExpansion})
$$f=\sum_{\alpha\in S(f)}c_{\alpha}(f)u^{\alpha},
\qquad
g=\sum_{\alpha\in S(g)}c_{\alpha}(g)u^{\alpha}.$$
Set: 
$$\alpha_0:=\min\left\lbrace
\alpha : \alpha\in S(f)
\right\rbrace,
\qquad
\beta_0:=\min\left\lbrace
\beta : \beta\in S(g)
\right\rbrace,$$
where the minimum is taken for the total ordering $\prec$. 
Then
$$\alpha_0+\beta_0\in S(f\cdot g) \qquad
\text{and} \qquad
\alpha_0+\beta_0=\min\left\lbrace
\gamma : \gamma\in S(f\cdot g)
\right\rbrace.$$

Moreover there is a choice of units $\{c_{\gamma}(f\cdot g): \gamma\in S(f\cdot g)\}$
for the CP-expansion of $f\cdot g$
such that $c_{\alpha_0+\beta_0}(f\cdot g)=c_{\alpha_0}(f)c_{\beta_0}(g)$.
\end{lemma}

\begin{proof}
First note that by definition and the fact that $\prec$ is a monomial ordering
we have that
$$\alpha_0+\beta_0=\min\{\alpha+\beta : \alpha\in S(f),\ \beta\in S(g)\},$$
since for $\alpha\in S(f)$ and $\beta\in S(g)$,  we have that 
$$\alpha_0+\beta_0\prec \alpha+\beta_0\prec\alpha+\beta.$$
The product of $f$ and $g$ is
\begin{equation}\label{EqCPproducto}
f\cdot g=
\sum_{\alpha\in S(f),\beta\in S(g)}c_{\alpha}(f)c_{\beta}(g)u^{\alpha+\beta}=
\sum_{\gamma}\left(\sum_{\substack{{\alpha\in S(f), \beta\in S(g)}\\{\alpha+\beta=\gamma}}} c_{\alpha}(f)c_{\beta}(g)\right)u^{\gamma}.
\end{equation}
The CP-expansion of $f\cdot g$ can be obtained from (\ref{EqCPproducto}) by 
factoring-out monomials in the $u$'s for the 
terms multipliying $u^{\gamma}$
$$\sum_{\substack{{\alpha\in S(f), \beta\in S(g)}\\{\alpha+\beta=\gamma}}}	 c_{\alpha}(f)c_{\beta}(g),$$
and rearranging.
Note that some $\gamma=\alpha_1+\beta_1$, $\alpha_1\in S(f)$ and $\beta_1\in S(g)$,
is such that $\gamma\not\in S(f\cdot g)$ if there are 
$\alpha\in S(f)$ and $\beta\in S(g)$ with $u^{\alpha+\beta}$ dividing
$u^{\gamma}=u^{\alpha_1+\beta_1}$.

For $\gamma=\alpha_0+\beta_0$,
if $\alpha\in S(f)$ and $\beta\in S(g)$ are such that 
$u^{\alpha+\beta}$ divides $u^{\alpha_0+\beta_0}$, then
$\alpha+\beta\preceq\alpha_0+\beta_0$ and 
now the minimality of $\alpha_0$ and $\beta_0$
implies that $\alpha_0=\alpha$ and $\beta_0=\beta$.
Thus, $\alpha_0+\beta_0\in S(f\cdot g)$ and we can take
$c_{\alpha_0+\beta_0}(f\cdot g)=c_{\alpha_0}(f)c_{\beta_0}(g)$, since it is the only
term in the expression (\ref{EqCPproducto}) contributing to $u^{\alpha_0+\beta_0}$.
\end{proof}

As a consequence of the previous result, we have the following: 

\begin{proposition}\label{valoracion_definida}
Let $(R,\mathfrak{m})$ be a regular local   ring of Krull dimension $n$, and let 
 $u_1,\ldots,u_n\in R$ be a regular system of parameters. Then, for every    $(c_1,\ldots,c_n)\in\mathbb{Z}^n_{>0}$,
there is a (monomial) valuation $\omega:R\to \mathbb{Z}$ such that $\omega(u_i)=c_i$, and for $f\in R$ with CP-expansion 
$$f=\sum_{\alpha\in S(f)}c_{\alpha}(f)u^{\alpha},$$
 we have that:
$$\omega(f)=\min\{\omega(u^{\alpha}) \mid \alpha\in S(f)\},$$ 
where $\omega(u^{\alpha})=c_1\alpha_1+\cdots+c_n\alpha_n$.
\end{proposition}

\begin{proof}
Since $S(f)$ is well defined, the function $\omega$ is well defined.
Thus, we only have to prove that $\omega$ is a valuation.
Let $f,g\in R$ and consider the  CP-expansions, 
$$f=\sum_{\alpha\in S(f)}c_{\alpha}(f)u^{\alpha},
\qquad
g=\sum_{\alpha\in S(g)}c_{\alpha}(g)u^{\alpha}.$$
Then,  
$$f+g=\sum_{\alpha\in S(f)\cap S(g)}(c_{\alpha}(f)+c_{\alpha}(g))u^{\alpha}+
\sum_{\alpha\in S(f)\setminus S(g)}c_{\alpha}(f)u^{\alpha}+
\sum_{\alpha\in S(g)\setminus S(f)}c_{\alpha}(g)u^{\alpha},$$
and we conclude that
$$S(f+g)\subset \left( S(f)\cup S(g)\right)+\mathbb{N}^n.$$
Note that
$\min\{\omega(f),\omega(g)\}=\min\{\omega(u^{\alpha})\mid \alpha\in S(f)+S(g)\}$,
and the previous inclusion implies that
$\omega(f+g)\geq\min\{\omega(f),\omega(g)\}$.

To prove that $\omega(f\cdot g)=\omega(f)+\omega(g)$ consider any monomial
ordering $\prec$ in $\mathbb{N}^n$ such that
$$\omega(u^{\alpha})<\omega(u^{\beta}) \Longrightarrow \alpha\prec\beta.$$
Let $\alpha_0$ and $\beta_0$ as in Lemma \ref{LemaCPproduct}, 
 then  $\omega(f)=\omega(u^{\alpha_0})$,  
$\omega(g)=\omega(u^{\beta_0})$, $\alpha_0+\beta_0\in S(f\cdot g)$, and
$\omega(f\cdot g)=\omega(u^{\alpha_0+\beta_0})=
\omega(u^{\alpha_0})+\omega(u^{\beta_0})=\omega(f)+\omega(g)$.
\end{proof}

\medskip

\noindent{\bf Pseudo-Weierstrass elements}

\medskip

\begin{lemma}
	Let $(R,\mathfrak{m})$ be a  regular local ring of dimension $n$.
	Let $(u_1,\ldots,u_{n-1},y)$ be a regular system of parameters of $R$. Let $f\in R$.  Then
	$$f\not\in\langle u \rangle
	\quad\Longleftrightarrow\quad
	S(f)\cap(\{0\}\times\cdots\{0\}\times\mathbb{N})\neq \emptyset.$$
\end{lemma}

\begin{proof}
	If $S(f)\cap(\{0\}\times\cdots\{0\}\times\mathbb{N})=\emptyset$ then it is clear that
	$f\in\langle u\rangle$.
	
	On the other hand, if $S(f)\cap(\{0\}\times\cdots\{0\}\times\mathbb{N})$ is non-empty, then it contains a unique element,   $(0,\ldots,0,\ell)$.
	By condition (i) in Proposition \ref{ExpanCosPil} we have that if
	$(\alpha,i)\in S(f)$ then
	\begin{itemize}
		\item either $\alpha=0$ and $i=\ell$, or
		\item $i<\ell$ and $\alpha\neq 0$.
	\end{itemize}
	By Proposition \ref{ExpanCosPil}(ii), there are units
	$\{c_{\alpha,i}\}_{(\alpha,i)\in S(f)}$ such that $f$ can be written as
	\begin{equation}\label{EqExpGradl}
		f=c_{0,\ell}y^{\ell}+\sum_{\substack{(\alpha,i)\in S(f) \\ i<\ell}}c_{\alpha,i}u^{\alpha}y^i.
	\end{equation}
	Now, the claim is  clear since $f-c_{0,\ell}y^{\ell}\in\langle u\rangle$ and
	$c_{0,\ell}y^{\ell}\not\in\langle u\rangle$.
\end{proof}

\begin{definition} \label{DefWeierstrass}
Let $(R,\mathfrak{m})$ be a regular local ring of dimension $n$, and let 
$(u_1,\ldots,u_{n-1},y)$ be a regular system of parameters of $R$. 
We say that $f\in R$ is a \emph{pseudo-Weierstrass} element of degree $\ell$ w.r.t. $(u,y)$ if: 
\begin{enumerate}
\item[(i)] $(0,\ldots,0,\ell)\in S(f)$; 
\item[(ii)] there are elements 
$a_1,\ldots,a_{\ell}\in\langle u\rangle$ such that
\begin{equation} \label{EqWeierstrass}
  f=y^{\ell}+a_1y^{\ell-1}+\cdots+a_{\ell-1}y+a_{\ell};
\end{equation}
\item[(iii)] for any $i=1,\ldots,\ell$,
\begin{equation} \label{CoefSai}
  S(a_i)\subset \mathbb{N}^{n-1}\times\{0\}.
\end{equation}
\end{enumerate}   	 
\end{definition}
Note that, once $u,y$ are fixed, the degree $\ell$ just depends on $f$.

\begin{proposition} \label{uyforma}
Let $(R,\mathfrak{m})$ be a  regular local ring of dimension $n$.
Let $(u_1,\ldots,u_{n-1},y)$ be a regular system of parameters of $R$. Then, for any $f\in R$ such that $f\not\in\langle u_1,\ldots,u_{n-1}\rangle$ there is a unique
$\ell=\ell_{f}\in\mathbb{N}$ and some unit $v\in R$ such that
$vf$ is a pseudo-Weierstrass element of degree $\ell$ w.r.t. $(u,y)$.
\end{proposition}

\begin{proof}
Since $f\not\in\langle u\rangle$, there is an expansion as in (\ref{EqExpGradl}).
With the notation of (\ref{EqExpGradl}),
set $v:=c_{0,\ell}^{-1}$, and  for $i=1,\ldots,\ell$ let:
\begin{equation}\label{coeficientes_Weier} 
	a_i:=v\cdot \sum_{(\alpha,i)\in S(f)}c_{\alpha,i}u^{\alpha}.
\end{equation} 
Note that the previous expression is the expansion of $a_i$ as in (\ref{EqExpansion}) and
$$S(a_i)=\{(\alpha,0) \mid (\alpha,i)\in S(f)\}.$$
\end{proof}

\begin{lemma} \label{LemPropCoef}
Let $(R,\mathfrak{m})$ be a regular local ring of dimension $n$.
Let $(u_1,\ldots,u_{n-1},y)$ be a regular system of parameters of $R$. Let $f$ be a pseudo-Weierstrass element w.r.t. $(u,y)$,  and consider an expression as in (\ref{EqWeierstrass}) satisfying (\ref{CoefSai}).
Then
\begin{enumerate}
\item[(i)] $\ord_{\mathfrak{m}}(a_i)=\ord_{\langle u\rangle}(a_i)$, for all $i=1,\ldots,\ell$, and 
\item[(ii)] $S(f)=\{(0,\ell)\}\sqcup\left(\bigsqcup_{i=1}^{\ell}\left(S(a_i)\times\{\ell-i\}\right)\right)$.
\end{enumerate}	
\end{lemma}

\begin{proof} 
(i) This  is a direct consequence of Remark \ref{ObsCosPil}(v) and (\ref{CoefSai}).

(ii)  Note that $S(a_iy^{\ell-i})=S(a_i)\times\{\ell-i\}$, and the claim follows
from Remark \ref{ObsCosPil}(vi).
 \end{proof}

\begin{proposition}\label{mismo_orden}
Let $(R,\mathfrak{m})$ be a regular local ring of dimension $n$, and suppose 
$(u_1,\ldots,u_{n-1},y)$ form   a regular system of parameters of $R$. Let $f\in R$ be a pseudo-Weierstrass element w.r.t. $(u,y)$, and let   $a_1,\ldots,a_{\ell}\in R$ and $a'_1,\ldots,a'_{\ell}\in R$ be elements 
satisfying properties (\ref{EqWeierstrass}) and (\ref{CoefSai}) of Definition \ref{DefWeierstrass}.
Then  $\ord_{\mathfrak{m}}(a_i)=\ord_{\mathfrak{m}}(a'_i)$ for all $i$.
\end{proposition}

\begin{proof}
	By Lemma \ref{LemPropCoef} (ii) we have that $S(a_i)=S(a'_i)$, and the result follows 
	from Remark \ref{ObsCosPil}~(v).
\end{proof}

\medskip

\noindent{\bf A version of Hickel's Theorem}

\begin{theorem} \label{GeneralHickel}
Let $(R,\mathfrak{m})$ be a regular local ring of dimension $n$, and let  $(u_1,\ldots,u_n)$ be any regular system of parameters in $R$.
Let $f\in R\setminus \{0\}$ be a non-unit,  set $R'=R/\langle f\rangle$ and let ${\mathfrak m}'={\mathfrak m}/\langle f\rangle$.  Then: 
\begin{itemize}
\item[(i)] After relabeling the elements $u_1,\ldots,u_n$, we may assume that
$$\nub_{{\mathfrak m}'}(u'_i)=1, \ \ \ \text{ for } i=1,\ldots, n-1,$$
where $u'_i$ denotes  the class of $u_i$ in $R'$.
\end{itemize}
Set $y=u_{n}$ and let $y'$ denote the class of $y=u_{n}$  in $R'$. 
\begin{itemize} 
\item[(ii)] If  $\nub_{{\mathfrak m}'}(y')\in {\mathbb Q}_{>1}\cup \{\infty\}$ then: 
\begin{enumerate}
\item $f\notin\langle u_1,\ldots, u_{n-1}\rangle$ and $y'\in \overline{ \langle u_1',\ldots, u'_{n-1}\rangle} \subset R'$;
\item up to a unit,  $f$ can be written as a pseudo-Weierstrass element of degree $\ell=e(R')>1$ w.r.t. $(u_1,\ldots, u_{n-1},y=u_{n})$, 
\begin{equation} \label{expresion_f} 
	f=y^{\ell}+\sum_{i=1}^{\ell}a_i y^{\ell -i}, 
\end{equation}
and 
\begin{equation} \label{formula_orden}
	\nub_{{\mathfrak m}'}(y')=\min_{i}\left\{\frac{\ord_{\mathfrak m}(a_i)}{i}: i=1,\ldots, \ell\right\}.
\end{equation}
\end{enumerate}
\end{itemize}
\begin{itemize}
\item[(iii)] If $\nub_{{\mathfrak m}'}(y')=1$ and if $f\notin\langle u_1,\ldots,u_{n-1}\rangle$, then up to a unit $f$ can be written as a pseudo-Weierstrass element w.r.t. $(u_1,\ldots, u_{n-1};y=u_{n})$ as in (\ref{expresion_f}) and formula (\ref{formula_orden}) holds if and only if $\langle u_1',\ldots, u'_{n-1}\rangle$ is a reduction of ${\mathfrak m'}$.  
\end{itemize}
\end{theorem}

\begin{proof} 
(i) The degree one part of $\Gr_{{\mathfrak m'}}(R')$ is generated by $\In_{{\mathfrak m}'} u_1', \ldots, \In_{{\mathfrak m}'} u'_{n-1}, \In_{{\mathfrak m}'}u_{n}'$, and 
		$$\Gr_{{\mathfrak m}'}(R')=\Gr_{{\mathfrak m}}(R)/\langle \In_{\mathfrak m} f\rangle \simeq (R/{\mathfrak m})[U_1,\ldots, U_{n}]/\langle \In_{\mathfrak m} f\rangle.$$

Now assertion 
follows from the fact that $\dim(\ker \lambda_{{\mathfrak m}'})\leq 1$: after relabeling the parameters, we can assume that $\In_{{\mathfrak m}'}u_1',\ldots, \In_{{\mathfrak m}'}u'_{n-1}\notin \ker\lambda_{{\mathfrak m}'}$, and hence $\nub_{{\mathfrak m}'} (u_i')=1$, for $i=1,\ldots,n-1$. 
\medskip

(ii) If $\nub_{{\mathfrak m}'}(y')>1$ then $\In_{{\mathfrak m}'}y'$ is nilpotent in $\Gr_{{\mathfrak m}'}(R')$  and therefore $\langle {u}_1',\ldots, {u}'_{n-1}\rangle$ is a reduction of ${\mathfrak m}'$. Now, since $(u_1,\ldots, u_{n-1})$ is part of a regular system of parameters in $R$, $$\In_{\mathfrak m} \langle u_1,\ldots, u_{n-1}\rangle= \langle  \In_{\mathfrak m}u_1,\ldots, \In_{\mathfrak m}u_{n-1}\rangle.$$

If $f\in \langle u_1,\ldots, u_{n-1}\rangle$ then $\In_{\mathfrak m} f\in \langle  \In_{\mathfrak m} u_1,\ldots, \In_{\mathfrak m} u_{n-1}\rangle$, and the graded ring 
$$\Gr_{{\mathfrak m}'}(R')/\langle\In_{\mathfrak m'}u_1',\ldots, \In_{\mathfrak m'}u'_{n-1}\rangle =
\Gr_{{\mathfrak m}}(R)/\langle \In_{\mathfrak m}f,\In_{\mathfrak m}u_1,\ldots, \In_{\mathfrak m}u_{n-1}\rangle=$$
$$=\Gr_{{\mathfrak m}}(R)/\langle  \In_{\mathfrak m}u_1,\ldots, \In_{\mathfrak m}u_{n-1}\rangle\simeq (R/{\mathfrak m})[U_n]$$
would have dimension 1, a contradiction to the fact that $u_1',\ldots, u'_{n-1}$ generate a reduction of ${\mathfrak m}'$.

Hence $f\notin \langle u_1,\ldots, u_d\rangle$, and by Proposition \ref{uyforma}, up to a unit, $f$ can be written  as a pseudo-Weierstrass element of degree $\ell$ w.r.t. $(u,y)$: 
\begin{equation}
\label{desarrollo_f}
f=y^{\ell}+\sum_{i=1}^{\ell}a_i y^{\ell -i}. 
\end{equation}
\medskip

Observe first    that if in (\ref{desarrollo_f})  $a_i=0$ for $i=1,\ldots, \ell$, then  $y'$ is nilpotent, thus  $\nub_{{\mathfrak m}'}(y')=\infty$ and  formula  (\ref{formula_orden}) trivially  holds. Next, we will check that if $\nub_{{\mathfrak m}'}(y')=\infty$, then, necessarily $a_i=0$ for $i=1,\ldots, \ell$, and therefore,   formula (\ref{formula_orden}) holds too.	
\medskip

The next argument has been adapted from Hickel's \cite{Hickel}.	
Suppose, to arrive to a contradiction, that there is some $i\in \{1,\ldots, \ell\}$ such that $a_i\neq 0$. Then consider the rational number
	\begin{equation} \label{def_p_q} 
		\frac{p}{q}:=\min_{i}\left\{\frac{\ord_{\mathfrak m}(a_i)}{i}: i=1,\ldots, \ell\right\}=\frac{\ord_{\mathfrak m}(a_{i_0})}{i_0},
	\end{equation}
	for some	
	$i_0\in \{1,\ldots, \ell\}$.

	\medskip
	
Let $r,s\in {\mathbb Z_{\geq 0}}$ with $\frac{r}{s}\geq 1$. From the assumption,       $\nub_{{\mathfrak m}'}({y'})\geq \frac{r}{s}\geq 1$. Our purpose is to show that  $\frac{p}{q}\geq \frac{r}{s}$.     Since    $\langle {u}_1',\ldots, {u}_d'\rangle$ is a reduction of ${\mathfrak m}'$, 
	\begin{equation} \label{y_en_clausura}
		{y'}^s\in \overline{\langle {u}_1',\ldots, {u}_d' \rangle^r} = \overline{\langle  {u'}_1^r,\ldots, {u'}_d^r \rangle}\subset R'. 
	\end{equation}
	Thus, there is a monic polynomial 
	$Q(Z)\in R'[Z]$, 
	$$Z^m+\alpha_1Z^{m-1}+\ldots +\alpha_m$$
	so that $\alpha_i\in \langle {u'}_1^r,\ldots, {u'}_d^r\rangle^i$ for $i=1,\ldots, m$, and 
	$Q({y'}^s)=0\in R'$. Setting 
	$g(Z)=Z^{sm}+\alpha_1Z^{s(m-1)}+\ldots +\alpha_m$ we have that $g({y}')=0\in R'$. Now choose $A_i\in \langle u_1^r,\ldots, u_d^r\rangle^i\subset R$ so that 
	$\alpha_i={A}_i'\in R'$, $i=1,\ldots,m$,   and define, 
	\begin{equation} \label{g_en_R}
		g(y):=y^{sm}+A_1y^{s(m-1)}+\ldots+A_m\in R. 
	\end{equation}
	Since $g({y}')=0\in R'$, $g(y)\in \langle f\rangle$,  
	thus, there is some $h\in R$ such that 
	\begin{equation} \label{factorizacion_g} 
		g(y)=y^{sm}+A_1y^{s(m-1)}+\ldots+A_m=h\cdot f\in R.
	\end{equation}
	Because $g\notin \langle u_1,\ldots, u_d\rangle$, we have  that $h\notin \langle u_1,\ldots, u_d\rangle$.  From here it follows that  if we write  $h$ in pseudo-Weierstrass form, there must be a term in $h$ of the form $y^{sm-l}$.

	\medskip

	Let $\omega$ be the monomial valuation at $R$ as in
	Proposition \ref{valoracion_definida}, defined by 
	$w(u_i):=s$, for $i=1,\ldots, d$, and $\omega(y):=r$.
	Now, from (\ref{g_en_R}), $\omega(g(y))\geq rsm$, since $\omega(A_i)\geq rsi$.
	On the other hand, by the definition of $\omega$ (see (\ref{factorizacion_g})),
	$\omega(h)\leq r(sm-l)$.
	Finally, since $\ord_{\langle u\rangle}(a_i)=\ord_{\mathfrak m}(a_i)$,
	$$\omega(f)=\min_i\left\{rl, s\ord_{\mathfrak m}(a_i)-ri+rl, i=1,\ldots, \ell\right\}=\min_i\left\{rl, s  i \frac{\ord_{\mathfrak m}(a_i)}{i}-ri+rl, i=1,\ldots, \ell\right\}.$$
	Taking   $i=i_0$ from (\ref{def_p_q}), if $p/q<r/s$, then
	$$s i_0 \frac{\ord_{\mathfrak m}(a_{i_0})}{i_0}-ri_0+r\ell=r\ell+si_0\left(\frac{p}{q}-\frac{s}{r}\right)<r\ell. $$
	Thus $\omega(f)< r\ell$ which leads to a contradiction because
	$$rsm\leq\omega(g(y))=\omega(h)+\omega(f)<r(sm-\ell)+r\ell=rsm.$$
	Therefore, $\frac{p}{q}\geq \frac{r}{s}$.  Since $\frac{r}{s}$  can be taken arbitrarily large, 
	$$\min_{i}\left\{\frac{\ord_{\mathfrak m}(a_i)}{i}: i=1,\ldots, \ell\right\}$$
cannot be  finite. Thus $a_i=0$ for $i=1,\ldots, \ell$ as claimed. 
\medskip 

In conclusion,  $\nub_{{\mathfrak m}'}(y')=\infty$ if and only if $a_1=\ldots=a_{\ell}=0$ in (\ref{forma_f}) and in this case  $e(R')=\ell$ and formula (\ref{formula_orden}) holds.

\medskip 

Suppose next that $\nub_{{\mathfrak m}'}(y')=\frac{r}{s}<\infty$. Then, 
\begin{equation} \label{def_p_q_new} 
	\frac{p}{q}:=\min_{i}\left\{\frac{\ord_{\mathfrak m}(a_i)}{i}: i=1,\ldots, \ell\right\}=\frac{\ord_{\mathfrak m}(a_{i_0})}{i_0},
\end{equation}
is a finite number and the previous argument already shows that necessarily 
$\frac{r}{s}\leq \frac{p}{q}$. To check that the equality holds we argue as follows. 

\medskip

		Let $v$ be any valuation of $Q(R')$, the quotient ring of $R'$, dominating $(R',{\mathfrak m}')$. Since $f=0$ in $R'$, 
	\begin{equation} \label{desigualdad_valoraciones}
		\ell \cdot v(y')\geq \min_i\{v({a'}_i)+(\ell-i)\cdot v(y')\},
	\end{equation}
	where $a'_i$ denotes the class of $a_i$ in $R'$.
	By \cite[Lemma 10.1.5, Theorem 10.2.2]{Hu_Sw}, \cite[Proposition 2.2]{Irena}, there is a Rees valuation of ${\mathfrak m}'$, $\tilde{v}$, such that 
	$$\nub_{{\mathfrak m}'}(y')=\frac{\tilde{v}(y')}{\tilde{v}({\mathfrak m}')}.$$
	Let $j\in \{1,\ldots,\ell\}$ be so that the minimum in (\ref{desigualdad_valoraciones}) holds for $\tilde{v}$, i.e.,  
	$$\tilde{v}(a'_j)+(\ell-j)\cdot \tilde{v}(y')=\min_i\{\tilde{v}(a'_i)+(\ell-i)\cdot \tilde{v}(y')\}.$$
	Then 
	$$\ell \cdot \tilde{v}(y')\geq \tilde{v}({a}_j')+(\ell-j)\cdot \tilde{v}(y'),$$
	i.e., 
	$$\tilde{v}(y')\geq \frac{\tilde{v}(a'_j)}{j},$$
	and 
	\begin{equation} \label{primera_desigualdad}
		\frac{r}{s}=\nub_{{\mathfrak m}'}(y')=
		\frac{\tilde{v}(y')}{\tilde{v}({\mathfrak m}')}\geq 
		\frac{\tilde{v}({a}'_j)}{j\cdot \tilde{v}({\mathfrak m}')}\geq \frac{\nub_{{\mathfrak m}'}(a'_j)}{j}\geq \frac{\ord_{\mathfrak m}(a_j)}{j}\geq \frac{\ord_{\mathfrak m}(a_{i_0})}{i_0}=\frac{p}{q}.	 
	\end{equation}

	To conclude, since   $p/q=r/s\geq 1$, we have that $\ord_{\mathfrak m}(a_i)\geq i$, for $i=1,\ldots, \ell$, thus $\ord_{\mathfrak m}(f)=\ell$, and therefore    the multiplicity of the local ring $R'$ is $\ell$. 
	
	\medskip 
	
	(iii) Since the hypothesis is that $f\notin \langle u_1,\ldots, u_d\rangle$,
	up to a unit, $f$ can be written as a pseudo-Weierstrass element of degree $\ell$ w.r.t. $(u,y)$: 
	\begin{equation}
		\label{desarrollo_f2}
		f=y^{\ell}+\sum_{i=1}^{\ell}a_i y^{\ell -i}. 
	\end{equation}
	From here, taking a valuation $\tilde{v}$ with   $\nub_{{\mathfrak m}'}(y')=\frac{\tilde{v}(y')}{\tilde{v}({\mathfrak m}')}$ we deduce that 
	\begin{equation}
		\label{primera_desigualdad2}
		1=	\nub_{{\mathfrak m}'}(y')=
		\frac{\tilde{v}(y')}{\tilde{v}({\mathfrak m}')}\geq 
		\frac{\tilde{v}({a}'_j)}{j\cdot \tilde{v}({\mathfrak m}')}\geq \frac{\nub_{{\mathfrak m}'}(a'_j)}{j}\geq \frac{\ord_{\mathfrak m}(a_j)}{j}\geq \frac{\ord_{\mathfrak m}(a_{i_0})}{i_0}=\frac{p}{q}.	 
	\end{equation}
	If $\langle u_1',\ldots, u_d'\rangle$ is a reduction of ${\mathfrak m}'$, the argument in (ii) gives us the equality (\ref{formula_orden}).
	
	Conversely, if equality (\ref{formula_orden}) holds  then
	$\ord_{\langle u\rangle}(a_i)=\ord_{\mathfrak{m}}(a_i)\geq i$ and $f=0$ gives an
	equation of integral dependence for $y'$ over $\langle u_1',\ldots, u_d'\rangle$.
\end{proof}

\begin{theorem}
\label{General_General_Hickel}
Let $(R',{\mathfrak m}',k)$ be a local ring. Suppose there is an excellent subring $S$ of $R'$  such that the extension $S\subset R'$ is finite-transversal. Let $\theta\in R'$. If 
$$p(z)=z^{\ell}+a_1z^{\ell-1}+\ldots+a_{\ell}$$
is the minimal polynomial of $\theta\in R'$ over the fraction field of $S$, $K(S)$, then 
$p(z)\in S[z]$ and 
\begin{equation}
	\label{EqHandbookHickel}
	\nub_{\mathfrak{m}}(\theta)=\min_i \left\{\frac{\nu_{{\mathfrak m}_S} (a_i)}{i}: i=1,\ldots,\ell\right\}.
\end{equation}
\end{theorem} 

\begin{proof} When $R'$ is equicharacteristic this is a straight generalization of Hickel's Theorem  (see 
\cite[Theorem 11.6.8]{Handbook}).
For the general case, consider the sequence of finite transversal extensions,
$$S\longrightarrow S[\theta] \longrightarrow R'.$$
By   \cite[Lemma 5.2]{V}, $p(z)\in S[z]$ and $S[\theta]\simeq S[z]/\langle p(z)\rangle$. Let ${\mathfrak m}_{\theta}:={\mathfrak m'}\cap  S[\theta]$  be the (unique) maximal ideal of $S[\theta]$ and let ${\mathfrak m}_S\subset S$ be the maximal ideal of $S$.   Let ${\mathfrak m}_1\subset S[z]$ be the (unique) maximal ideal of $S[z]$ mapping to ${\mathfrak m}_{\theta}$. Applying    Theorem \ref{GeneralHickel}  to  the quotient $S[z]/\langle p(z)\rangle$, we find that 
$$\nub_{{\mathfrak m}_{\theta}}(\theta)= \min_i \left\{\frac{\ord_{{\mathfrak m}_1} (a_i)}{i}: i=1,\ldots,\ell\right\}.$$
Since $S\to S[z]_{{\mathfrak m}_1}$ is faithfully flat, $\ord_{{\mathfrak m}_1} (a_i)=\ord_{{\mathfrak m}_S} (a_i)$, for $i=1,\ldots,\ell$. 
Now, the result follows because,   by \S \ref{transversal_equiv} (iii), the ideal ${\mathfrak m}_S$ generates a reduction of both,  ${\mathfrak m}_{\theta}$ in $S[\theta]$ and ${\mathfrak m}'$ in $R'$, and  therefore, $\nub_{{\mathfrak m}_{\theta}}(\theta)=\nub_{{\mathfrak m}'} (\theta)$.
\end{proof}

\begin{corollary} \label{maximo_hipersuperficies}
Let $(R',{\mathfrak m}',k)$ be a local ring. Suppose there is an excellent subring $S$ of $R'$  such that the extension $S\subset R'$ is   finite-transversal and write $R'=S[\theta_1,\ldots,\theta_e]$.  Then: 
$$\SSl(R')=\min\{\SSl(S[\theta_i]): i=1,\ldots,e\}.$$
	\end{corollary}
\begin{proof} 
As in the proof of the previous theorem, $S{[\theta_i]}\simeq S[z_i]/\langle p_i(z_i)\rangle$, where $p_i(z_i)\in S[z_i]$ is the minimal polynomial of $\theta_i$ over $K(S)$. Set ${\mathfrak  m}_{\theta_i}:={\mathfrak m}'\cap S[\theta_i]$. By Lemma \ref{lema_lema_v4}  and Remark \ref{lemma_lemma_remark}, $$\SSl(S[\theta_i])=\sup\{\nub_{{\mathfrak m}_{\theta_i}}(\theta_i-s): s\in S\}.$$
Also, by Lemma \ref{lema_lema_v4} and Remark \ref{lemma_lemma_remark}, 
$$\SSl(R)=\sup\{\nub_{{\mathfrak m}'}(\theta_i-s): s\in S\}.$$
Since $\nub_{{\mathfrak m}'} (\theta_i-s)=\nub_{{\mathfrak m}_{\theta_i}} (\theta_i-s)$, the claim follows. 
\end{proof}

To conclude this part, we emphasize that when $S\to R'$ is finite transversal, then the restriction of the asymptotic Samuel function to the regular ring $S$ is the usual order at the maximal ideal, and therefore a valuation.
In the absence of a finite transversal inclusion $S\subset R'$, we can reinterpret this property  in terms of CP-expansions. 

\begin{definition}
	Let $(R,\mathfrak{m})$ be a regular ring of dimension $n$. Let $d\leq n$ and let
	$(u_1,\ldots,u_d)$ part of a regular system of parameters of $R$.
	We say that an element $g\in R$ has a \emph{$(u)$-CP-expansion} if
	there is a regular system of parameters $u_1,\ldots,u_d,y_1,\ldots,y_{n-d}$ such that
	for the CP-expansion w.r.t. $(u,y)$ we have
	$S(g)\subset\mathbb{N}^d\times\{0\}$, i.e.,   
	$$g=\sum_{(\alpha,0)\in S(g)}c_{(\alpha,0)}(g)u^{\alpha},$$
	for some units $c_{(\alpha,0)}(g)\in R$.
\end{definition}
Note that this definition only depends on $u_1,\ldots,u_d$ and it does not depend on the choice of the $y$'s.

\begin{example} The fact that an element $g\in R$ has a $(u)$-CP-expansion does not imply that all powers of $g$ have a  $(u)$-CP-expansion. Consider for instance the ring   $R=k[u_1,u_2,y]_{\langle u_1,u_2,y\rangle}$, where $k$ is a field of
characteristic different from 2, and let 
  $g=u_1^2+u_1u_2(1+y)-\dfrac{1}{2}u_2^2$.
Here $S(g)=\{u_1^2, u_1u_2, u_2^2\}$ and $g$ has a $(u)$-CP-expansion.
But $g^2$ does not have a $(u)$-CP-expansion
$$g^2=u_1^4+2(1+y)u_1^3u_2+(2+y)u_1^2u_2^2y-(1+y)u_1u_2^3+\dfrac{1}{4}u_2^4,$$
in fact $S(g^2)=\{u_1^4, u_1^3u_2,u_1^2u_2^2y, u_1^3u_2, u_2^4\}$.
\end{example}
However, as we will see in the next proposition,  the powers of $g$ do  have a good behavior for the order with respect to the ideal $\langle u \rangle$.

\begin{proposition}
\label{orden_reduction}
Let $(R,\mathfrak{m})$ be a regular local ring,  let $(u_1,\ldots, u_d)$ be part of a regular system of parameters in $R$ and let ${\mathfrak a}\subset R$ be an ideal.  Suppose that $R'=R/\mathfrak{a}$ is a $d$-dimensional ring whose maximal ideal has a reduction generated by the classes  $u'_1,\ldots,u'_d\in R'$.
Let $g\in R$ be an element that has a $(u)$-CP-expansion, and write 
$$g=\sum_{(\alpha,0)\in S(g)}c_{(\alpha,0)}(g)u^{\alpha},$$
for some units $c_{(\alpha,0)}(g)\in R$.
Then if $g'\in R'$ is the class of $g$, 
$$\nub_{\mathfrak{m}'}(g')=\ord_{\mathfrak{m}}(g)=
\min\{|\alpha| : (\alpha,0)\in S(g)\}.$$
\end{proposition}

\begin{proof}
Set $b=\ord_{\mathfrak{m}}(g)$. If $b=0$ there is nothing to prove, thus we can assume that $b>0$. In addition, we can also 
 assume that $R$ has infinite residue field: if the residue field of $R$ is finite, then consider the faithfully flat extension $R'\to R'_1=R'[t]_{\mathfrak{m}'[t]}$,  and use  Proposition \ref{PotenciaNuBar} (ix) and the fact that  the CP-expansion 
 $g$   is the same in both rings.

Since $\ord_{\mathfrak{m}}(g)=b\geq 1$,   $g'\in\mathfrak{m}'$ and   $\nub_{\mathfrak{m}'}(g')\geq b$. Suppose  that $\nub_{\mathfrak{m}'}(g')>b$.  Then, there is an integer $s\geq 1$ such that
$(g')^s\in(\mathfrak{m}')^{bs+1}$.
Set
$$G(U_1,\ldots,U_d)=
\sum_{\substack{{\beta\in S(g)}\\ {|\beta|=b}}}c_{(\beta,0)}(g)'U^{\beta}
\in R'[U],
\quad \text{and} \quad F(U)=(G(U))^s.$$
Observe that  $G(U)$ is  the class in $R'[U]$  of the sum of terms of degree
$b$ in the CP-expansion of $g$. Note that $F(U)$ is a  homogeneous polynomial of degree $bs$ and that 
$F(u'_1,\ldots,u'_d)\in(\mathfrak{m}')^{bs+1}$.
Since $u'_1,\ldots,u'_d$ are analytically independent in $R'$ then all the coefficients of $F(U)$ are in $\mathfrak{m}'$.
	
Fix any monomial ordering $\prec$ in $\mathbb{N}^n$ such that
$|\gamma_1|<|\gamma_2|$ implies $\gamma_1\prec\gamma_2$.
Let $(\alpha_0,0)=\min S(g)$. As in the proof of Lemma \ref{LemaCPproduct}
the term $c_{(\alpha_0,0)}(g)^s$ is the only one contributing to $u^{s\alpha_0}$ in
$g^s$. We conclude that the coefficient of $U^{s\alpha_0}$ in $F(U)$ is
$(c_{(\alpha_0,0)}(g)')^s$ and this is a
contradiction.
\end{proof}

\begin{remark}
\label{RemLemaLema}
Note that in the proof of Lemma \ref{lema_lema_v4}, in case 2 or 3, the chosen $s_i$
have $(u)$-CP-expansions.
\end{remark}

\section{The functions $\Hord_X$ and $\ord_X$}\label{Rees_Algebras} 

In the following lines we present a brief overview on the way the functions $\Hord_X^{(d)}$ and $\ord_X^{(d)}$ are defined for an equidimensional variety $X$ of dimension $d$ over a perfect field.   To this end we need to introduce some notions on Rees algebras and the way they are use in resolution.  
We refer to \cite{V3} and \cite{E_V} for further details.

\begin{definition}
	Let $R$ be a Noetherian ring. A \textit{Rees algebra $\mathcal{G}$ over $R$} is a finitely generated graded $R$-algebra, $\mathcal{G}=\bigoplus _{l\in \mathbb{N}}I_{l}W^l\subset A[W]$, 
	for some ideals $I_l\in R$, $l\in \mathbb{N}$ such that $I_0=R$ and $I_lI_j\subset I_{l+j}\mbox{,\; }$ for all $ l,j\in \mathbb{N}$. Here, $W$ is just a variable to keep track of the degree of the ideals $I_l$. Since $\mathcal{G}$ is finitely generated, there exist some $f_1,\ldots ,f_r\in  R$ and positive integers (weights) $n_1,\ldots ,n_r\in \mathbb{N}$ such that
	$\mathcal{G}=R[f_1W^{n_1},\ldots ,f_rW^{n_r}]$. The previous definition extends to Noetherian schemes in the obvious manner.
\end{definition}

Suppose now that $R$ is  essentially of finite type and smooth over a perfect field $k$, and  let     $\mathcal{G}=\oplus _{l\geq 0}I_lW^l$  be  a Rees algebra  over  $R$. Then, 
the  \textit{singular locus} of $\mathcal{G}$, Sing$(\mathcal{G})$,  is the closed set  in $\Spec(R)$ given by all the points $\zeta \in \Spec(R)$ such that
$\ord_{\zeta }(I_l)\geq l$, $\forall l\in \mathbb{N}$,
where $\ord_{\zeta}(I)$ denotes the order of the ideal $I$ in the regular local ring $\mathcal{O}_{V,\zeta }$. The {\em order of the Rees algebra at a point    $\zeta\in \Sing(\mathcal{G})$} is: 
$$\ord_{\zeta}(\G):=\inf \left\{\frac{\ord_{\zeta}(I_n)}{n}: n\geq 1\right\}.$$

If  $\mathcal{G}=R[f_1W^{n_1},\ldots ,f_rW^{n_r}]$, then  $
\Sing(\mathcal{G})=\left\{ \zeta \in \mathrm{Spec}(R)\ |\, \ord_{\zeta }(f_i)\geq n_i,\;  i=1,\ldots ,r\right\} \subset V$ and $\ord_{\zeta}(\G):=\min \left\{\frac{\ord_{\zeta}(f_i)}{n_i}: i=1,\ldots,r\right\}$   (see \cite[Proposition 1.4]{E_V}).

\begin{example} \label{Ex:Multiplicity}
	Suppose that
	$X\subset\Spec(R)=V$ is  a hypersurface with $I(X)=\langle f\rangle$.  Let $m>1$ be the maximum   multiplicity  at the points of $X$. Then the singular locus of 
	$\mathcal{G}=R[fW^m]$   is the set of points of $X$ having maximum multiplicity $m$, $\Mm$. This idea can be generalized for arbitrary algebraic varieties and we will go back to it in forthcoming paragraphs (see Theorem \ref{presentaciones_mult} below).   
\end{example}

In the previous example, the link between $\G$ and the closet set $\Mm$ is stronger than the 
 equality $\Sing(\G)= \Mm$.  To start with,  
by defining  a suitable law of transformations of Rees algebras  after  a {\em permissible  blow up}, we can establish the same  link  between the  top multiplicity locus  of  the strict transform of $X$,  and the singular locus of the {\em transform} of $\G$ (at least if  the maximum multiplicity  of the transform of  $X$ does not drop). This link is also  preserved   under smooth morphisms. In general, given an arbitrary equidimensional algebraic variety $X$ embedded in a smooth scheme $V$ over a field $k$,  and a sheaf of Rees algebras $\G$ over $V$, we say that    $\G$ {\em represents} the top multiplicity locus of $X$ is $\Sing(\G)=\Mm$ and this equality is preserved after smooth morphisms and permissible blow ups (if the maximum multiplicity of $X$ does not drop), see \cite{E_V} and \cite{V}.

\begin{parrafo} \label{integral_diff} {\bf Uniqueness of the representations of the multiplicity.} The Rees algebra of Example \ref{Ex:Multiplicity} is not the unique representing $\Mm$.
	In the following lines  we consider two operations on Rees algebras that preserve this property: 
	
	{\bf (i) Rees algebras and integral closure.} Two Rees algebras over a (not necessarily regular) Noetherian ring $R$  are \textit{integrally equivalent} if their integral closure in $\mathrm{Quot}(R)[W]$ coincide.
	We use $\overline{\mathcal{G}}$ for the integral closure of $\mathcal{G}$, which can be shown to also be a Rees algebra over $R$ (\cite[\S 1.1]{Br_G-E_V}).  
	
	{\bf (ii) Rees algebras and  saturation by differential operators.}   Let $\beta^*: S\to R$ be a smooth morphism of smooth algebras that are essentially of finite type  over a perfect field $k$ with $\dim_{\text{Krull}} S < \dim_{\text{Krull}}  R$, inducing a morphism $\beta: \Spec(R)\to \Spec(S)$. Then, for any integer $\ell$, the ${R}$-module  of relative differential operators of order at most $\ell$,   $\Diff_{{R}/S}^{\ell}$,
	is locally free    (\cite[(4)~\S~16.11]{EGA_4}). We will say that an $R$-Rees algebras ${\mathcal G}=\oplus_l I_lW^l$ is a
	{\em $\beta$-differential Rees algebra} if   for  every homogeneous element $fW^N\in \G$ and every $\Delta\in \Diff_{R/S}^{\ell}$ with $\ell<N$,  we have that $\Delta(f)W^{N-\ell}\in {\mathcal G}$ (in particular, $I_{i+1}\subset I_i$  since  $ \Diff_{R/S}^0\subset \Diff_{R/S}^1$).
	Given an arbitrary Rees algebra $\G$ over $R$  there is a natural way to construct a $\beta$-relative differential algebra with the property of being the smallest containing $\G$,  and we will denote it by $\Diff_{R/S}(\G)$   (see \cite[Theorem 2.7]{Villa2007}).    
	
	We say that $\mathcal{G}$ is \textit{differentially closed} if it is closed by the action of the sheaf of (absolute) differential operators  $\Diff_{R/k}$. We use  $\mathrm{Diff}(\mathcal{G})$ to denote the smallest differential Rees algebra containing $\mathcal{G}$ (its \textit{differential closure}). See \cite[Theorem 3.4]{Villa2007} for the existence and construction. 
	
	It can be shown that
	$\Sing(\G) =\Sing (\overline{\G})= \Sing (\mathrm{Diff}(\mathcal{G})),$
	(see \cite[Proposition 4.4 (1), (3)]{V3}), and that 
	for $\zeta\in \Sing(G)$, $\ord_{\zeta}(\G)=\ord_{\zeta}(\overline{\G})=\ord_{\zeta}(\mathrm{Diff}(\mathcal{G}))$ (see \cite[Remark 3.5, Proposition 6.4
	(2)]{E_V}). 
	 In addition, it can be checked that if $\G$ represents $\Mm$ for al algebraic variety $X$, then the Rees algebras  $\overline{\G}$ and $\Diff(\G)$ also represent $\Mm$. In fact, it can be shown that 
	the integral closure of $\mathrm{Diff}(\mathcal{G})$  is the largest algebra in $V$ with this property. 
	The previous discussion motivates the following definition: two Rees algebras on $V$, $\G$ and $\H$, are said to be {\em weakly equivalent}  if  $\overline{\mathrm{Diff}(\mathcal{G})}=\overline{\mathrm{Diff}(\mathcal{H})}$ (see \cite{Br_G-E_V} and \cite{HirThree}). 
\end{parrafo}

\begin{parrafo} \label{Emininacion_Hiper} {\bf Elimination algebras: the hypersurface case.}
	Let $S$ be a  smooth  $d$-dimensional algebra essentially  of finite type over a perfect field $k$,  with $d>0$.  Let $z$ be a variable. Then  the  natural inclusion $S {\longrightarrow } R=S[z]$ induces a   smooth projection $\beta: \Spec(R)\to \Spec(S)$. Let $f(z)\in S[z]$ be a polynomial of degree $m>1$,
	 \begin{equation}
	 	\label{expresion_f_1} 
	f(z)=z^m+a_1z^{m-1}+\ldots+a_m, \ \ a_i\in S, \ i=1,\ldots m,
	 \end{equation} 
	defining a hypersurface $X$  in $V^{(d+1)}:=\Spec(R)$. Set $X=\mathrm{Spec}(S[z])/\langle f(z) \rangle$. Suppose that  $\xi \in X$ is  a point  of multiplicity $m$.  
	Then, the $R$-algebra, 
	$$\mathcal{G}^{(d+1)}=\mathrm{Diff}(S[z][fW^{m}])\subset S[z][W]$$ represents the multiplicity function on $X$ locally at $\xi$. The {\em elimination algebra of $\mathcal{G}^{(d+1)}$} is the $S$-Rees algebra:   
		$$\G^{(d)}:=\G\cap S[W].$$
	There are other ways to define an  elimination algebra for $\G^{(d+1)}$, and it can be checked that all them lead to the same $S$-Rees algebra up to weak equivalence (see \cite[Theorem 4.11]{Villa2007}). In particular,  it can be shown that the elimination algebra is generated by a finite set of   symmetric (weighted homogeneous) functions evaluated   on the  coefficients of $f(z)$ (cf.  \cite{V00}, \cite[\S1, Definition 4.10]{Villa2007}).  It is worthwhile noticing that  the elimination algebra $\Gd$ is invariant under changes of the form $z_1=z+s$ with $s\in S$ (see \cite[\S 1.5]{Villa2007}).

	If the characteristic is zero, or if $m$ is not a multiple of the characteristic,  then we can  assume that  $f$ has the form of Tschirnhausen (there is always a change of coordinates that leads   to this form):
	\begin{equation}\label{ec:f_Tsch}
		f(z)=z^m+a_{2}z^{m-2}+\ldots + a_{m-i}z^i+\ldots +a_m\in S[z]\mbox{,}
	\end{equation}
	where $a_i\in S$ for $i=0,\ldots ,m-2$. In this case, 
	\begin{equation}
		\label{elimination_char_zero} 
	\mathcal{G}^{(d)}=\mathrm{Diff}(S[z][a_{2}W^2,\ldots ,a_{m-i}W^{m-i},\ldots ,a_mW^m])\mbox{,}
	\end{equation}
	is an elimination algebra of $\G^{(d+1)}$.  It can be shown that, in this setting, $\G^{(d)}$ also represents the (homeomorphic image via $\beta$ of) the top multiplicity locus of $X$, now in a smooth ambient space of lower dimension.

\end{parrafo}

 When $m$ is a multiple of the characteristic, $\G^{(d)}$ falls short to provide all the information codified in the coefficients of $f$. In this setting the first author jointly with O. Villamayor introduced another function: the function $\Hord_X$. 

\begin{parrafo} {\bf The $\Hord_X$-function in the hypersurface case.} \cite[\S 5.5]{BVIndiana} \label{DefSlope1} 
Consider the setting and the notation in \S \ref{Emininacion_Hiper}. The data   given by the smooth projection $\beta$,   a selection of a section of $\beta$, $z$, and the polynomial $f(z)$, is denoted by    $\mathcal{P}(\beta,z,fW^m)$. Then, 
	the \emph{slope  of $\mathcal{P}(\beta,z,fW^m)$ at a point $\zeta\in\Sing(\mathcal{G}^{(d+1)})\subset V^{(d+1)}$} is defined as: 
	\begin{equation} \label{EqDefSlope1}
		Sl(\mathcal{P})(\zeta):=\min\left\{
		\nu_{\beta(\zeta)}(a_1),\ldots,\frac{\nu_{\beta(\zeta)}(a_j)}{j},\ldots,\frac{\nu_{\beta(\zeta)}(a_m)}{m}, \ \ord_{\beta(\zeta)}(\mathcal{G}^{(d)})
		\right\}.
	\end{equation}

	The value $Sl(\mathcal{P})(\zeta)$ changes after 
	translations of the form $z+s$, with $s\in S$, given new values $Sl(\mathcal{P})(\zeta+s)$. 
 	Thus we consider, 	
	\begin{equation} \label{EqBetaOrd}
		\sup_{s\in S}\left\{Sl(\mathcal{P}(\beta,z+s,fW^m))(\zeta) \right\}
	\end{equation}
	which is an invariant of the point. Moreover, the supremum in (\ref{EqBetaOrd}) is a maximum for a suitable selection of $s\in S$. See \cite[\S 5.2 and Theorem 7.2]{BVComp}.
	
	To conclude, the previous discussion applies to a singular point in  $\zeta\in X$ in  a hypersurface $X$ of multiplicity $m>1$ in the following way: 
	 consider an \'etale neighborhood
	$X'\to X$ of a closed point of multiplicity $m$, $\xi\in\overline{\{\zeta\}}$,
	such that the setting of \S\ref{Emininacion_Hiper}  holds,
	and let $\zeta'\in X'$ be a point mapping to $\zeta$.
	Then  define
	\begin{equation*}
		\Hord^{(d)}_X(\zeta):=\Hord^{(d)}_{X'}(\zeta'):=\max_{s\in S}\left\{Sl(\mathcal{P}(\beta,z+s,gW^{N}))(\zeta') \right\}.
	\end{equation*} 
See 	\cite[\S 5, Definition 5.12]{BVIndiana}.  
\end{parrafo}

To conclude,  to understand and define  elimination algebras and the $\Hord_X$ function for arbitrary algebraic varieties,  it suffices to treat the hypersurface case, as stated in the following result of Villamayor (see also Remark \ref{caso_general} below):

\begin{theorem}\label{presentaciones_mult} \cite[Lemma 5.2, \S6, Theorem 6.8]{V}  (Presentations for the Multiplicity  function)
	Let $X=\text{Spec}(B)$ be an affine equidimensional algebraic variety of dimension $d$  defined over a perfect field $k$, and let $\xi\in  \underline{\text{Max}}\text{\,Mult}_X$ be a closed  point of   multiplicity $m>1$. Then,  there is an \'etale neighborhood $B'$ of $B$,   mapping $\xi'\in \Spec(B')$ to $\xi$, so that there is   a smooth $k$-algebra $S$ together with a
	finite  morphism     $\alpha: \Spec(B')\to \Spec(S)$ of generic rank $m$, i.e., if $K(S)$ is the quotient field of $S$, then $[K(S)\otimes_S B: K(S)]=m$. Write $B'=S[\theta_1,\ldots, \theta_e]$.  Then:  
	\begin{itemize}
		\item[(i)] If $f_i(z_i)\in K(S)[z_i]$ denotes the minimum polynomial of $\theta_i$ over $K(S)$ for $i=1,\ldots,e$, then $f_i(z_i)\in S[z_i]$  and there is a commutative diagram: 
		\begin{equation}
			\begin{aligned}
				\label{diagrama_presentacion}
				\xymatrix{R=S[z_1,\ldots, z_e] \ar[r] & S[z_1,\ldots, z_e]/\langle f_1(z_1),\ldots, f_e(z_e)\rangle \ar[r]   & B'  \\
					& S \ar[u] \ar[ur]_{\alpha^*} \ar[ul]^{\beta^*} &  }
			\end{aligned}
		\end{equation}

		\item[(ii)]  Let $V^{(d+e)}=\Spec(R)$, and let ${\mathcal I}(X')$ be the defining ideal of $X'$ at $V^{(d+e)}$. Then 
		$$\langle f_1,\ldots, f_e\rangle \subset {\mathcal I}(X');$$ 
		\item[(iii)] Denoting by $m_i$ the maximum order of the hypersurface $H_i=\{f_i=0\}\subset V^{(d+e)}$,  
		the differential Rees algebra
		\begin{equation}
			\label{G_Representa}
			\mathcal{G}^{(d+e)}=\Diff(R[f_1(z_1)W^{m_1}, \ldots, f_{e}(z_e)W^{m_{e}}])
		\end{equation}
		represents  the top  multiplicity locus of $X$, $\underline{\text{Max}}\text{\,Mult}_X$,  at   $\xi$ in $V^{(d+e)}$.		
	\end{itemize}
\end{theorem}

\begin{remark} \label{caso_general}
	For $\mathcal{G}^{(d+e)}$ as in (\ref{G_Representa}) an elimination algebra over $S$, $\Gd$, is also defined. There are different ways to give a definition   but it can be shown that, up to weak equivalence, 
	$$\Gd:=\Diff(\G^{(d+e)})\cap S[W],$$
see \cite[Theorem 4.11]{Villa2007} and \cite{Br_V}. Let $\G_i^{(d+1)}=\Diff(S[z_i][f_{z_i}W^{m_i}])$ be the Rees algebra representing the top multiplicity locus of $X_i:={\mathbb V}(\langle f_i(z_i)\rangle)\subset \Spec(S[z_i])$, and let  $\G_i^{(d)}\subset S[W]$ be the corresponding elimination algebra. Then, from the description in (\ref{G_Representa}),  it can be shown that,   up to weak equivalence, 
\begin{equation}
	\label{amalgama_eliminacion}
	\mathcal{G}^{(d)}=\G_1^{(d)}\odot \cdots \odot \G_e^{(d)}.	
\end{equation}

As it happens,  in the hypersurface case, if the characteristic is zero,  $\G^{(d)}$ also represents the (homeomorphic image via $\beta$ of) the top multiplicity locus of $X'=\Spec(B')$,  in a smooth ambient space of the same dimension as $X'$.  In positive characteristic,   the function $\Hord_{X'}$ can also be defined.  Let $X_i$ be the hypersurface defined by $f_i(z_i)\in S[z_i]$.
Then, 
$$\Hord^{(d)}_{X'}:=\min_{i=1\ldots,e}\Hord^{(de)}_{X_i},$$
see \cite[Definition 6.6]{BVIndiana}.  
\end{remark}

\begin{parrafo} \label{DefOrdHironaka}
	\textbf{$\Hord_X^{(d)}$ of an algebraic variety.} 
	Let $X$ be an equidimensional variety of dimension $d$ over a perfect field $k$ and let $\zeta\in X$ be a point of maximum  multiplicity $m>1$. We can assume that $X=\Spec(B)$ is affine. 
	Let $\xi\in\overline{\{\zeta\}}$ be a closed of multiplicity $m$. 
	Then, as indicated in Theorem  \ref{presentaciones_mult}, there is an 
	\'etale neighborhood of $\Spec(B)$, $X'=\Spec(B')$,   an embedding in some  smooth $(d+e)$-dimensional scheme $V^{(d+e)}$,
	and a differential Rees algebra $\G^{(d+e)}$ representing the top multiplicity locus of $X'$, and   there is a $\G^{(d+e)}$-admissible projection to some $d$-dimensional smooth  scheme where an elimination algebra $\Gd$ can be defined. Let $\zeta'\in X'$ be a point mapping to $\zeta$.  Then: 
	$$\Hord^{(d)}_X(\zeta):=\Hord^{(d)}_{\G^{(d+e)}}({\zeta'}).$$
	does not depend on the selection of the \'etale neighborhood, nor on the choice of  Rees algebra representing the top multiplicity locus,  nor on   the  admissible projection (see \cite[\S 5.2, Theorem 7.2]{BVComp}). 
\end{parrafo}

\section{Samuel slope, $\ord_X^{(d)}$ and $\Hord_X^{(d)}$}
\label{SecSSl_Hord}

Let $X$ be an equidimensional algebraic variety of dimension $d$ defined over a perfect field $k$, and let $\zeta\in X$ be a point of maximum multiplicity $m>1$.  The goal in this section is to establish a connection among  the values  $\ord_X^{(d)}(\zeta)$, $\Hord_X^{(d)}(\zeta)$,    and  $\SSl({\mathcal O}_{X,\zeta})$.

As indicated in section \ref{Rees_Algebras}, the computation of the functions  $\ord_X^{(d)}$ and $\Hord_X^{(d)}$  somehow reduces to the hypersurface case. Recall that, after considering a convenient local \'etale neighborhood of a   closed point of a  hypersurface, Weirerstrass Preparation Theorem can be applied,  and we can be assume to be  working in a setting as the one described in \S \ref{HyperLocal} below. Although this is our original motivation,  part of the arguments in the  present section  do not need the presence of a ground field, nor   extra assumptions on the rings involved. Thus,  we will  work with arbitrary commutative rings with unit and will impose additional conditions when needed.

\begin{parrafo}
Let $S$ be a commutative ring witn unit,  and let $z$ be a variable. Then the natural inclusion $S\subset S[z]$ is smooth, and the $S[z]$-module of relative differential operators, $\Diff_{S[z]/S}$ is generated by the Taylor operators. More precisely, if $T$ is another variable, then there is a morphism of $S$-algebras, $$\begin{array}{rrcl} 
		\text{Tay} : & S[z] & \to &  S[z,T] \\
		& f(z) & \mapsto & f(z+T)=\sum\Delta^i_z(f(z))T^i, 
	\end{array}$$
	where  the operators $\Delta^i_z: S[z]\to S[z]$ are induced by $\text{Tay}$. The set $\{\Delta^0_z,\ldots, \Delta^r_z\}$ is a basis of the free $S[z]$-module of $\Diff_{S[z]/S}^{\leq r}$, the $S$-differential operators of $S[z]$ of order $\leq r$, $r\geq 0$. If we replace $z$ by $t=z+s$, with $s\in S$, then $S[z]=S[t]$ and $\{\Delta^0_t,\ldots, \Delta^r_t\}$ is also an $S[z]$-basis of $\Diff_{S[z]/S}^{r}$.

	For a polynomial of degree $m$,
	$$h(z)=c_0z^m+c_{1}z^{m-1}+\ldots+c_{m-1}z+c_m\in S[z],$$
	we define, for every $r\geq 0$,  the ideal  
	\begin{equation}\label{derivadas_relativas}
		\Delta_{z}^{(r)}(\langle h(z)\rangle)=\langle \Delta_z^{i}(h(z)): i=0,\ldots, r\rangle=\Delta_t^{(r)}(\langle h(t-s)\rangle)=\langle \Delta_t^{i}(h(t-s)): i=0,\ldots, r\rangle.	 
	\end{equation}
	
\end{parrafo}

\begin{parrafo} \label{HyperLocal} Assume  now that    $(S, {\mathfrak m}_S)$ is regular local ring  of Krull dimension $d$,  and suppose that  
	\begin{equation}\label{forma_f}
		f=z^m+a_1z^{m-1}+\cdots+a_{m-1}z+a_m, \qquad a_i\in S, \  i=1,2,\ldots,m,
	\end{equation}
	is an element of order $m$ at some maximal ideal ${\mathfrak m}\subset S[z]$, dominating ${\mathfrak m}_S$,  i.e., $\ord_{\mathfrak m}(f)=m$ and ${\mathfrak m}\cap S={\mathfrak m}_S$. Let $R=S[z]_{\mathfrak m}$.

Set $R'=S[z] /\langle f\rangle$. Then the generic rank of the extension $S\to R'$ is $m$, and the multiplicity of $R'$ at $\mathfrak{m}'=\mathfrak{m}/\langle f\rangle$ is $m$,   i.e.,  the extension $S\to R'$ is finite transversal   and hence (i)-(iii) from Remark \ref{transversal_equiv} hold.

Let $z'\in R'$ be the class of $z$ in $R'$. From Remark \ref{transversal_equiv} (ii),  it follows that  there is some $s\in S$ such that   $z'+s\in {\mathfrak m}'$. Thus, after a translation, we can assume that $z\in {\mathfrak m}$, $z'\in {\mathfrak m'}$
and $\mathfrak{m}=\langle z\rangle+\mathfrak{m}_SR$. Therefore $\nub_{\mathfrak{m}'}(z')\geq	1$.
Since $z\in R$ is a regular parameter, then $R/\langle z\rangle$ is also regular and
the order of the class of $f$ in $R/\langle z\rangle$ is greater than or equal to $m$.
Therefore $\ord_{\mathfrak{m}}(a_m)\geq m$.
We have that $\ord_{\mathfrak{m}}(z^{m-1}+a_1z^{m-2}+\cdots+a_{m-1})\geq m-1$ and
iterating the argument we find that
	\begin{equation}
		\label{orden_coeficientes}
		\ord_{{\mathfrak m}_S}(a_i)\geq i,   \   \  \text{ for } i=1,\ldots, m,
	\end{equation}	
and, under these hypotheses, by Theorem \ref{GeneralHickel}, and using the expression in (\ref{forma_f}),
	\begin{equation}
		\label{nubarra_theta}
		1\leq 	\nub_{\mathfrak{m}'}(z')=\min_i\left\{\frac{\ord_{\mathfrak{m}}(a_i)}{i}\right\} =\min_i\left\{\frac{\ord_{\mathfrak{m}_S}(a_i)}{i}\right\},
	\end{equation}
	hence,
	\begin{equation}
		\label{orden_pasado_coef}
		\nub_{{\mathfrak m}'}(z')\leq \ord_{{\mathfrak m}_S} (a_1), \  \   \text{  and } \  \   \nub_{{\mathfrak m}'}(z') <\ord_{{\mathfrak m}_S} (a_i)  \   \  \text{ for } i=2,\ldots, m.
	\end{equation} 
	From (\ref{orden_coeficientes}) it also follows that 
	${\mathfrak m}'={\mathfrak m}_S+\langle z'\rangle$. 
	Next, note that for $i=0,\ldots,m-1$,
	\begin{equation}
		\label{expresion_derivadas}
		\Delta_z^i(f(z))=\binom{m}{i}z^{m-i}+
		\sum_{j=1}^{m-i}\binom{m-j}{i}a_jz^{m-j-i},
	\end{equation}	
	hence, 
	\begin{equation} \label{EqOrdDelta}
		\ord_{\mathfrak{m}}\left(\Delta^{i}(f(z))\right)\geq m-i,
		\qquad \text{ for } i=1,\ldots, m-1.
	\end{equation}
	
	If the characteristic of   $S$ is zero,      
	$\ord_{\mathfrak{m}}(\Delta^{(i)})(f(z))=m-i$, for $i=0,\ldots, m-1$, 
	thus, in particular, 
	$\ord_{\mathfrak{m}}\left(\Delta^{(m-1)}(f(z))\right)=1$, and there is some element of order one in  $\Delta^{(m-1)}(f(z))$. 
\end{parrafo}

The following proposition is inspired in a computation by V. Cossart (see \cite{CossartBulletin}). 

\begin{proposition} \label{orden_1_ideal_derivadas}
	With the hypothesis and notation of \S   \ref{HyperLocal}, 
	suppose that there exists some element $\omega \in \Delta_z^{(m-1)}(\langle f(z)\rangle)$   with $\ord_{\mathfrak{m}}(\omega)=1$.
	If ${\omega}'$ denotes the class of $\omega$ in $R'$ then:
	\begin{enumerate}
		\item[(i)] The element $\omega'$ is part of a minimal set of generators of ${\mathfrak m'}$; 
		\item[(ii)] If $\ker(\lambda_{\mathfrak{m}'})=(0)$, then $\nub_{\mathfrak{m}'}({\omega}')=\nub_{\mathfrak{m}'}(z')=1$;
		\item[(iii)] $\nub_{\mathfrak{m}'}({\omega}')\geq \nub_{\mathfrak{m}'}(z')$;
		\item[(iv)] If $\ker(\lambda_{\mathfrak{m}'})=\langle \In_{\mathfrak{m}'}(z')\rangle$  then  $\ker(\lambda_{\mathfrak{m}'}\rangle=\langle \In_{\mathfrak{m}'}({\omega}')\rangle$.
		
	\end{enumerate}
\end{proposition}

\begin{proof} Note first  that $\nub_{\mathfrak{m}'}({\omega}')\geq
	\ord_{\mathfrak{m}}(\omega)=1$. Now, 	since  $\omega\in \Delta_z^{(m-1)}(\langle f(z)\rangle)$,   there are elements $g_0,\ldots, g_{m-1}\in S[z]$ 
	such that 
	\begin{equation}
		\label{expresion_w}
		\omega=g_0f(z)+g_1\Delta_z^{1}(f(z))+\ldots+ g_{m-1}\Delta_z^{m-1}(f(z)).
	\end{equation}
	The hypothesis is that $\ord_{\mathfrak{m}}(\omega)=1$,  hence, by (\ref{EqOrdDelta}), necessarily, 
	\begin{equation}
		\label{nonulo}
		g_{m-1}\Delta_z^{m-1}(f(z))=g_{m-1}(mz+ (m-1){a_1})\neq 0,
	\end{equation}
	and $g_{m-1}$ and $m$ have to be units in $R$. 
	
	After looking at the equality (\ref{expresion_derivadas}) in $R'$ 	
	and by (\ref{nubarra_theta}),  we have that 
	$$\nub_{\mathfrak{m}'}\left({\Delta_z^i(f(z))'} \right)\geq
	(m-i)\nub_{\mathfrak{m}'}(z'),\qquad \text{ for } i=0,\ldots,m-1.$$ 
	In particular,
	\begin{equation}
		\label{desigualdad_z_prima} 
		\nub_{\mathfrak{m}}\left( {\Delta_z^i(f(z))'} \right)>\nub_{\mathfrak{m}'}(z'),\qquad \text{ for } i=0,\ldots,m-2.
	\end{equation}
	
	To prove  (i)  observe that by (\ref{desigualdad_z_prima}) and the  expresion in  (\ref{expresion_w})  we have that 
	$$\omega'=g_{m-1}(mz'+ (m-1){a_1}) \text{ mod } {\mathfrak m}^2.$$
	
	Now, since $a_1\in {\mathfrak m}_S$ and because ${\mathfrak m}'=\langle z'\rangle +{\mathfrak m}_SR'$, we have that also ${\mathfrak m}'=\langle w'\rangle+{\mathfrak m}_SR'$. Thus, $w'$ is part of a minimal set of generators of ${\mathfrak m}'$. From here we also conclude (ii), since the hypothesis 
	$\ker(\lambda_{\mathfrak{m}'})=(0)$, implies that necesarily $\nub_{{\mathfrak m}'}(\omega')=1$.

	(iii)   If $\nu_{\mathfrak{m}_S}(a_1)=1$ then by (\ref{nubarra_theta})  $\nub_{\mathfrak{m}'}(z')=1$,   and
	the result follows. Suppose, otherwise,  that $\nu_{\mathfrak{m}_S}(a_1)\geq 2$. 
	
Then, by (\ref{desigualdad_z_prima}) and considering the equality (\ref{expresion_w}) in $R'$, we have that  
	$$\nub_{\mathfrak{m}'}({\omega}')\geq
	\min\{\nub_{\mathfrak{m}'}(z'),\ord_{\mathfrak{m}_S}(a_1)\}=
	\nub_{\mathfrak{m}'}(z').$$
	
	Finally, if $\ker(\lambda_{\mathfrak{m}})=\langle \In_{\mathfrak{m}'}(z')\rangle$, then $\nub_{\mathfrak{m}'}(z')>1$, and  therefore, $\nub_{\mathfrak{m}'}(w')>1$. Finally, again using (\ref{expresion_w}), the strict inequality in (\ref{desigualdad_z_prima}), 
	equality (\ref{nonulo})  and the fact that  $g_{m-1}$ and $m$ are units, we have that  $\In_{\mathfrak{m}'}(z')=\In_{\mathfrak{m}'}({\omega}')$. Thus  (iv) follows.
\end{proof}

\begin{corollary} \label{orden_maximal}
	Under  the hypothesis and with the notation  of Proposition \ref{orden_1_ideal_derivadas},
	$\SSl(R')=\nub_{\mathfrak{m}'}({\omega}')$.
\end{corollary}

\begin{proof} If $\SSl(R')=1$ the claim follows from  Proposition \ref{orden_1_ideal_derivadas} (ii).  Otherwise, there is some $\delta'\in\mathfrak{m}'$ with $\ker\lambda_{\mathfrak{m}'}=\langle \In_{\mathfrak{m}'}\delta' \rangle$.
By Lemma \ref{lema_lema_v4} 
there is some translation of $z'$, $t'=z'+s'$, with $s\in S$,  so that  $\nub_{\mathfrak{m}'}(t')\geq \nub_{\mathfrak{m}'}({\delta}')$, 
	$$\langle \In_{{\mathfrak m}'}(t')\rangle =\langle \In_{{\mathfrak m}'}(\delta')\rangle,$$
	and $R'=S[t']$.
Then,    $R'=S[t']= S[t]/\langle f(t-s)\rangle$. By (\ref{derivadas_relativas}), $w\in \Delta_t^{m-1}(f(t-s))$.   Applying Proposition \ref{orden_1_ideal_derivadas} (iii) and (iv) to  $t'$ instead of $z'$, the result follows. 

\end{proof}

\begin{theorem}
	\label{todas_igual}
	Assume the setting and the notation of 
	\S  \ref{HyperLocal}, and suppose now that $S$ is a smooth and  essentially of finite type over a perfect field $k$. Let   $\xi$ be the closed point of  $X=\Spec(R')$  determined by ${\mathfrak m}'$. 
	Then
	$$\SSl(\mathcal{O}_{X,\xi})=\Hord_X^{(d)}(\xi)\leq \ord_{X}^{(d)}(\xi).$$
	Moreover if $\caract(k)=0$, then
	$\Hord_X^{(d)}(\xi)=\SSl(\mathcal{O}_{X,\xi})=\ord_{X}^{(d)}(\xi)$.
\end{theorem}

\begin{proof}   
	Observe first that   from the definition of  $\Hord_X^{(d)}$ in \S \ref{DefSlope1}, the equality in (\ref{nubarra_theta}), and Lemma \ref{lema_lema_v4}, 
	 $\Hord_X^{(d)}(\xi)=\min\{ \SSl(\mathcal{O}_{X,\xi}), \ord_{X}^{(d)}(\xi) \}$.  Next, note that  the inequality 
	\begin{equation}\label{desigual} 
	\SSl(\mathcal{O}_{X,\xi})\leq \ord_{X}^{(d)}(\xi)
	\end{equation}
	always holds. This  is straightforward if  $\mathcal{O}_{X,\xi}$ is not in the extremal case (i.e., if $\SSl(\mathcal{O}_{X,\xi})=1$). Otherwise, by Lemma \ref{lema_lema_v4}, after a suitable translation of $z'$  by some element of $S$, we can assume that $\SSl(\mathcal{O}_{X,\xi})=\nub_{{\mathfrak m}'}(z')>1$. The Rees algebra representing $\Mm$ is $\G^{(d+1)}=\Diff(R[fW^m])$, with $f$ as in (\ref{forma_f}), and then, the elimination algebra, is, up to weak equivalence,  
	 an $S$-Rees algebra generated by weighted homogeneous polynomials in the coefficients of $f$  (see \S \ref{Emininacion_Hiper}). Now the inequality (\ref{desigual}) follows  from  (\ref{nubarra_theta}). 
\medskip 

It remains to see that when the characteristic of $k$ is zero, we have that   $\SSl(\mathcal{O}_{X,\xi})=\ord_{X}^{(d)}(\xi)$. With the same notation as in \S \ref{HyperLocal}, note that, after a convenient traslation of $z$, we can assume that 
	$$f(z)=z^m+a_2z^{m-2}+\ldots+a_{m-1}z+a_m,$$
	i.e., $a_1=0$, and   $z\in \Delta_z^{(m-1)}(\langle f(z)\rangle)$ is an element of order 1 at ${\mathfrak m}$. In such case, up to weak equivalence,  the elimination algebra of $\G=\Diff(R[fW^m])$ is $S[a_2W^2,\ldots, a_mW^m]$ (see \S \ref{Emininacion_Hiper}), hence 
	$$\ord_{X}^{(d)}(\xi)=\min \left\{\frac{\ord_{{\mathfrak m}_S}(a_i)}{i}, i=2,\ldots,m\right\}=\nub_{{\mathfrak m}'}(z')=\SSl(\mathcal{O}_{X,\xi}),$$
	where the last equality follows from Corollary \ref{orden_maximal}. 	
\end{proof}
\begin{corollary}
\label{corolario_H_ord}
Let $X$ be an equidimensional algebraic variety of dimension $d$  defined over a perfect field $k$. Let $\zeta\in X$ be a  singular point of multiplicity $m>1$.  
	Then
	$$\SSl(\mathcal{O}_{X,\zeta})=\Hord_X^{(d)}(\zeta)\leq \ord_{X}^{(d)}(\zeta).$$
	Moreover if $\caract(k)=0$ then
	$\Hord_X^{(d)}(\zeta)=\SSl(\mathcal{O}_{X,\zeta})=\ord_{X}^{(d)}(\zeta)$.
\end{corollary}

\begin{proof} If $\zeta$ is not a closed point, then select some closed point $\xi\in X$ with the same multiplicity $m$, so that $\xi\in \overline{\{\zeta\}}$, the closure of $\zeta$.
	By Theorem \ref{presentaciones_mult},  after considering a convenient local  \'etale neighborhood of ${\mathcal O}_{X,\xi}\to (R',{\mathfrak m}')$,  we can assume that there is a smooth $k$-algebra $S$ and   a finite transversal morphism $S\to R'$, so that if we write $R'=S[z_1,\ldots,z_r]/I=S[z_1',\ldots,z_r']$, where $z_1,\ldots, z_r$ are variables and $z_i'$, $i=1,\ldots,r$, denotes the class module $I\subset S[z_1,\ldots,z_r]$,  and, letting $f_i(z_i)\in S[z_i]$ be the minimal polynomial of $z_i$ over $K(S)$, the fraction field of $S$, we have that 
	$$\Hord^{(d)}_{X,\xi}=\min_i\{\Hord^{(d)}_{X_i,\xi}\},$$
	where $X_i=\Spec(S[z_i]/\langle f_i(z_i)\rangle)$ (see \S \ref{DefOrdHironaka}).  Since the Samuel slope of a local ring does not change after a local \'etale extension,  
	$\SSl({\mathcal O}_{X,\xi})=\SSl (R')$ (see \S \ref{properties_slope}). Now,  for $i=1,\ldots, r$ there are     finite transversal extensions of local rings,
	$$\xymatrix{S \ar[r] &  S[z_i'] \simeq S[z_i]/\langle f(z_i)\rangle \ar[r] &  R'.   }
	$$
	   By Corollary  \ref{maximo_hipersuperficies} 
	$$\SSl({\mathcal O}_{X,\xi})=\min_i\{\SSl({\mathcal O}_{X_i,\xi})\},$$
 and  the claim follows from Theorem \ref{todas_igual} if  $\zeta=\xi$. Otherwise, select some prime ${\mathfrak p}\subset R'$ dominating   $\zeta$ and let ${\mathfrak q}:={\mathfrak p}\cap S$. Then the claim follows in a similar way as before after localizing and considering the   finite transversal extension 
	$S_{\mathfrak q}\to R'_{\mathfrak p}=S_{\mathfrak q}[z_1',\ldots, z_r']$. 
\end{proof}

\section{Hironaka's polyhedron}
\label{poliedro_Hironaka}

In this section we recall the definition of Hironaka's polyhedron and some  properties.
The original reference is \cite{HirPoly}, other references are
\cite{CossartBulletin} or \cite[Appendix]{CossartGiraudOrbanz}.
See \cite{Schober} and \cite[Appendix]{CossartJannsenSaito} for more recent papers.

We end this section with Theorem \ref{ExtenPolHir}, from where it follows that, for the computation of the Hironaka's polyhedrom of a  singularity, we can assume that the residue field of the corresponding local ring is infinite,
a proof that we include to make the presentation self-contained.
This will pay a role in the last section of this paper. 
	
\begin{definition} \label{DefHirPolyY}
Let $(R,\mathfrak{m})$ be a regular local ring of dimension $n$.
Fix some $0<r<n$ and let $$(u,y)=(u_1,\ldots,u_{n-r},y_1,\ldots,y_r)$$ be a regular system of parameters in  $R$.
For a given $f\in R$   consider the CP-expansion w.r.t. $(u,y)$ as in Proposition \ref{ExpanCosPil}:
$$f=\sum_{(\alpha,\beta)\in S(f)}c_{\alpha,\beta}u^{\alpha}y^{\beta},$$
where $c_{\alpha,\beta}$ are units in $R$ and the finite set $S(f)$ is
determined by $f$ and $(u,y)$. Let $m=\ord_{\mathfrak m}(f)$, assume that
$f+\langle u\rangle\in R/\langle u\rangle$ is non-zero and that 
$m=\ord_{(y)}(f+\langle u\rangle)$. Under these assumptions,  the {\em Hironaka's polyhedron of $f$ w.r.t. $(u,y)$} is defined as: 
$$\Delta_R(f,u,y)=\Delta(f,u,y):={\mathrm{Conv}}\left(
\left\lbrace
\frac{\alpha}{m-|\beta|} \mid (\alpha,\beta)\in S(f), |\beta|<m
\right\rbrace + \mathbb{N}^{n-r}\right),$$
where $\mathrm{Conv}$ denotes the convex hull  of the set in ${\mathbb R}^{n-r}$ 
(note that $\Delta(f,u,y)\subset\mathbb{R}^{n-r}_{\geq 0}$).
In addition, one can associate  to this polyhedron  the rational number: 
$$\delta(\Delta(f,u,y))=\min\{|\gamma|\mid \gamma\in\Delta(f,u,y)\}\in\mathbb{Q}. $$
\end{definition}

Note also that if $\widehat{R}$ is the $\mathfrak{m}$-adic completion of $R$ and if we also denote   by $f$ and $(u,y)$ the corresponding images in $\widehat{R}$,  then
$$\Delta_{\widehat{R}}(f,u,y)=\Delta_R(f,u,y).$$

\begin{definition} \label{DefHirPolyU}
Let $(R,\mathfrak{m})$ be a regular local ring of dimension $n$.
Let $(u_1,\ldots,u_{n-r})$ be part of a regular system of parameters in $R$.
Let $f\in R$, set $m=\ord_{\mathfrak m}(f)$, assume that
$f+\langle u\rangle\in R/\langle u\rangle$ is non-zero and that 
$m=\ord_y(f+\langle u\rangle)$. Then 
the {\em Hironaka's polyhedron of $f$ w.r.t. $(u)$} is defined as: 
$$\Delta(f,u):=\bigcap_{\hat{y}}\Delta(f,u,\hat{y}),$$ 
where intersection runs over all $\hat{y}=\hat{y_1},\ldots,\hat{y_r}\in\widehat{R}$ such that $(u,\hat{y})$ is a regular system of 
parameters in  $\widehat{R}$. Analogously, associated to this polyhedron there is the rational number:  
$$\delta(\Delta(f,u))=\min\{|\gamma|\mid \gamma\in\Delta(f,u)\}\in\mathbb{Q}.$$
\end{definition}

The following theorem of Hironaka says that if $R$ is complete,  then  $\Delta(f,u)$ can be realized as the polyhedron of $f$ respect to some regular system of parameters in $\hat{R}$:

\begin{theorem} \label{ThHirPoly}
	\textbf{(Hironaka)} \cite[(4.8)]{HirPoly}
	Given $u_1,\ldots,u_{n-r}$ and $f\in R$ as in Definition \ref{DefHirPolyU}, 
	there exist regular parameters $\hat{y}_1,\ldots,\hat{y}_r$ in $\widehat{R}$, the ${\mathfrak m}$-adic completion of $R$,  such that $\Delta(f,u)=\Delta(f,u,\hat{y})$. Moreover, if $\widehat{R}$ is isomorphic to a power series ring in variables $(u,y)$ then there are series in $(u)$, 
	$s_i(u)$, $i=1,\ldots,r$,  such that
	$$\Delta(f,u)=\Delta(f,u,y_1-s_1(u),\ldots, y_r-s_r(u)).$$
	
\end{theorem}

\medskip

In principle, in the previous definitions and statements, we do not impose  a specific condition on the elements $y=(y_1,\ldots, y_r)$ of the regular local ring other than that of being part of a regular system of parameters and   that $f+\langle u\rangle \in R/\langle u\rangle$ preserves the order $m$ of $f$. However, from this point on, we will be paying attention to some special collections of elements  $y=(y_1,\ldots, y_r)$ that are  related to the  singularity of $f$ at the maximal ideal ${\mathfrak m}$. This is the motivation for the following definition.
 
\begin{definition} \label{Directriz}
Let $k$ be a field, let $I\subset k[X_1,\ldots,X_n]$ be a homogeneous ideal, and let  $\mathcal{C}=\Spec(k[X]/I)$ be the cone defined by $I$. The {\em directrix of the cone $\mathcal{C}$}, $\Dir(\mathcal{C})$, is the largest linear subspace
which leaves the cone invariant after translations:
$$\Dir(\mathcal{C})=\sum_{\mathcal{C}+W=\mathcal{C}} W.$$

If $(R,\mathfrak{m},k)$ is a regular local ring and $R'=R/\mathfrak{a}$, with
$\mathfrak{a}\subset R$ is an ideal,  the  {\em directrix
of $R'$}  is the directrix of the cone
${\mathcal  C}=\Spec(\Gr_{\mathfrak{m}'}(R'))\subset\Spec(\Gr_{\mathfrak{m}}(R))$, where ${\mathfrak m}'$ is the maximal ideal of $R'$.
\end{definition}

In terms of equations, if $(R,\mathfrak{m})$ is a regular local ring and $f\in R$,
then the ideal of the directrix of $R/\langle f\rangle$ is generated by a minimal
set of linear forms $Y_1,\ldots,Y_r\in\Gr_{\mathfrak{m}}(R)$ such that the 
initial part $\In_{\mathfrak m}(f)$ is a polynomial in $Y_1,\ldots,Y_r$. In this sense, we say that a collection of elements $y_1,\ldots,y_r\in R$ {\em determine the directrix of
$R/\langle f\rangle$} if $(y_1,\ldots,y_r)$ is part of a regular system of parameters in $R$ and the ideal of the directrix is generated by 
$\In_{\mathfrak m} (y_1),\ldots,\In_{\mathfrak m} (y_r)\in \Gr_{\mathfrak{m}}(R)$.
See \cite{Schober} and \cite{BHM_Ridge} for thorough expositions on this topic.

\begin{parrafo} \label{FactsRidge} {\bf The ridge and the directrix of a cone.}
In the following theorem  the notion of {\em ridge of a cone}, $\Rid(\mathcal{C})$,  is mentioned. For the 
 precise definition we refer to \cite[Faîte d'un cône 1.5]{Giraud},  
\cite[Definitions 2.6 and 2.7]{Schober} or \cite{BHM_Ridge}.  In the following lines we recall some known facts about the ridge and the directrix that  will be used later in the paper. 

Let $k$ be a field, and suppose that  ${\mathcal C}\subset {\mathbb A}_k^n$ is a cone.
From the definitions, it follows that
$\Dir(\mathcal{C})\subset\Rid(\mathcal{C})$ as
cones in $\mathbb{A}^n_k$,
and in general they are different.
However, if the field $k$ is perfect, then   we have the equality
$\Dir(\mathcal{C})=(\Rid(\mathcal{C}))_{\mathrm{red}}$ 
see \cite[Remark 2.8]{Schober} or \cite[Remark 3.12]{BHM_Ridge}.

Moreover, the ridge of a cone can be computed with differential operators, see 
\cite[Lemme 1.7]{Giraud}.
More precisely, if $F(X_1,\ldots,X_n)$ is a homogeneous polynomial of degree $m$ defining a cone ${\mathcal C}\subset {\mathbb A}_k^n$, then the ideal of  $\Rid(\mathcal{C})$ in $k[X_1,\ldots, X_n]$   is generated by the elements 
	$$\Delta_{X^{\alpha}}(F), \qquad |\alpha|<m.$$
\end{parrafo}

\begin{theorem} \label{ThCossScho}
\cite[Theorem 2.5]{CossartSchober2021}, \cite{CossartPiltant2015}
With the same assumptions as in Definition \ref{DefHirPolyU},
suppose that the codimension of the directrix of $R/\langle f \rangle$
is $r$.
If the regular local ring $R$ satifies one of the following conditions
\begin{description}
\item[Hensel] $R$ is henselian, or
\item[Directrix] the dimension of the ridge coincides with the dimension of the directrix of $R/\langle f \rangle$, or
\item[Char] $\mathrm{char}(k)\geq\dfrac{\dim(R)-1}{2}+1$, or
\item[Pol] there exist an excellent regular local ring $(S,\mathfrak{n})$,
$S\subset R$ and there are $y_1,\ldots,y_r\in R$ such that they 
determine the directrix of $R/\langle f \rangle$ and
$R=S[y_1,\ldots,y_r]_{\mathfrak{m}}$,
\end{description}
then there are $z_1,\ldots,z_r\in R$, such that $(u,z)$ is a regular
system of parameters of $R$ and
$$\Delta(f,u)=\Delta_R(f,u,z).$$
\end{theorem}

In general, given a regular system of parameters $(u,y)$, the elements $z_1,\ldots,z_r$ cannot
be obtained by translating the $y_i$'s as in Theorem \ref{ThHirPoly}.
If the ring $R$ satisfies condition (Pol) of Theorem \ref{ThCossScho}
then there are $s_i\in\langle u\rangle\subset S$ with $z_i=y_i-s_i$, $i=1,\ldots, r$.

\medskip

The results in the next lines give sufficient conditions to guarantee that  a given  regular system of parameters $(u,y)$ in $R$ is so that $$\Delta(f,u)=\Delta_R(f,u,y),$$
see Theorem \ref{teorema_preparado} below. This result will be used in the proof of Theorem \ref{ExtenPolHir}.

\begin{definition}
\cite[Definitions  8.2, 8.13, 8.15]{CossartJannsenSaito}
With the same  assumptions and notation  as in Definition \ref{DefHirPolyY}, 
assume also that $y_1,\ldots,y_r$ determine the directrix of 
$R/\langle f\rangle$.
Denote by 
$k=R/\mathfrak{m}$ the residue field and consider the isomorphism
$\Gr_{\mathfrak{m}}(R)=k[U,Y]$.
Consider the expansion of $f$ as in Proposition \ref{ExpanCosPil}
$$f=\sum_{(\alpha,\beta)\in S(f)}c_{\alpha,\beta}u^{\alpha}y^{\beta}.$$
For $v\in S(f)$   set
$$\In_v(f)=\In_{\mathfrak{m}}(f)+\In_v(f)^{+},$$
where
$$\In_v(f)^{+}=\sum_{(\alpha,\beta)}\overline{c}_{\alpha,\beta}
U^{\alpha}Y^{\beta},$$
where the sum runs over all $(\alpha,\beta)\in S(f)$ such that
$|\beta|<m$ and $\dfrac{\alpha}{m-|\beta|}=v$. Note that the initial part $\In_{\mathfrak{m}}(f)$ is polynomial in
$k[Y]$.
Set $F(Y)=\In_{\mathfrak{m}}(f)$, and
we say that $(f,u,y)$ is \emph{solvable at $v$} if there are 
$\lambda_1,\ldots,\lambda_r\in k[U]$ such that
$F(Y_1+\lambda_1,\ldots,Y_1+\lambda_1)=\In_v(f)$.
 We say that $(f,u,y)$ is \emph{well-prepared} if it is not solvable 
for any $v\in S(f)$.
\end{definition}

\begin{theorem}\label{teorema_preparado} 
\cite[Theorem 8.16]{CossartJannsenSaito}
With the same  assumptions as in Definition \ref{DefHirPolyY}, 
assume also that $y_1,\ldots,y_r$ determine the directrix of 
$R/\langle f\rangle$. If $(f,u,y)$ is well prepared, then $\Delta(f,u)=\Delta(f,u,y)$.
\end{theorem}

\begin{theorem} \label{ExtenPolHir}
Let $(R,\mathfrak{m}, k)$ be a regular local ring. let $t$ be a variable, set  $R_1=(R[t])_{\mathfrak{m}[t]}$
and  let $\varphi:R\to R_1$ be the canonical morphism.
Let $(u_1,\ldots,u_{n-r},y_1,\ldots,y_r)$ be a regular system of parameters of $R$. Let $f\in R$ and suppose  that the directrix of $R/\langle f\rangle$ is determined 
by $y_1,\ldots,y_r$.
Then 
$$\Delta_{R}(f,u)=\Delta_{R_1}(\varphi(f),\varphi(u)).$$
\end{theorem}

\begin{proof}
Since
$\widehat{R_1}=(\widehat{R}[t])_{\widehat{\mathfrak{m}}[t]}$,
we can assume that $R$ is complete.
 Note that if $z_1,\ldots,z_r\in R$ are elments such that $(u,z)$ is a regular system of 
parameters in  $R$, then $(\varphi(u),\varphi(z))$ is a regular system of parameters
in  $R_1$ and $\Delta_{R}(f,u,z)=\Delta_{R_1}(\varphi(f),\varphi(u),\varphi(z))$. 
 Thus,  we have the inclusion
$$\Delta_{R_1}(\varphi(f),\varphi(u))\subset\Delta_{R}(f,u).$$

To check that the other inclusion also holds, observe that, since 
 ince $R$ is complete, by \cite[(4.8)]{HirPoly} there exist
$z_1,\ldots,z_r\in R$ such that $(u,z)$ is a regular system of parameters of $R$ and  $\Delta_R(f,u)=\Delta_R(f,u,z)$.  We have that
$$\Gr_{\mathfrak{m}}(R)\cong k[U,Z],\qquad \text{ and } \qquad 
\Gr_{\mathfrak{m}_1}(R_1)\cong k(t)[U,Z].$$
The initial part $\In_{\mathfrak m}(f)\in \text{Gr}_{\mathfrak m}(R)$ is a homogeneous polynomial of degree $m$ and it
 is expressed in the variables $Z$,  $F(Z):=\In_{\mathfrak m}(f)\in k[Z]$.

Set $u'=\varphi(u)$,  $z'=\varphi(z)$ and  $f'=\varphi(f)$. Now, to get to a contradiction, suppose that $\Delta_{R_1}(f',u',z')\neq\Delta_{R_1}(f',u')$.
Since $R_1$ is complete, by Hironaka \cite[Theorem (3.17) and Theorem (4.8)]{HirPoly},
there is some  vertex $v$ of $\Delta_{R_1}(f',u',z')=\Delta_{R}(f,u,z)$ such that
$v$ is solvable for $f'$. This means that
there are $\lambda_1(t),\ldots,\lambda_r(t)\in k(t)[U]$ such that
$$F(Z_1+\lambda_1(t),\ldots,Z_r+\lambda_r(t))=\In_v(\varphi(f)).$$

Note that $\In_v(\varphi(f))=\In_v(f)\in k[U,Z]$.  Set $G(U,Z)=\In_v(f)$.
Then, we have the following equality of polynomials in $k(t)[U,Z]$:
\begin{equation} \label{EqPolys}
F(Z_1+\lambda_1(t),\ldots,Z_r+\lambda_r(t))=G(U,Z). 
\end{equation}
 
\medskip

We claim that:
\begin{itemize}
\item If $k$ is infinite then there are 
$\tilde{\lambda}_i\in k[U]$, $i=1\ldots,r$, such that 
$F(Z_1+\tilde{\lambda}_1(t),\ldots,Z_r+\tilde{\lambda}_r(t))=G(U,Z)$.
\item If $k$ is finite then $\lambda_i(t)\in k[U]$ for $i=1,\ldots,r$.
\end{itemize}

Assume that the field $k$ is infinite.
Since $\lambda_i(t)\in k(t)[U]$,   there is some  $t_0\in k$ 
such that $\tilde{\lambda}_i=\lambda_i(t_0)\in k[U]$ and the claim follows.

If   $k$ is finite, then it is perfect and the dimensions of
the directrix and the ridge of $F(Z)$ are equal  (see \S \ref{FactsRidge}).
In addition in this case there are differential operators $D_i$
such that
$$D_i(F(Z))=(L_i(Z))^{p^{e_i}}, \qquad i=1,\ldots,r,$$
where $L_i(Z)$ is a linear form on $Z_1,\ldots,Z_r$, $e_i\geq 0$, and
$L_1,\ldots,L_r$ are linearly independent (see again \S \ref{FactsRidge}).
Now, note that $D_i(F(Z+\lambda(t))=(D_iF)(Z+\lambda(t))$
(if $H(Z)=F(Z+\lambda(t))$ then $D_i(H(Z))=D_i(F(Z+\lambda(t))$).
Applying $D_i$ to the equation (\ref{EqPolys}) we have
$$(L_i(Z_1+\lambda_1(t),\ldots,Z_r+\lambda_r(t)))^{p^{e_i}}=
D_i(G(U,Z))\in k[U,Z],
\qquad i=1,\ldots,r.$$
In particular
$L_i(\lambda_1(t),\ldots,\lambda_r(t))^{p^{e_i}}\in k[U]$ and
hence (since $k$ is perfect)
$L_i(\lambda_1(t),\ldots,\lambda_r(t))\in k[U]$, 
and the claim follows for the fact that $L_1,\ldots,L_r$ are linearly independent.
Hence, in either case,    the vertex $v$ is solvable for $\Delta_R(f,u,z)$ which gives us a contradiction.
\end{proof}

\section{Samuel slope, refined slope and Hironaka's polyhedron}
\label{Sec_Refined}

\begin{parrafo}\label{hipotesis_notacion} 
{\bf Hypotheses and notation.} 
Along this section $(R,\mathfrak{m})$ will denote a regular local ring of Krull dimension $n$.
We will be considering quotients for the form $R/\langle f\rangle$ such that $f\in R$ is not a unit.
We will write $R'=R/(f)$ and  will use   $\mathfrak{m}'$ to refer to  the maximal ideal of $R'$. We will assume that $m=ord_{\mathfrak m}(f)>1$, i.e., $m=\ord_{\mathfrak m}(f)=e(R')>1$.
As in previous sections, we write  $z'$  to denote the natural image in $R'$ of an element $z\in R$.

We will be using regular systems of parameters of $R$
of the form  
$$(u,y)=(u_1,\ldots,u_{n-r},y_1,\ldots,y_r),$$ for some $0<r<n$, so that $\langle \In_{\mathfrak{m}}(y_1), \ldots,  \In_{\mathfrak{m}}(y_r)\rangle$ is the ideal of the directrix of $f$ (in general, we will use $u$ to refer to  $\{u_1,\ldots,u_{n-r}\}$ and $y$ to refer to the set $\{y_1,\ldots,y_{r}\}$).  In this context,  we will consider the corresponding   Hironaka's polyhedron $\Delta(f,u,y)$, and under these assumptions,
in \cite[Corollary 5.1]{CossartJannsenSchober2019} it is proven that
$\delta(\Delta(f,u))$ is an invariant of the singularity $\Spec(R')$, therefore, it does not depend on the choice of $u_1,\ldots,u_{n-r}$.
So we will refer to this rational number as the \emph{$\delta$ invariant of Hironaka's polyhedron}.
\end{parrafo}

This section is devoted to establishing a connection between the $\delta$ invariant of 
Hironaka's polyhedron of a hypersurface at a singular point and the   
Samuel slope of the corresponding local ring.
Corollary \ref{delta_dir_dim_one} and Theorem \ref{RSSl_completion} reprove
\cite[Corollary 5.1]{CossartJannsenSchober2019} and provide a way  to interpret
Hironaka's polyhedron invariants in terms of the asymptotic Samuel function of a local ring.
In our study we will distinguish between the extremal and the non-extremal cases. 
\medskip

\noindent{\bf Case 1. Suppose $R'$ is in the extremal case}

\medskip 

\noindent Observe that this corresponds to the case in which the initial part of  $f$ is
$\In_{\mathfrak{m}}{f}=\In_{\mathfrak m}{y}_1^m$, where $y_1\in R$ is a regular parameter, i.e., the ideal of the directix of $f$ has one generator which is $\In_{\mathfrak m} y_1$. 

\begin{theorem}
\label{Casor1}
Let $R'=R/(f)$ be a hypersurface in the extremal case. Let $u_1,\ldots,u_{n-1}, y_1\in R$ be  a regular system in  parameters of $R$ and suppose that 
$\In_{\mathfrak m}{f}=\In_{\mathfrak m}{y}_1^m$.
Then
$$\nub_{\mathfrak{m}'}({y}'_1)=\delta(\Delta(f,u,y_1)).$$
\end{theorem}

\begin{proof}
Consider the CP-expansion of $f$ w.r.t. $(u,y_1)$, 
$$f=\sum_{(\alpha,i)\in S(f)}c_{\alpha,i}u^{\alpha}y_1^i,$$
where $c_{\alpha,i}$ are units. From the  hypothesis $(0,m)\in S(f)$, thus,    up to a unit $f$ can be expressed in pseudo-Weierstrass form,
$$f=y_1^m+a_1y_1^{m-1}+\cdots+a_m,$$
for some $a_i\in R$, and by Theorem \ref{GeneralHickel} we have 
$$\nub_{\mathfrak{m}'}(y'_1)=\min\left\lbrace
\frac{\ord_{\mathfrak{m}}(a_i)}{i} \mid i=1,\ldots,m
\right\rbrace=\delta(\Delta(f,u,y_1)).$$
\end{proof}

\begin{corollary}
\label{delta_dir_dim_one}
Suppose  $R'=R/(f)$ is a hypersurface in the extremal case, and let $y_1\in R$ be a regular parameter such that  $\In_{\mathfrak{m}}{f}=\In_{\mathfrak{m}}{y}_1^m$. Then, for any collection of regular parameters   $u_1,\ldots, u_{n-1}\in R$ such that  $(u_1,\ldots,u_{n-1},y_1)$ is a regular system of parameters in $R$ we have that:
$$\SSl(R')=\delta(\Delta(f,u)).$$
Moreover, if $R$ is excellent, then there is a regular parameter $\tilde{y}_1\in R$ 
such that  $\In_{\mathfrak{m}}{f}=\In_{\mathfrak{m}}\tilde{y}_1^m$ and
$$\SSl(R')=\nub_{\mathfrak{m}'}(\tilde{y}'_1)=
\delta(\Delta(f,u,\tilde{y}_1))=\delta(\Delta(f,u)).$$
\end{corollary}

\begin{proof} 
If $\SSl(R')=\infty$ then there exists a sequence $\{y'_n\}_{n=1}^{\infty}\subset R'$
such that $\nub_{\mathfrak{m'}}(y'_n)\to \infty$.
We have that $\delta(\Delta(f,u))\geq\delta(\Delta(f,u,y_n)=\nub_{\mathfrak{m'}}(y'_n)$ and then $\delta(\Delta(f,u))=\infty$.

If $\SSl(R')<\infty$, using the fact that $\SSl(R')=\SSl(\widehat{R'})$ (see \S \ref{properties_slope}), we have that
$$\exists y'\in R' \text{ such that } \nub_{\mathfrak{m'}}(y')=\SSl(R')
\quad\Longleftrightarrow\quad
\exists \hat{y}'\in \widehat{R'} \text{ such that }
\nub_{\widehat{\mathfrak{m'}}}(\hat{y}')=\SSl(\widehat{R'}).$$
Hence, we can assume that $R'$ is complete.
Now by Lemma \ref{lema_lema_v4} and (\ref{casi_natural}) there exists
$y'\in R'$ such that $\SSl(R')=\nub_{\mathfrak{m}'}(y')=
\delta(\Delta(f,u,y))$.
If $\delta(\Delta(f,u,y))<\delta(\Delta(f,u))$, then there is some
$z\in R$ such that 
$\delta(\Delta(f,u,y)<\delta(\Delta(f,u,z)=\nub_{\mathfrak{m}'}(z')\leq\SSl(R')$,
which is a contradiction.
\medskip

Finally assume that $R'$ is excellent. We only have to consider the case 
$\SSl(R')=\infty$.
The previous sequence $\{y'_n\}_{n=1}^{\infty}$ is Cauchy in $R'$ and
it converges in $\widehat{R'}$ to some $\hat{y}'\in\widehat{R'}$.
By Lemma \ref{LimiteNil} $\hat{y}'$ is nilpotent.
Since $R'$ is excellent, there are nilpotents $\omega'_1,\ldots,\omega'_m\in R'$ such that
$\hat{y}'\in\langle\omega'_1,\ldots,\omega'_m\rangle\widehat{R'}$.
Then there is some $\omega'_i$ such that $\In_{\mathfrak{m}'}\omega'_i\neq 0$,
and therefore $\In_{\mathfrak{m}'}\omega'_i\in\ker\lambda_{\mathfrak{m}'}$.
From here it follows that
$\nub_{\mathfrak{m}'}(\omega'_i)=\infty$.
\end{proof}

\medskip

\noindent{\bf Case 2. When $R'$ is not in the extremal case: $y$-linear cuts}

\medskip 

\noindent When $R'$ is not in the extremal case, $\SSl(R)=1$. Under this hypothesis, we are going to define a second invariant of the local ring: the {\em refined Samuel slope of $R$}.  This notion involves the  use of  the asymptotic Samuel function and the concept of {\em $y$-linear cuts} and {\em $y$-generic linear cuts}.

\begin{definition} \label{DefGenericos}
Let $y'_1,\ldots,y'_r\in R'=R/\langle f\rangle$ be such that
$\In_{\mathfrak{m}'}y'_1,\ldots,\In_{\mathfrak{m}'}y'_r$ are linearly independent
in $\mathfrak{m}'/(\mathfrak{m}')^2$ and so that $\langle \In_{\mathfrak{m}'}y'_1,\ldots,\In_{\mathfrak{m}'}y'_r\rangle \subset \text{Gr}_{{\mathfrak m}'}(R')$ is the ideal of the directrix of $R'$. Observe that since $R'$ is not in the extremal case, $r>1$. A \emph{$y'$-linear cut in $R'$} is a collection of $(r-1)$-expressions
$$L'=\left\{L'_i=\sum_{j=1}^{r}a'_{i,j}y'_j,\quad i=2,\ldots,r\right\},$$
where $a'_{i,j}\in R'$, and such that
the $(r-1)\times r$ matrix $(a'_{i,j})$  contains an $(r-1)$-minor that is a unit in $R'$. 
\medskip 

For a given linear cut, $L'$, we will use the following notation:    ${R'}^{(y',L')}:=R'/\langle L'_2,\ldots,L'_r\rangle$, 
${\mathfrak{m}'}^{(y',L')}\subset {R'}^{(y',L')}$ for the maximal ideal and 
if $z'\in R'$ then  ${z'}^{(y',L')}$ for  the class in ${R'}^{(y',L')}$.
\end{definition}  
 
\begin{remark} \label{RemCuts}
Let $\{L'_2,\ldots,L'_r\}$ be a $y'$-linear cut as above.  Then we can choose elements $y_1,\ldots,y_r\in R$ such that their classes in $R'$ are 
$y'_1,\ldots,y'_r$, and then 
note that $(y_1,\ldots,y_r)$ is part of a regular system of parameters in $R$.
 Also we can choose elements $a_{i,j}\in R$ such that their classes in $R'$ are the $a'_{i,j}$.
Set $L_i=\sum_{j}a_{i,j}y_j$.  Since both  $R$ and $R'$   are local rings,  
  the matrix
$(a_{i,j})$  contains some  $(r-1)$-minor that is a unit in $R$.

Analogously, we will use the following notation:  ${R}^{(y,L)}=R/\langle L_2,\ldots,L_r\rangle$, 
${\mathfrak{m}}^{(y,L)}\subset {R}^{(y,L)}$ for the maximal ideal, and 
  if $z\in R$ then we denote $z^{(y,L)}$ the class in $R^{(y,L)}$. Note that, from the hypothesis $R^{(y,L)}$ is a regular local ring of dimension $n-r+1$.
In fact, if $(u_1,\ldots,u_{n-r},y_1,\ldots,y_r)$ is a regular system of parameters of $R$
and $\ell=1,\ldots,r$,
then $(u_1^{(y,L)},\ldots,u_{n-r}^{(y,L)},y_{\ell}^{(y,L)})$ is a regular system of parameters
of $R^{(y,L)}$ if and only if the determinant of the matrix $(a_{i,j})_{j\neq\ell}$ is a unit in $R$.

In addition, if $\ell_1$ and $\ell_2$ are such that the determinants of
$(a_{i,j})_{j\neq\ell_1}$ and $(a_{i,j})_{j\neq\ell_2}$ are units,
then there is a unit $\omega\in R^{(y,L)}$ such that
$y_{\ell_1}^{(y,L)}=\omega y_{\ell_2}^{(y,L)}$.

Observe also that  ${R'}^{(y',L')}=R^{(y,L)}/\langle f^{(y,L)}\rangle$.
From the discussion above, it follows that  if  both
$(u_1^{(y,L)},\ldots,u_{n-r}^{(y,L)},y_{\ell_1}^{(y,L)})$ and
$(u_1^{(y,L)},\ldots,u_{n-r}^{(y,L)},y_{\ell_2}^{(y,L)})$ are regular systems of parameters  in  
$R^{(y,L)}$, then 
\begin{equation} \label{EqNubCut}
\nub_{{\mathfrak{m}'}^{(y',L')}}({y'}_{\ell_1}^{(y,L)})=
\nub_{{\mathfrak{m}'}^{(y',L')}}({y'}_{\ell_2}^{(y,L)})\leq
\nub_{{\mathfrak{m}'}^{(y',L')}}({y'}_i^{(y,L)}),
\qquad i=1,\ldots,r.
\end{equation}
Given $y'$, for a fixed linear cut $L'$ we will
consider the following value associated to the linear cut $L'$ and $y'$:
\begin{equation}
	\label{minimo}
	\nublin(y',L'):=
	 \min\{\nub_{{\mathfrak{m}'}^{(y',L')}}({y'}_i^{(y',L')})\mid i=1,\ldots,r\}.
\end{equation}
\end{remark}

\begin{remark} \label{RemSimpleCut}
From Definition \ref{DefGenericos} and Remark \ref{RemCuts}, after relabeling $y_1',\ldots, y'_r$, and $L_2',\ldots,L_r'$ we can assume that all $y'$-linear cuts are of the form
$$L'=\{L'_j=y_j'-a_j'y_1': a_j'\in R';  j=2,\ldots, r\}.$$
\end{remark}

\begin{lemma} \label{LemaInclusion}
Let $L'=\{L'_2,\ldots,L'_r\}$ be a $y'$-linear cut such that $e({{R'}^{(y,L)}})=e(R')=m$. Lift $y'_1,\ldots,y'_r$  to some 
  $y_1,\ldots,y_r\in R$  and
let $u_1,\ldots,u_{n-r}\in R$ be such that $(u_1,\ldots,u_{n-r},y_1,\ldots,y_r)$ is a regular system of parameters in $R$.
 If $(u_1^{(y,L)},\ldots,u_{n-r}^{(y,L)},{y_{\ell}^{(y,L)}})$ is a regular system of parameters in $R^{(y,L)}$ then
$$\Delta(f^{(y,L)},u^{(y,L)},y^{(y,L)}_{\ell})\subset\Delta(f,u,y).$$
\end{lemma}

\begin{proof} Lift $L'$ to some $L=\{L_1,\ldots, L_r\}\subset R$. 
After reordering $y_1,\ldots,y_r$
we can assume that $\ell=1$.
Then  by Remark \ref{RemSimpleCut}  there are some  $a_2,\ldots,a_r\in R$ such that
$$\langle L_2,\ldots,L_r\rangle=
\langle a_2y_1-y_2,\ldots,a_ry_1-y_r\rangle.$$
	
Consider the CP-expansion of $f$ w.r.t.  $(u,y)$:
$$ f=\sum_{(\alpha,\beta)\in S(f)}c_{\alpha,\beta}u^{\alpha}y^{\beta}.$$
In the quotient $R^{(y,L)}$ we have the expression
\begin{equation}\label{EqExpQuot}
	f^{(y,L)}=\sum_{(\alpha,\beta)\in S(f)}c^{(y,L)}_{\alpha,\beta}(u^{(y,L)})^{\alpha}(a^{(y,L)})^{\beta}(y_1^{(y,L)})^{\beta},
\end{equation}
where $(a^{(y,L)})^{\beta}=(a_2^{(y,L)})^{\beta_2}\cdots(a_r^{(y,L)})^{\beta_r}$.
	
Let $\Pi:\mathbb{Z}^{n-r}\times\mathbb{Z}^r\to \mathbb{Z}^{n-r}\times\mathbb{Z}$ be the
projection defined by $\Pi(\alpha,\beta)=(\alpha,|\beta|)$.
The expression (\ref{EqExpQuot}) is not, in general, the CP-expansion of 
$f^{(y,L)}$ but we have the inclusion
$$S(f^{(y,L)})\subset \Pi(S(f))+\mathbb{N}^{d+1}.$$
Let  $\varphi:\mathbb{Z}^{n-r}\times\mathbb{Z}^r\to \mathbb{Q}^{n-r}$, and
$\varphi':\mathbb{Z}^{n-r}\times\mathbb{Z}\to \mathbb{Q}^{n-r}$ be defined by
$$\varphi(\alpha,\beta)=\frac{\alpha}{m-|\beta|},
\qquad
\varphi'(\alpha,i)=\frac{\alpha}{m-i}.$$
Note that $\varphi=\varphi'\circ\Pi$ and the Hironaka polyhedra are obtained by 
projection via $\varphi$ and $\varphi'$.
$$\Delta(f,u,y)=\varphi\left(S(f)+\mathbb{N}^{n}\right),
\qquad
\Delta(f^{(y,L)},u^{(y,L)},y_1^{(y,L)})=
\varphi'\left(S(f^{(y,L)})+\mathbb{N}^{n-r+1}\right).$$
Now the result follows.
\end{proof}

\begin{corollary} \label{pendiente_cociente}
Let $L'=\{L'_2,\ldots,L'_r\}$ be a $y'$-linear cut such that
$e({R'}^{(y',L')})=e(R')=m$.
Suppose that 
$$(u_1^{(y,L)},\ldots,u_{n-r}^{(y,L)},y_{\ell}^{(y,L)})$$ is  a regular system of parameters in  $R^{(y,L)}$.  Then
\begin{equation} \label{desigualdad_pendiente_delta_corte}
\nublin(y',L')=
\nub_{{\mathfrak{m}'}^{(y',L')}}({y'}^{(y,L)}_{\ell})
=\delta(\Delta(f^{(y,L)},u^{(y,L)},y_{\ell}^{(y,L)})\geq\delta(\Delta(f,u,y)).
\end{equation}
\end{corollary}

\begin{proof}
The condition $e({R'}^{(y',L')})=e(R')=m$ together with the fact that $L'$ is a $y'$ linear cut guarantees that $f^{(y,L)}$ can be written as a  pseudo-Weierstrass element w.r.t. $(u^{(y,L)},y_{\ell}^{(y,L)})$ for some $\ell\in \{1,\ldots, r\}$. Now, the equalities follow from Theorem \ref{Casor1} and the definition in (\ref{minimo}),
while the inequality on the right is a consequence of Lemma \ref{LemaInclusion}.
\end{proof}

The inequalities (\ref{EqNubCut}) and  (\ref{desigualdad_pendiente_delta_corte}),
together with (\ref{minimo}) motivate the following definitions. 

\begin{definition} \label{Defnugen}
Fix some $y'_1,\ldots,y_r\in R'$ determining the directrix of $R'$ and assume that
the residue field of $R'$ is infinite.
Then the {\em generic $y'$-value} is the minimum among  all possible $\nublin(y',L')$ when $L'$ runs though all $y'$-linear cuts giving the same multiplicity, i.e., 	
$$\nubgen(y'):=\min\left\lbrace\nublin(y',L') \mid
L' \text{ is a ${y'}$-linear cut with }
e({R'}^{(y',\tilde{L}')})=e(R')\right\rbrace.$$
\end{definition}
\begin{remark}
Notice that  $\nublin(y',L')\in\dfrac{1}{m!}\mathbb{N}\cup \{\infty\}$ for all linear cuts $L'$,  thus it makes sense taking the minimum in the right hand side of the expression above.
Note also that the condition $e({R'}^{(y',\tilde{L}')})=e(R')$ can be achieved 
since the residue field is infinite.
\end{remark}
	
\begin{definition}\label{def_generic_linear_cut}
A ${y'}$-linear cut $\tilde{L}'=\{\tilde{L}_2,\ldots,\tilde{L}_r\}$ is said to be a \emph{${y'}$-generic linear cut} if:
\begin{itemize}	
\item[(i)] $e({R'}^{(y',\tilde{L}')})=e(R')=m$;

\item[(ii)] and
\begin{equation}\label{condicion_general_cut}
	\nub\text{-lin}(y',\tilde{L}')=\nubgen(y'). 
\end{equation}
\end{itemize}	
\end{definition}

\begin{remark} Notice that 
the only obstruction to the  existence of ${y'}$-generic linear cuts is condition (i),
and this can always be achieved  if the residue field of $R'$ is infinite.
	\end{remark}

We want to maximize the values of $\nub$ obtained for generic linear cuts.
The following example illustrates that we need to restrict to only considering   generic linear cuts.

\begin{example}
Let $R=(k[y_1,y_2,y_3,u])_{\langle y_1,y_2,y_3,u\rangle}$, where $k$ is a field of characteristic different from 2,  and set  
$$f=y_1^4+y_1^2(y_2+u^a)^2+y_3^4+y_3u^b+u^c,$$
where $2\leq a<\dfrac{b}{3}<\dfrac{c}{4}$. Let $R'=R/\langle f\rangle$.   Then 
	$$\delta(\Delta(f,u))=\frac{b}{3}.$$ 
The ideal of the directrix of $f$ generated by $\In_{{\mathfrak m}}y_1,\In_{{\mathfrak m}}y_2,
\In_{{\mathfrak m}}y_3$. It can be checked that the $y$-linear cut $H=\{y_2,y_3\}$ is generic for $R'$, that 
$$f^{(y,H)}=y_1^4+y_1^2 u^{2a}+u^c,$$
and that 
$$\nub_{{\mathfrak{m}'}^{(y,L)}}({y'}_1^{(y',L')})=a.$$

Now set $z_1=y_1$, $z_2=y_2+u^a$ and $z_3=y_3$.
Considering the ${z}$-linear cut  $L:= \{z_2, z_3\}$, we have  
$$f^{(y,L)}=z_1^4+u^c,$$
and then 
$$\nub_{{\mathfrak{m}'}^{(z,L)}}({z'}_1^{(z',L')})=\frac{c}{4}.$$
On the other hand, taking the ${z}$-linear cut $F=\{z_2, z_3-z_1\}$, 
we get that: 
$$f^{(z,F)}=2z_1^4+z_1^b+u^c,$$ and 
$$\nub_{{\mathfrak m}^{(z,F')}}(\overline{z}_1^{(z,F')})=\frac{b}{3}.$$
Therefore, $L$ is not a generic $z$-linear cut for $R'$. On the other hand,  by Corollary \ref{pendiente_cociente}, $F$ is, necessarily, a generic ${z}$-linear cut for $R'$. 
\end{example}

\begin{theorem} \label{TeoremaVerticesGenericos}
Let $(y_1,\ldots,y_r)$ be part of a regular system of parameters of $R$ that determine the directrix of $f$,
and let  $u_1,\ldots,u_{n-r}$ be such that $(u,y)$ is a regular system of parameters of $R$. Suppose that the residue field of $R$ is infinite. 
Then, there are elements $a_2,\ldots,a_r\in R$ such that if $L:=\{L_j=a_jy_1-y_j, j=2,\ldots, r\}$, and  we consider the
${y'}$-linear cut $L' $ induced by $L$,    then:
\begin{enumerate}
\item[(i)]  $e({R'}^{(y',L')})=e(R')=m$, and
\item[(ii)] $\Delta(f^{(y,L)},u^{(y,L)},y^{(y,L)}_1)=\Delta(f,u,y)$.
\end{enumerate}
In addition, there exists a non-empty Zariski open set $U\subset (R/\mathfrak{m})^{r-1}$ such that
if $(\bar{a}_2,\ldots,\bar{a}_r)\in U$ then (i) and (ii) hold. Moreover,  all  
${y'}$-linear cuts defined by $\{L_j=a_jy_1-y_j, j=2,\ldots,r\}$ with $(\bar{a}_2,\ldots,\bar{a}_r)\in U$
    are generic. 
\end{theorem}

\begin{proof}
Consider the CP-expansion of $f$ w.r.t. to $(u,y)$,
$$f=\sum_{(\alpha,\beta)\in S(f)}c_{\alpha,\beta}u^{\alpha}y^{\beta},$$ 
where $c_{\alpha,\beta}\in R$ are units. Since the initial parts of $y_1,\ldots,y_r$ generate the ideal of the directrix of $f$, there exist $\beta\in\mathbb{N}^r$ such that $|\beta|=m$ and $c_{0,\beta}\neq 0$.

Let $a_2,\ldots,a_r\in R$ and let $L=\{L_2,\ldots,L_r\}$ be the ${y}$-linear cut
given by $L_j=a_jy_1-y_j$, $j=2,\ldots,r$.
 Since the local ring $R^{(y,L)}$ is regular and
$(u_1^{(y,L)},\ldots,u_d^{(y,L)},y_1^{(y,L)})$ is a regular system of parameters, we can consider the expression of f $f^{(y,L)}$, 
$$f^{(y,L)}=\sum_{(\alpha,\beta)\in S(f)}c^{(y,L)}_{\alpha,\beta}
(u^{(y,L)})^{\alpha}(a^{(y,L)})^{\beta}(y^{(y,L)}_1)^{|\beta|}.$$

For $\alpha\in\mathbb{N}^d$ and $i\in\mathbb{N}$, set
$$g_{\alpha,i}=
\sum_{{\substack{(\alpha,\beta)\in S(f) \\ |\beta|=i}}}
	c^{(y,L)}_{\alpha,\beta}(a^{(y,L)})^{\beta},$$
so that 
\begin{equation} \label{Eqfprima}
	f^{(y,L)}=\sum_{(\alpha,i)\in\Pi(S(f))}g_{\alpha,i}
	(u^{(y,L)})^{\alpha}(y^{(y,L)}_1)^i,
\end{equation}
where $\Pi:\mathbb{Z}^d\times\mathbb{Z}^r\to \mathbb{Z}^d\times\mathbb{Z}$ is the 
projection defined by $\Pi(\alpha,\beta)=(\alpha,|\beta|)$.

Note that $(0,m)\in\Pi(S(f))$ since the  $\In_{\mathfrak{m}}(f)$ can be expressed in terms  of $\In_{\mathfrak{m}}(y_1),\ldots,\In_{\mathfrak{m}}(y_r)$.
In fact, if $g_{0,m}$ is a unit in $R^{(y,L)}$ then $f^{(y,L)}$ has order $m$.
Moreover if for every $i$ such that there is $\alpha$ with $(\alpha,i)\in\Pi(S(f))$,
we have that $g_{\alpha,i}$ is a unit in $R^{(y,L)}$, then 
$S(f^{(y,L)})=\Pi(S(f))$ and the expression (\ref{Eqfprima}) is the CP-expansion
of $f^{(y,L)}$ w.r.t. $u^{(y,L)},y^{(y,L)}_1$.

Let  $\bar{g}_{\alpha,i}$ be the class of $g_{\alpha,i}$ in $R/\mathfrak{m}$.
Then  $\bar{g}_{\alpha,i}$ is a polynomial in $\bar{a}_2,\ldots,\bar{a}_r$, where $\bar{a}_j$ is the class of $a_j$ in $R/\mathfrak{m}$.
 Thus, to obtain the result, we need to restrict to the values of $a_2,\ldots,a_r$ that guarantee that all the elements $g_{\alpha,i}$ are units. One way to do this is by restricting to the Zariski open set in $(R/{\mathfrak m})^{r-1}$ determined by the open set:  
$$\bigcap_{{(\alpha,i)\in \Pi(S(f))}}\left(
(R/{\mathfrak m})^{r-1}\setminus{\mathbb V}(\overline{g}_{\alpha, i})\right).$$
\end{proof}

\begin{corollary}\label{conexion_con_pendiente_ref}
Let $(y_1,\ldots,y_r)$ be part of a regular system of parameters of $R$ that determine the directrix of $f$,
and let  $u_1,\ldots,u_{n-r}$ be such that $(u,y)$ is a regular system of parameters in $R$. Suppose that the residue field of $R$ is infinite and let
$L'$ be a $y'$-linear cut. Then $L'$ is a $y'$-generic linear cut if and only if 
$$\nubgen(y')=\delta(\Delta(f,u,y)).$$
\end{corollary}

\begin{proof}
By Corollary \ref{pendiente_cociente} we have that 
$\nubgen(y')\geq\delta(\Delta(f,u,y))$ and by
Theorem \ref{TeoremaVerticesGenericos}  the equality follows.
\end{proof}

The following Corollary is closely related to \cite[Theorem A]{CossartJannsenSchober2019} in the hypersurface case and for the linear form $\Lambda=\lambda_1+\cdots+\lambda_d$.

\begin{corollary} \label{IndepHirPoly} Suppose $R/{\mathfrak m}$ is infinite. 
Let $(y_1,\ldots,y_r)$ be part of a regular system of parameters in $R$ 
  that  determine the directrix of $f$.
 Let $u_1,\ldots,u_d\in R$ and let $v_1,\ldots,v_d\in R$ such that both, 
$(u,y)$ and $(v,y)$ are regular systems of parameters in $R$.
Then
$$\delta(\Delta(f,u,y))=\delta(\Delta(f,v,y)).$$
\end{corollary}

\begin{proof}
By Theorem \ref{TeoremaVerticesGenericos} there are non-empty Zariski open sets 
$U$ and $V$ in $(R/\mathfrak{m})^{r-1}$ such that
$$\Delta(f^{(y,L)},u^{(y,L)},y^{(y,L)}_1)=\Delta(f,u,y),
\qquad \forall \bar{a}\in U,$$
$$\Delta(f^{(y,L)},v^{(y,L)},y^{(y,L)}_1)=\Delta(f,v,y),
\qquad \forall \bar{a}\in V.$$
Where $L$ is the $y$-linear-cut defined by $a=(a_2,\ldots,a_r)$ as in Remark \ref{RemSimpleCut}.
Since $U\cap V$ is non-empty and by Corollary \ref{pendiente_cociente}, 
there is a $y$-linear cut $L$ such that
$$\nublin(y',L')=\delta\left(\Delta(f^{(y,L)},u^{(y,L)},y^{(y,L)}_1)\right)$$
and the result follows.
\end{proof}

As indicated previously, in the following lines we  define an invariant of the singularitiy of $R'$ that will be obtained using $y'$-generic linear cuts. In order to be able to consider {\em a wide range} of $y'$-generic linear cuts when $R'/{\mathfrak m}'$ is finite,  we will enlarge the  residue field in the {\em less harmful possible way}.   
\begin{definition}
\label{def_refined_Samuel_slope}
If the residue field of $R'$ is infinite then,
the \emph{refined Samuel slope}
of the local ring $R'$ is defined as 
$$\mathcal{R}\SSl(R'):=\sup\left\lbrace
 \nubgen(y') :
y'=(y'_1,\ldots,y'_r)\subset R' 
 \text{ determining the directrix}
\right\rbrace.$$
If the residue field of $R'$ is finite then,
let $R'_1=R'[t]_{\mathfrak{m}'[t]}$, and let
$\mathfrak{m}'_1$ be the maximal ideal of $R'_1$.
Then the \emph{refined Samuel slope}
of the local ring $R'$ is defined as the refined Samuel slope of $R_1$, i.e., 
$$\mathcal{R}\SSl(R'):=\mathcal{R}\SSl(R'_1).$$
\end{definition}

Note that, from the definition, $\mathcal{R}\SSl(R')$ depends only on the local ring $R'$. By  Lemma \ref{Ssl_InfinitoK} below   the  refined Samuel slope of the local ring $R'$ can always be defined as the refined Samuel slope of $R_1'$.
\medskip

\begin{lemma} \label{Ssl_InfinitoK}
Let  $R'=R/\langle f \rangle$, and let $R'_1=R'[t]_{\mathfrak{m}'[t]}$.
Then $\mathcal{R}\SSl(R')=\mathcal{R}\SSl(R'_1)$.
\end{lemma}

\begin{proof}
By Definition \ref{def_refined_Samuel_slope} it is enough to prove the equality if 
the residue field of $R'$ is infinite. Let  $k=R'/\mathfrak{m}'$ be the residue field of $R'$, and denote by $\mathfrak{m}_1$ the maximal ideal of $R_1=R[t]_{\mathfrak{m}[t]}$.
Note that $\Gr_{\mathfrak{m}'_1}(R'_1)=k(t)\otimes_k\Gr_{\mathfrak{m}'}(R')$, 
and the directrix of $R'_1$ is obtained by scalar extension of the directrix of $R'$.

From the definition,  $\mathcal{R}\SSl(R')\leq\mathcal{R}\SSl(R'_1)$.
To see the equality, suppose  $y'_1(t),\ldots,y'_r(t)\in R'_1$   determine the directrix of $R'_1$.
Choose representatives $y_1(t),\ldots,y_r(t)\in R_1$ of $y'_1(t),\ldots,y'_r(t)$, and 
let $u_1(t),\ldots,u_{n-r}(t)\in R_1$  be such that 
$u_1(t),\ldots,u_{n-r}(t),y_1(t),\ldots,y_r(t)$ form  a regular system of parameters of $R_1$.
Consider the CP-expansion of $f$ w.r.t. $(u(t),y(t))$:
$$f=\sum_{(\alpha,\beta)\in S_1(f)}c_{\alpha,\beta}(t)u(t)^{\alpha}y(t)^{\beta}.$$
	
There is an infinite set $W\subset k$ such that for all
$t_0\in R$ with $\bar{t}_0\in W$, the evaluation
$$f=\sum_{(\alpha,\beta)\in S_1(f)}c_{\alpha,\beta}(t_0)u(t_0)^{\alpha}y(t_0)^{\beta}$$
is the CP-expansion of $f$ w.r.t. $(u(t_0),y(t_0))$ in $R$.
In particular, we have that
$\Delta(f,u(t),y(t))=\Delta(f,u(t_0),y(t_0))$.
	
By Theorem \ref{TeoremaVerticesGenericos} there exists a $y(t)$-generic linear cut
$L(t)$ for $R'_1$ such that
$$\Delta(f^{(y(t),L(t))},u(t)^{(y(t),L(t))},y_1(t)^{(y(t),L(t))})=
\Delta(f,u(t),y(t)).$$
Shrinking $W$ if necessary,
we have that $L(t_0)$ is a $y'(t_0)$-generic linear cut and
$$\Delta(f^{(y(t_0),L(t_0))},u(t_0)^{(y(t_0),L(t_0))},y_1(t_0)^{(y(t_0),L(t_0))})=
\Delta(f,u(t_0),y(t_0)) \qquad \forall t_0\in R',\ \bar{t}_0\in W.$$
Now, by Theorem \ref{Casor1} and (\ref{desigualdad_pendiente_delta_corte}) 
$$\nublin(y(t),L(t))=\nublin(y(t_0),L(t_0))
\qquad \forall t_0\in R',\ \bar{t}_0\in W,$$
and the result follows.
\end{proof}

As in the case of the Samuel slope $\SSl$, the refined Samuel slope can also  be computed in the ${\mathfrak m}'$-adic completion
of the local ring $R'$.
As a byproduct we will see that the refined Samuel slope gives the $\delta$ invariant of Hironaka's polyhedron.
\begin{theorem} \label{RSSl_completion}
Let $R$ be a local regular ring, let $R'=R/\langle f\rangle$ and let $\widehat{R'}$   be the ${\mathfrak m}'$-adic completion of $R'$. Then
\begin{enumerate}
\item[(i)] $\mathcal{R}\SSl(R')=\mathcal{R}\SSl(\widehat{R'})$.
\item[(ii)] Let $u_1,\ldots,u_{n-r}\in R$  be a collection of elements so  that there are $z_1,\ldots,z_r\in R$ determining the directrix of $f$  and  $(u,z)$ is a regular system of parameters of $R$.
Then
$$\mathcal{R}\SSl(R')=\mathcal{R}\SSl(\widehat{R'})=\delta(\Delta(f,u)).$$
\item[(iii)] Moreover, if $\delta(\Delta(f,u))<\infty$ 
 and the residue field of $R'$ is infinite,
then there are elements  $y_1,\ldots,y_r\in R$ determining the directrix
of $f$ such that
$$\nubgen(y)=\mathcal{R}\SSl(R')=\delta(\Delta(f,u)).$$
\end{enumerate} 
\end{theorem}

\begin{proof} By Theorem \ref{ExtenPolHir} and Lemma \ref{Ssl_InfinitoK}  we can assume that  the residue  field of $R'$ is infinite. By definition,  the inequality    
$\mathcal{R}\SSl(R')\leq\mathcal{R}\SSl(\widehat{R'})$ is straightforward.

By Theorem \ref{ThHirPoly} there exist $\hat{y}_1,\ldots,\hat{y_r}\in\widehat{R}$ determining the directrix of $R'$
such that $\Delta(f,u)=\Delta(f,u,\hat{y})$, and note that
$$\nubgen(\hat{y}')=\delta(\Delta(f,u,\hat{y}))=\delta(\Delta(f,u)).$$
We also have that $\nubgen(\hat{y}')=\mathcal{R}\SSl(\widehat{R'})$. Now, consider the CP-expansion of $f$ w.r.t. $(u,\hat{y})$:
$$f=\sum_{(\alpha,\beta)\in S_{\hat{y}}(f)}c_{\alpha,\beta}(\hat{y})u^{\alpha}\hat{y}^{\beta}.$$

Since $\hat{y}\in\widehat{R}$, then there exists a sequence $\{y_t\}_{t=1}^{\infty}$,
$y_t=(y_{t,1},\ldots,y_{t,r})\subset R$ such that 
$y_{t,j}$ converges to $\hat{y}_j$ for $j=1,\ldots,r$.
For $t\gg 0$, $(u,y_t)$ is a regular system of parameters in $R$ and
we can also consider the $CP$-expansion of $f$ w.r.t. $(u,y_t)$:
$$f=\sum_{(\alpha,\beta)\in S_{y_t}(f)}c_{\alpha,\beta}(y_t)u^{\alpha}y_t^{\beta}.$$
\medskip

Now, assume first  that $\delta(\Delta(f,u))<\infty$. Then, there is some $(\alpha_0,\beta_0)\in S_{\hat{y}}(f)$
such that $|\beta_0|<m$ and
\begin{equation} \label{EqHaty}
\nubgen(\hat{y}')=\delta(\Delta(f,u,\hat{y}))=\frac{|\alpha_0|}{m-|\beta_0|}.
\end{equation}

Let $N>m!\cdot\nubgen(\hat{y}')+m$ and choose some  non-negative integer $t_0$ such that
if $t\geq t_0$ then $y_{t,j}-\hat{y}_j\in\mathfrak{m}^{N+1}$ for $j=1,\ldots,r$.
By Lemma \ref{LemaCPaprox}, 
$$S_{\hat{y}}(f)\cap\mathcal{M}_N=
S_{y_t}(f)\cap\mathcal{M}_N, \quad \forall t\geq t_0.$$

Note that $N>|\alpha_0|+|\beta_0|$. For $t\geq t_0$ we have that 
$(\alpha_0,\beta_0)\in S_{y_t}(f)\cap\mathcal{M}_N=S_{\hat{y}}(f)\cap\mathcal{M}_N$,
and then
$$\nubgen(y_t)\leq \frac{|\alpha_0|}{m-|\beta_0|}=\nubgen(\hat{y}').$$
Let $(\alpha,\beta)\in S_{y_t}(f)$ with $|\beta|<m$:
\begin{itemize}
\item If $|\alpha|+|\beta|\leq N$ then
$(\alpha,\beta)\in S_{\hat{y}}(f)\cap\mathcal{M}_N$ and by (\ref{EqHaty})
$$\frac{|\alpha_0|}{m-|\beta_0|}\leq \frac{|\alpha|}{m-|\beta|}.$$
\item If $|\alpha|+|\beta|> N$ then
$$\frac{|\alpha|}{m-|\beta|}>\frac{N-|\beta|}{m-|\beta|}\geq
\frac{m!\nubgen(\hat{y})+m-|\beta|}{m-|\beta|}>\nubgen(\hat{y}).$$
\end{itemize}
Thus: 
$$\nubgen(y_t)=\delta(\Delta(f,u,y_t))=\frac{|\alpha_0|}{m-|\beta_0|}=\delta(\Delta(f,u,\hat{y}))=\nubgen(\hat{y}'),$$
and we conclude that
$$\nubgen(y_t)=\mathcal{R}\SSl(R')=\mathcal{R}\SSl(\widehat{R'})=
\nubgen(\hat{y}')\qquad \forall t\geq t_0.$$

Next, suppose that   $\delta(\Delta(f,u))=\infty$. Then,  
$S_{\hat{y}}(f)\cap\{(\alpha,\beta)\mid |\beta|<m\}=\emptyset$. Let $C$ be any positive natural number, and choose  $N>Cm+m$.
As we argued above, there is  some $t_0$ such that if $t\geq t_0$ then 
$$S_{\hat{y}}(f)\cap\mathcal{M}_N=
S_{y_t}(f)\cap\mathcal{M}_N,$$
and hence
$$S_{y_t}(f)\cap\mathcal{M}_N\cap\{(\alpha,\beta)\mid |\beta|<m\} =\emptyset.$$
As a consequence, $\nubgen(y'_t)=\delta(\Delta(f,u,y_t))\geq \frac{N-m}{m}>C$ for $t\geq t_0$.
Thus,   $\lim_{t\to\infty}\nubgen(y'_t)=\infty$ and
$\mathcal{R}\SSl(R')=\infty$.
\end{proof}

\begin{remark}
From Corollary \ref{delta_dir_dim_one} and Theorem \ref{RSSl_completion}
it follows that if the invariant $\delta(\Delta(f,u))$ is finite, then the rational 
number can be achieved with a polyhedron without passing to the completion.
The whole polyhedron, can be realized, without passing to the completion, when
$r=1$, see \cite{CossartPiltant2015}, or for $r\geq 1$ with some additional
hypotheses, see \cite{CossartSchober2021} Theorem \ref{ThCossScho} in this paper.
\end{remark}

\section*{Statements and Declarations}

The authors declare no conflicts of interest.

\end{document}